\documentclass[a4paper,11pt]{article}

\usepackage{amsmath,amsthm}
\usepackage{amssymb}
\usepackage{breqn}
\usepackage{enumerate}
\usepackage{graphicx}
\usepackage{bm}
\usepackage{bbm}
\usepackage{fancyhdr}
\usepackage[affil-it]{authblk}
\usepackage{tabu}
\usepackage{bold-extra}
\usepackage[hang,flushmargin]{footmisc}
\usepackage[hypertex]{hyperref}
\usepackage{mathtools}
\usepackage{relsize}
\usepackage{xcolor}

\newtheorem{theorem}{Theorem}
\newtheorem*{theorem*}{Theorem}
\newtheorem{corollary}[theorem]{Corollary}
\newtheorem{lemma}[theorem]{Lemma}
\newtheorem{proposition}[theorem]{Proposition}

\newtheorem{conjecture}[theorem]{Conjecture}
\newtheorem{claim}[theorem]{Claim}

\theoremstyle{definition}
\newtheorem{defin}[theorem]{Definition}

\newcounter{boldSectionCounter}
\stepcounter{boldSectionCounter}
\newcounter{boldSubsectionCounter}
\stepcounter{boldSubsectionCounter}

\allowdisplaybreaks

\newcommand{\ls}[2]{{}^{#1}\,#2}

\newcommand{\tss}[1]{\textsuperscript{#1}}

\newcommand{\inj}{\xhookrightarrow{}}

\newcommand\apps[1]{
	\overset{#1}{\approx}
}
\newcommand{\on}[1]{
	\operatorname{#1}
}

\newcommand{\bias}{\operatorname{bias}}

\newcommand{\boldSubsection}[1]{
	 \addtocounter{boldSectionCounter}{-1}
   \noindent {\bfseries{\scshape \arabic{boldSectionCounter}.\arabic{boldSubsectionCounter}. #1}}\\[6pt]
   \stepcounter{boldSubsectionCounter}
   \addtocounter{boldSectionCounter}{1}
}		

\newcommand{\boldSection}[1]{
   \large\begin{center}\noindent {\bfseries{\scshape \S \arabic{boldSectionCounter} #1}}\\[6pt]\end{center}\normalsize
   \stepcounter{boldSectionCounter}
	 \setcounter{boldSubsectionCounter}{1}
}	

\newcommand\mder{{\Delta\!\!\!\!\!\hbox{\raisebox{0.3ex}{\tiny\ \textbullet}}}\,}

\newcommand{\tdt}{\times\cdots\times}

\newcommand{\cmm}[1]{\ignorespaces}

\newcommand{\tightoverset}[2]{
  \mathop{#2}\limits^{\vbox to -.5ex{\kern-1.15ex\hbox{$#1$}\vss}}}
\newcommand{\conv}{
	\tightoverset{\boldsymbol{-}}{\ast}
}

\newcommand{\prank}{\operatorname{prank}}

\newcommand\restr[2]{{
  \left.\kern-\nulldelimiterspace 
  #1 
  \vphantom{\big|} 
  \right|_{#2} 
}}

\makeatletter
\newcommand{\subalign}[1]{%
  \vcenter{%
    \Let@ \restore@math@cr \default@tag
    \baselineskip\fontdimen10 \scriptfont\tw@
    \advance\baselineskip\fontdimen12 \scriptfont\tw@
    \lineskip\thr@@\fontdimen8 \scriptfont\thr@@
    \lineskiplimit\lineskip
    \ialign{\hfil$\m@th\scriptstyle##$&$\m@th\scriptstyle{}##$\hfil\crcr
      #1\crcr
    }%
  }%
}
\makeatother

\newcommand\blfootnote[1]{%
  \begingroup
  \renewcommand\thefootnote{}\footnote{#1}%
  \addtocounter{footnote}{-1}%
  \endgroup
}

\newcommand\ssk[1]{
	\substack{#1}
}

\newcommand\ex{\mathop{\mathbb{E}}}

\newcommand{\exx}{
  \mathop{
    \mathchoice{\vcenter{\hbox{\larger[4]$\mathbb{E}$}}}
               {\kern0pt\mathbb{E}}
               {\kern0pt\mathbb{E}}
               {\kern0pt\mathbb{E}}
  }\displaylimits
}

\makeatletter
\newcommand*\bcdot{\mathpalette\bigcdot@{0.5}\,\,}
\newcommand*\bigcdot@[2]{\mathbin{\vcenter{\hbox{\scalebox{#2}{$\m@th#1\bullet$}}}}}
\makeatother

\makeatletter
\def\blfootnote{\gdef\@thefnmark{}\@footnotetext}
\makeatother	

\newcommand{\mls}{\on{Spec}^{\on{ml}}}
\newcommand\fco{\lbrack}
\newcommand\fcc{\rbrack^\wedge}

\begin{document}
\thispagestyle{empty}
\begin{center}\Large\noindent{\bfseries{\scshape An Inverse Theorem for Certain Directional Gowers Uniformity Norms}}\\[24pt]\normalsize\noindent{\scshape Luka Mili\'cevi\'c\dag}\\[6pt]
\end{center}
\blfootnote{\noindent\dag\ Mathematical Institute of the Serbian Academy of Sciences and Arts\\\phantom{\dag\ }Email: luka.milicevic@turing.mi.sanu.ac.rs\\\phantom{\dag\ }2010 \emph{Mathematics Subject Classification}: 11B30.}

\footnotesize
\begin{changemargin}{1in}{1in}
\centerline{\sc{\textbf{Abstract}}}
\phantom{a}\hspace{12pt}~Let $G$ be a finite-dimensional vector space over a prime field $\mathbb{F}_p$ with some subspaces $H_1, \dots, H_k$. Let $f \colon G \to \mathbb{C}$ be a function. Generalizing the notion of Gowers uniformity norms, Austin introduced directional Gowers uniformity norms of $f$ over $(H_1, \dots, H_k)$ as
\[\|f\|_{\mathsf{U}(H_1, \dots, H_k)}^{2^k} = \exx_{x \in G,h_1 \in H_1, \dots, h_k \in H_k} \mder_{h_1} \dots \mder_{h_k} f(x)\]
where $\mder_u f(x) \colon= f(x + u) \overline{f(x)}$ is the discrete derivative.\\
\phantom{a}\hspace{12pt}~Suppose that $G$ is a direct sum of subspaces $G = U_1 \oplus U_2 \oplus \dots \oplus U_k$. In this paper we prove the inverse theorem for the norm
\[\|\cdot\|_{\mathsf{U}(U_1, \dots, U_k, \smash[b]{\underbrace{{\scriptstyle G, \dots, G}}_{{\scriptscriptstyle \ell}}})},\]
which is the simplest interesting unknown case of the inverse problem for the directional Gowers uniformity norms. Namely, writing $\|\cdot\|_{\mathsf{U}}$ for the norm above, we show that if $f \colon G \to \mathbb{C}$ is a function bounded by 1 in magnitude and obeying $\|f\|_{\mathsf{U}} \geq c$, provided $\ell < p$, one can find a polynomial $\alpha \colon G \to \mathbb{F}_p$ of degree at most $k + \ell - 1$ and functions $g_i \colon \oplus_{j \in [k] \setminus \{i\}} G_j \to \{z \in \mathbb{C} \colon |z| \leq 1\}$ for $i \in [k]$ such that
\[\Big|\exx_{x \in G} f(x) \omega^{\alpha(x)} \prod_{i \in [k]} g_i(x_1, \dots, x_{i-1}, x_{i+1}, \dots, x_k)\Big| \geq \Big(\exp^{(O_{p,k,\ell}(1))}(O_{p,k,\ell}(c^{-1}))\Big)^{-1}.\]
The proof relies on an approximation theorem for the cuboid-counting function that is proved using the inverse theorem for Freiman multi-homomorphisms.
\end{changemargin}
\vspace{\baselineskip}

\boldSection{Introduction}

In his groundbreaking work~\cite{TimSzem} concerning Szemer\'edi's theorem on arithmetic progressions~\cite{SzemAP}, Gowers introduced the following norms.

\begin{defin}[Gowers uniformity norms]Let $G$ be a finite abelian group and let $f \colon G \to \mathbb{C}$. The \emph{$\mathsf{U}^k$ norm} of $f$ is given by the formula
\[\|f\|_{\mathsf{U}^k}^{2^k} = \exx_{x, a_1, \dots, a_k \in G} \prod_{\varepsilon \in \{0,1\}^k}\operatorname{Conj}^{|\varepsilon|} f\Big(x - \sum_{i = 1}^k \varepsilon_i a_i\Big),\]
where $\operatorname{Conj}^{l}$ stands for the conjugation operator being applied $l$ times and $|\varepsilon|$ is shorthand for $\sum_{i = 1}^k \varepsilon_i$.\end{defin}

These norms measure quasirandomness of a function $f$ in the sense that whenever $f$ has small $\mathsf{U}^k$ norm, it behaves like a randomly chosen function when it comes to counting objects of `complexity' $k-1$. We are deliberately vague about what complexity means, but in the context of arithmetic progressions, where the complexity of an arithmetic progression of length $k$ is $k-2$, this statement can be formalized as follows.

\begin{proposition}[Gowers~\cite{TimSzem}]Let $N$ be a sufficiently large prime, let $A \subset \mathbb{Z}_N$ be a set of size $\delta N$ and suppose that $\|\mathbbm{1}_A - \delta\|_{\mathsf{U}^k} \leq \varepsilon$. Then the number $n_{\text{AP}}$ of arithmetic progressions of length $k+1$ (and hence complexity $k-1$) inside $A$ satisfies $|N^{-2} n_{\text{AP}} - \delta^{k+1}| = O_k(\varepsilon)$.\end{proposition}

Thus, the problem of proving the existence of arithmetic progressions of length $k$ in a dense set $A$ is reduced to describing functions with large $\|\cdot\|_{\mathsf{U}^{k -1}}$. To complete the proof of Szemer\'edi's theorem on arithmetic progressions, Gowers obtained a partial description of such functions.

\begin{theorem}[Gowers~\cite{TimSzem}, Local inverse theorem for uniformity norms] Let $f \colon \mathbb{Z}_N \to \mathbb{D} = \{z \in \mathbb{C} \colon |z| \leq 1\}$ be a function such that $\|f\|_{\mathsf{U}^k} \geq c$. Then there exist a polynomial $\psi : \mathbb{Z}_N \to \mathbb{Z}_N$ of degree at most $k-1$ and an arithmetic progression $P$ of length $N^{\Omega(1)}$ such that $\sum_{x \in P} f(x) \exp\Big(\frac{2 \pi i}{N} \psi(x)\Big) = \Omega_c(|P|)$.\end{theorem}

That result led to many other efforts to reach a better understanding of functions with large Gowers uniformity norms. The overall goal was to replace the local correlation (over the arithmetic progression $P$ in the theorem above) by a global correlation (that is, over the whole group) with an algebraically structured function. There are many results in this direction that are worth noting, but let us first mention the remarkable results of Green, Tao and Ziegler~\cite{StrongUkZ} who proved a global correlation result in the setting of $\mathbb{Z}_N$, while in the case of $\mathbb{F}_p^n$ as the ambient group, Bergelson, Tao and Ziegler obtained analogous result~\cite{BergelsonTaoZiegler} (with a further refinement by Tao and Ziegler~\cite{TaoZiegler}). In both settings, the structured functions that are sufficient are explicitly described. In the case of $\mathbb{Z}_N$ these are the so-called \emph{nilsequences} that are an algebraically structured functions defined on nilmanifolds (we will not go in further details here, as we are primarily interested in the finite vector spaces case.) On the other hand, for $\mathbb{F}_p^n$ the structured functions are a generalization of the usual polynomials that Tao and Ziegler named \emph{non-classical polynomials}. Let us give a full definition here.\\
\indent Let $G$ and $H$ be finite-dimensional vector spaces over a prime field $\mathbb{F}_p$. Given a function $f \colon G\to H$, we write $\Delta_af$ for the function $\Delta_a f(x) = f(x + a) - f(x)$. We say that $f$ is a \emph{non-classical polynomial of degree $\leq d$} if 
\[\Delta_{a_1} \dots \Delta_{a_{d+1}} f(x) = 0\]
holds for all $a_1, \dots, a_{d+1}, x \in G$.\\
\indent In particular, in the so-called `high-characteristic case', which is the case when $k \leq p$, the only non-classical polynomials are the usual polynomials and therefore polynomial phases are again sufficient in the inverse theorem. Let us mention further works by Szegedy~\cite{Szeg}, jointly by Camarena and Szegedy~\cite{CamSzeg} and by Gutman, Manners and Varju~\cite{GMV1},~\cite{GMV2},~\cite{GMV3}.\\
\indent More recently, quantitative bounds were obtained. For the case of cyclic groups $\mathbb{Z}_N$ this was achieved by Manners in~\cite{Manners}, while in the case of finite vector spaces $\mathbb{F}_p^n$ for fixed $p$ and large characteristic, this was done by Gowers and the author in~\cite{genPaper}.\\

Having a reasonably good understanding of the theory of uniformity norms, and recalling that these were used to count arithmetic progressions, it is natural to go one step further and pose the general question of how to adapt this approach to proving the mutlidimensional version of Szemer\'edi's theorem, originally proved by Furstenberg and Katznelson in~\cite{FurstKatz}. Such considerations motivated Austin~\cite{Austin1},~\cite{Austin2} to generalize the notion of Gowers uniformity norms to that of the \emph{directional Gowers uniformity norms}. In this paper, we shall deal only with finite vector spaces, so we give the definition only in that setting, although it can be stated for arbitrary finite abelian group.

\begin{defin}Let $G, H_1, \dots, H_r$ be a finite-dimensional vector spaces over a prime field $\mathbb{F}_p$. Let $f \colon G \to \mathbb{C}$ be a function. We write $\|f\|_{\mathsf{U}(H_1, \dots, H_r)}$ for the non-negative real defined by
\[\|f\|_{\mathsf{U}(H_1, \dots, H_r)}^{2^r} = \exx_{h_1 \in H_1, \dots, h_r \in H_r} \exx_{x \in G} \mder_{h_1} \dots \mder_{h_r} f(x).\]\end{defin}

This indeed defines a norm when $r \geq 2$. This follows from standard and well-known ideas, so we only include a short sketch proof of this fact and leave it to the reader to fill in the details.\\
\indent \textbf{Sketch proof.} Set $H = H_1 + \dots + H_r$ and let $W \leq G$ be a subspace such that $G = H \oplus W$. For each $w \in W$ let $f_w \colon H \to \mathbb{C}$ be the function defined by $f_w(h) = f(w + h)$. Notice that the power of the norm $\|f\|_{\mathsf{U}(H_1, \dots, H_r)}^{2^r}$ is just the average
\[\exx_{w \in W}\|f_w\|_{\mathsf{U}(H_1, \dots, H_r)}^{2^r},\]
from which it follows by H\"older's inequality that we may without loss of generality assume that $G = H_1 + \dots + H_r$.\\ 
\indent Define a generalized inner product of functions $f_I \colon G \to \mathbb{C}$, where $I \subseteq [r]$, by
\[\langle f_I \rangle_{I \subseteq [r]} = \exx_{h_1, h'_1 \in H_1, \dots, h_r, h'_r \in H_r} \prod_{I \subseteq [r]} \on{Conj}^{r-|I|} f\Big(\sum_{i \in I} h_i + \sum_{i \in [r] \setminus I} h'_i\Big).\]
Define auxiliary functions $\tilde{f}_I \colon H_1 \tdt H_k \to \mathbb{C}$ by $\tilde{f}_I(h_1, \dots, h_r) = f_I(h_1 + \dots + h_r)$. Then the inner product above equals the Gowers-Cauchy-Schwarz inner product of the functions $\tilde{f}_I$ for $I \subseteq [r]$. We may use the Gowers-Cauchy-Schwarz inequality (see Lemma~\ref{gcs}) to bound the inner product from above by
\[\prod_{I \subseteq [r]} \|f_I\|_{\square(H_1, \dots, H_r)},\]
which turns out to be equal to 
\[\prod_{I \subseteq [r]}\|f_I\|_{\mathsf{U}(H_1, \dots, H_r)}.\]
Using this bound to the inner product of $2^r$ copies of $f + g$ for the given functions $f, g \colon G \to \mathbb{C}$, the claim follows.\qed

For an example of a directional Gowers uniformity norm, we remark that in order to count squares in a set $A \subseteq G \times G$, that is quadruples of the form $\Big((x,y), (x + a, y), (x, y + a), (x + a, y + a)\Big)$ one needs to understand the directional norm $\|f\|_{\mathsf{U}(H_1, H_2, H_3)}$ for the subgroups $H_1 = G \times \{0\}$, $H_2 = \{0\} \times G$ and $H_3 = \{(x,-x) \colon x \in G\}$.\\

\indent From this point on, we shall refer to the Gowers uniformity norms $\|\cdot\|_{\mathsf{U}^k}$ as the \emph{classical uniformity norms}, and to the directional Gowers uniformity norms simply as the \emph{directional uniformity norms}. There is another notable subfamily of directional uniformity norms, namely the (arithmetic\footnote{We stress that this is the definition in the arithmetic setting, as the box norms can be defined more generally for functions on products of sets, without additional algebraic structure.}) box norms defined for functions $f \colon H_1 \tdt H_k \to \mathbb{D}$ by
\[\|f\|_{\square{H_1, \dots, H_k}}^{2^k} = \exx_{h_1,a_1 \in H_1, \dots, h_k, a_k \in H_k} \mder_{(a_1, 0, \dots, 0)} \dots \mder_{(0,0, \dots, a_k)} f(h_1, \dots, h_k).\]
The inverse theorems are currently available only for the classical norms and for the box norms, the latter inverse theorem being trivial.\\
\indent Before stating our results in this paper, let us formulate the inverse conjecture for the directional uniformity norms in the case of finite vector spaces. It is partially motivated by Austin's work~\cite{Austin1},~\cite{Austin2} in which he described the functions with directional uniformity norms equal to 1 (i.e. the solutions to the extremal case of the inverse problem).

\begin{conjecture}[Inverse conjecture for the directional uniformity norms]\label{inverseconj}Let $p$ be a fixed prime. Suppose that $G$ is a finite-dimensional vector space over the field $\mathbb{F}_p$ with subgroups $H_1, \dots, H_r$. We write $\Sigma = \Sigma(H_1, \dots, H_r)$ for all subgroups of $G$ that can be obtained as sums of the form $H_{i_1}  \dots + H_{i_s}$ for some indices $i_1,\dots, i_s \in [r]$ and $s \geq 1$. For each $K \in \Sigma$ we write $d(K)$ for the number of $H_i$ that are contained in $K$.\\
\indent Suppose that $f \colon G \to \mathbb{D}$ is a function such that $\|f\|_{\mathsf{U}(H_1, \dots, H_r)} \geq c$. Then there exists functions $u_K \colon G \to \mathbb{F}_p$ for $K \in \Sigma$ such that $u_K$ is a non-classical polynomial of degree at most $d(K) - 1$ on every coset of the group $K$ in $G$ and
\[\Big|\exx_{x \in G} f(x) \prod_{K \in \Sigma} u_K(x)\Big| \geq \Omega_{p, c, r}(1).\]
\end{conjecture}

\vspace{\baselineskip}
Again, in the case of high characteristic, namely $r \geq p$, we get (classical) polynomials instead of the non-classical ones.\\
\indent In the setting of cyclic groups we believe that non-classical polynomials of degree $d(K) - 1$ can be replaced by appropriate nilsequences -- for the subgroup $K$, the structured functions should be the nilsequences appearing in the inverse theorem for $\mathsf{U}^{d(K)}$ norm on each coset of $K$. However, it is possible that one would like to use more general sets than just subgroups for the sets of directions in that setting, for example, generalized arithmetic progressions of bounded dimension.\\

\noindent\textbf{Results.} Our ambition in this paper is more modest and we prove an inverse theorem for the norms that could be seen as a combination of the box norms and the classical uniformity norms. More precisely, our main result is the following.

\begin{theorem}\label{inverseUmixed}Let $G_1, \dots, G_k$ be finite-dimensional vector spaces over $\mathbb{F}_p$. Let $G^{\oplus} = G_1 \oplus G_2 \oplus \cdots \oplus G_k$. We view each $G_i$ as a subspace of $G^\oplus$ and misuse the notation by writing $G_i$ instead of $\{0\} \oplus \cdots \oplus G_i \oplus \cdots \oplus \{0\}$ (where $G_i$ appears at $i$\textsuperscript{th} place). Let $r$ be a positive integer and suppose that $f \colon G^\oplus \to \mathbb{D}$ is a function such that
\[\|f\|_{\mathsf{U}\big(G_1, G_2, \dots, G_k, \underbrace{{\scriptstyle G^{\oplus}, \dots, G^{\oplus}}}_r\big)} \,\geq\, c.\]
Assume that $p \geq r$. Then we may find a polynomial $P$ on $G^\oplus$ of degree at most $k + r - 1$ and functions $g_i \colon G_{[k] \setminus \{i\}} \to \mathbb{D}$ for $i \in [k]$ such that
\[\exx_{x_{[k]} \in G^\oplus} f(x_{[k]}) \omega^{P(x_{[k]})} \Big(\prod_{i \in [k]} g_i(x_{[k] \setminus \{i\}}) \Big) \geq \Big(\exp^{(O_{k,r}(1))}(O_{k,r, p}(c^{-1}))\Big)^{-1}.\]
\end{theorem}

Note that the bound in the theorem is quantitative.\\

We prove this theorem by proving an approximation result for the function that counts cuboids. To state this result, we need a definition. Let $(f_{I})_{I \subseteq [k]}$ be a collection of $2^k$ functions $f_I \colon G_1 \tdt G_k \to \mathbb{D}$ indexed by subsets $I \subseteq [k]$. We define the \emph{cubical convolution} of functions $(f_{I})_{I \subseteq [k]}$ to be the function $\square f_{\bcdot} \colon G_{[k]} \to \mathbb{D}$ given by
\[\square f_{\bcdot}(a_1, \dots, a_k) = \exx_{x_1 \in G_1, \dots, x_k \in G_k} \prod_{I \subseteq [k]} \on{Conj}^{k - |I|} f_I\Big((x_i + a_i) \colon i \in I, x_i \colon i \in [k] \setminus I\Big).\]
Observe that in the case when $f$ is the indicator function of a set $A \subset G_1 \tdt G_k$, the value $|G_1| \cdots |G_k| \square f_{\bcdot}(a_1, \dots, a_k)$ is precisely the number of cuboids parallel to principal directions with side-lengths $a_1, \dots, a_k$ with all $2^k$ points lying in the set $A$, which is why we termed the function $\square f_{\bcdot}$ the cubical convolution. The approximation result we mentioned can be stated as follows.

\begin{theorem}\label{multConv}Let $f_I \colon G_1 \tdt G_k \to \mathbb{D}$ be a function for each subset $I \subseteq [k]$. Let $\varepsilon > 0$. Then, there are a positive integer $m \leq \exp^{(O_{k}(1))}\Big(O_{k,p}(\varepsilon^{-1})\Big)$, a multiaffine map $\alpha \colon G_1 \tdt G_k \to \mathbb{F}_p^m$ and a function $c \colon \mathbb{F}_p^m \to \mathbb{D}$ such that
\[\Big\|\square f_\bcdot - c \circ \alpha\Big\|_{L^2} \leq \varepsilon.\]
\end{theorem}
\vspace{\baselineskip}

We may think of this theorem as the direct generalization of the classical fact that the convolution of two functions of a single variable can be approximated in $L^2$ norm by a linear combination of linear phases. The proof of Theorem~\ref{multConv} depends crucially on the inverse theorem for Freiman multi-homomorphism, which was the main result of~\cite{genPaper}.\\

Once Theorem~\ref{multConv} is proved, we proceed to prove Theorem~\ref{inverseUmixed}. We prove the case $r = 1$ separately, and then use it to prove the general $r > 1$ case, similarly to that way the inverse theorem for the classical $\mathsf{U}^2$ norm is used in the proof of the inverse theorem for classical uniformity norms. We then use a symmetry argument, based on important ideas of Green and Tao~\cite{StrongU3}, to finish the proof. However, given a more complicated structure of directions in the current paper than that in the case of classical uniformity norms, the symmetry argument is significantly subtler than that in~\cite{genPaper}.
\vspace{\baselineskip}

\noindent\textbf{Large multilinear spectrum.} There is another notion in this paper that generalizes a classical one-dimensional counterpart, namely the \emph{large multilinear spectrum} of a function $f \colon G_1 \tdt G_k \to \mathbb{D}$ which we now define.

\begin{defin}\label{mlspecdefin}Let $f \colon G_1 \tdt G_k \to \mathbb{D}$ be a function and let $\varepsilon > 0$. We define \emph{$\varepsilon$-large multilinear spectrum} of $f$ to be the set
\[\mls_{\varepsilon}(f) = \Big\{\mu \in \on{ML}(G_{[k]} \to \mathbb{F}_p) \colon \|f \omega^{\mu}\|_{\square^k} \geq \varepsilon\Big\},\]
where $ \on{ML}(G_{[k]} \to \mathbb{F}_p) $ stands for the set of all multilinear forms on $G_1 \tdt G_k$ and where $\|\cdot\|_{\square^k}$ stands for the box norm with respect to sets $G_1, \dots, G_k$.\end{defin}

It generalizes the usual large spectrum of a function of a single variable reasonably directly. The analogies between the usual spectrum and the multilinear spectrum are explained more thoroughly later in the paper (see Definition~\ref{mlspecdefin2} and the discussion that follows it), but let us point out one of them here. As we have already remarked, Theorem~\ref{multConv} is a generalization of the fact that the convolution of two functions $f, g \colon G \to \mathbb{D}$ can be approximated by linear combinations of linear phases. Actually, these linear phases come from the large spectrum of $f$ (and $g$). It turns out that the analogous phenomenon occurs in Theorem~\ref{multConv}. Namely, the mutliaffine forms $\alpha_i$, $i \in [m]$, appearing in that theorem, can be taken to lie the large mutlilinear spectrum of the functions $f_I$ (see Proposition~\ref{cubconvmlsapprox}).\\    

We do not use the multilinear spectrum directly in this paper, but it motivated some of our steps in the proof of Theorem~\ref{inverseUmixed} and it seems to be closely related to some of ideas in this paper. Therefore, we decided to explore the properties of the large multilinear spectrum a bit further, which we do in the final section.\\

\noindent\textbf{Organization of the paper.} The paper is organized as follows. The next preliminary section lists some useful auxiliary results. After that we devote the following two sections to the proofs of Theorems~\ref{multConv} and~\ref{inverseUmixed}, respectively. Finally, in the last section we study the large multilinear spectrum.\\

\noindent\textbf{Acknowledgements.} This work was supported by the Serbian Ministry of Education, Science and Technological Development through Mathematical Institute of the Serbian Academy of Sciences and Arts. I would like to thank Tim Gowers for helpful discussions.\\ 

\boldSection{Preliminaries}

\noindent\textbf{Notation.} As above, we write $\mathbb{D} = \{z \in \mathbb{C}\colon |z| \leq 1\}$ for the unit disk. We use the standard expectation notation $\ex_{x \in X}$ as shorthand for the average $\frac{1}{|X|} \sum_{x \in X}$, and when the set $X$ is clear from the context we simply write $\ex_x$. As in~\cite{LukaRank},~\cite{genPaper}, we use the following convention to save writing in situations where we have many indices appearing in predictable patterns. Instead of denoting a sequence of length $m$ by $(x_1, \dots, x_m)$, we write $x_{[m]}$, and for $I\subset[m]$ we write $x_I$ for the subsequence with indices in $I$. This applies to products as well: $G_{[k]}$ stands for $\prod_{i \in [k]} G_i$ and $G_I = \prod_{i \in I} G_i$. For example, instead of writing $\alpha \colon \prod_{i \in I} G_i \to \mathbb{F}_p$ and $\alpha(x_i \colon i \in I)$, we write $\alpha \colon G_I \to \mathbb{F}_p$ and $\alpha(x_I)$. This notation is particularly useful when $I=[k]\setminus\{d\}$ as it saves us writing expressions such as $(x_1,\dots,x_{d-1},x_{d+1},\dots,x_k)$ and $G_1\tdt G_{d-1}\times G_{d+1}\tdt G_k$.\\
\indent We extend the use of the dot product notation to any situation where we have two sequences $x=x_{[n]}$ and $y=y_{[n]}$ and a meaningful multiplication between elements $x_i y_i$, writing $x\cdot y$ as shorthand for the sum $\sum_{i=1}^n x_i y_i$. For example, if $\lambda=\lambda_{[n]}$ is a sequence of scalars, and $A=A_{[n]}$ is a suitable sequence of maps, then $\lambda\cdot A$ is the map $\sum_{i=1}^n\lambda_iA_i$.\\  
\indent Frequently we shall consider `slices' of sets $S\subset G_{[k]}$, by which we mean sets $S_{x_I} = \{y_{[k]\setminus I} \in G_{[k] \setminus I} \colon (x_I, y_{[k] \setminus I}) \in S\}$, for $I \subset [k], x_I \in G_I$. (Here we are writing $(x_I,y_{[k]\setminus I})$ not for the concatenation of the sequences $x_I$ and $y_{[k]\setminus I}$ but for the `merged' sequence $z_{[n]}$ with $z_i=x_i$ when $i\in I$ and $z_i=y_i$ otherwise.) If $I$ is a singleton $\{i\}$ and $z_i\in G_i$, then we shall write $S_{z_i}$ instead of $S_{z_{\{i\}}}$. Sometimes, the index $i$ will be clear from the context and it will be convenient to omit it. For example, $f(x_{[k]\setminus\{i\}},a)$ stands for $f(x_1,\dots,x_{i-1},a,x_{i+1},\dots,x_k)$. If the index is not clear, we emhasize it by writing it as a superscript to the left of the corresponding variable, e.g.\ $f(x_{[k]\setminus\{i\}},{}^i\,a)$.\\ 
\indent More generally, when $X_1, \dots, X_k$ are finite sets, $Z$ is an arbitrary set, $f \colon X_1 \tdt X_k = X_{[k]} \to Z$ is a function, $I \subsetneq [k]$ and $x_i \in X_i$ for each $i \in I$, we define a function $f_{x_I} \colon X_{[k] \setminus I} \to Z$, by mapping each $y_{[k] \setminus I} \in X_{[k] \setminus I}$ as $f_{x_I}(y_{[k] \setminus I}) = f(x_I, y_{[k] \setminus I})$. When the number of variables is small -- for example, when we have a function $f(x,y)$ that depends only on two variables $x$ and $y$ instead of on indexed variables -- we also write $f_x$ for the map $f_x(y)=f(x,y)$.\\

Let $G, G_1, \dots, G_k$ be finite-dimensional vector spaces over a finite field $\mathbb{F}_p$ and let $\omega = \exp\Big(\frac{2 \pi i}{p}\Big)$. For maps $f,g \colon G \to \mathbb{C}$, we write $f\conv g$ for the function defined by $f \conv g(x) = \ex_{y \in G} f(x + y) \overline{g(y)}$. Fix a dot product $\cdot$ on $G$. The Fourier transform of $f \colon G \to \mathbb{C}$ is the function $\hat{f} \colon G \to \mathbb{C}$ defined by $\hat{f}(r) = \ex_{x \in G} f(x) \omega^{-r\cdot x}$.\\

The $L^q$ norms have their usual meaning: for a function $f \colon X \to \mathbb{D}$ we define $\|f\|_{L^q} = \Big(\exx_{x \in X} |f(x)|^q\Big)^{1/q}$. Frequently, when $f(x)$ is an explicit and complicated expression depending on the variable $x$, we write the dummy variable in the subscript $\|f(x)\|_{L^q, x}$ to stress that $L^q$ norm is calculated by averaging over $x \in X$. Furthermore, for two such expressions $f(x),g(x)$ we write
\[f(x) \apps{\varepsilon}_{L^q, x} g(x)\]
to mean that $\|f - g\|_{L^q} \leq \varepsilon$.\\ 
\vspace{\baselineskip}

\noindent\textbf{Useful lemmas and results.} We recall the definition of the Gowers box norms. Let $X_1, \dots, X_k$ be arbitrary sets. The \emph{Gowers box norm} of a function $f \colon X_1 \tdt X_k \to \mathbb{C}$ (Definition B.1 in the Appendix B of~\cite{GreenTaoPrimes}) is defined by
\[\|f\|_{\square^{k}}^{2^k} = \exx_{x_1, y_1 \in X_1, \dots, x_k, y_k \in X_k} \prod_{I \subset [k]} \operatorname{Conj}^{|I|} f(x_I, y_{[k] \setminus I}).\]
The following lemma is the Gowers-Cauchy-Schwarz inequality for the box norm.

\begin{lemma}\label{gcs}Let $f_I \colon X_1 \tdt X_k \to \mathbb{C}$ be a function for each $I \subset [k]$. Then
\[\Big|\exx_{x_1, y_1 \in X_1, \dots, x_k, y_k \in X_k} \prod_{I \subset [k]} \operatorname{Conj}^{|I|} f_I(x_I, y_{[k] \setminus I})\Big| \leq \prod_{I \subset [k]} \|f_I\|_{\square^k}.\]
\end{lemma}

A particularly useful fact is the following corollary.

\begin{corollary}\label{gcsCor}Let $f_I \colon X_I \to \mathbb{C}$ be a function for each $I \subset [k]$. Then
\[\Big|\exx_{x_1 \in X_1, \dots, x_k\in X_k} \prod_{I \subset [k]} f_I(x_I)\Big| \leq \|f_{[k]}\|_{\square^k}.\]
\end{corollary}

The following lemma is a technical result that allows us to replace values in the unit disk $\mathbb{D}$ by values of modulus exactly 1.

\begin{lemma}\label{unitLemma} Let $X$ be a finite set. Suppose that $f,g \colon X \to \mathbb{D}$ are two functions such that $|\ex_{x \in X} f(x) \overline{g(x)}| \geq c$. Then, there is another function $\tilde{g} \colon G \to \mathbb{D}$ such that $|g(x)| = 1$ for all $x \in X$ and $|\ex_{x \in X} f(x) \overline{\tilde{g}(x)}| \geq c$ as well.\end{lemma}

\begin{proof}Define $\tilde{g} \colon X \to \mathbb{D}$ by choosing each $\tilde{g}(x)$ independently according to the following probability distribution. If $v = g(x) \not= 0$, we set $\tilde{g}(x) = v/|v|$ with probability $\frac{1+|v|}{2}$, and $\tilde{g}(x) = -v/|v|$ with probability $\frac{1-|v|}{2}$. If $g(x) = 0$, then we simply set $\tilde{g}(x) = 1$ and $\tilde{g}(x) = -1$ with probabilities $\frac{1}{2}$ each.\\
\indent These distributions were chosen so that for each $x$ we have $\ex [f(x)\overline{\tilde{g}}(x)] = f(x)\overline{g(x)}$. Taking the expectation we obtain
\[\exx\Big[\sum_{x \in X} f(x) \overline{\tilde{g}(x)}\Big] = \sum_{x \in X} f(x) \overline{g(x)}.\]
By triangle inequality it follows that
\[\exx\Big|\sum_{x \in X} f(x) \overline{\tilde{g}(x)}\Big| \geq \Big|\sum_{x \in X} f(x) \overline{g(x)}\Big|.\]
In particular, there is a choice of $\tilde{g}$ such that
\[\Big|\frac{1}{|X|}\sum_{x \in X} f(x) \overline{\tilde{g}(x)}\Big| \geq c,\]
as desired.\end{proof}

Next we record a simple consequence of Hoeffding's inequality that allows us to approximate averages of long sequences by averages of short subsequences.

\begin{lemma}[Random sampling approximation]\label{randomsampling}Let $a_1, \dots, a_n \in \mathbb{D}$ and let $k \in [n]$, $\varepsilon > 0$. Pick indices $i_1, \dots, i_k \in [n]$ uniformly and independently at random. Then
\[\mathbb{P}\Big(\Big|\frac{1}{n} \sum_{j \in [n]} a_j - \frac{1}{k} \sum_{j \in [k]} a_{i_j}\Big| \leq \varepsilon\Big) \geq 1 - 4\exp\Big(-\frac{\varepsilon^2 k}{8}\Big).\]
\end{lemma}

\begin{proof}Write $\alpha = \frac{1}{n} \sum_{j \in [n]} a_j$. For $j \in [k]$, let $X_j$ be the random variable given by $\on{Re} a_{i_j}$ and let $Y_j$ be the random variable given by $\on{Im} a_{i_j}$. Then, random variables $X_1, \dots, X_k$ are independent and take values in $[-1,1]$. Likewise, $Y_1, \dots, Y_k$ are independent and take values in $[-1,1]$. Notice that $\ex X_j = \on{Re} \alpha$ and $\ex Y_j = \on{Im} \alpha$. Applying Hoeffding's inequality we obtain
\[\mathbb{P}\Big(\Big|\frac{1}{k} \sum_{j \in [k]} X_j - \on{Re} \alpha\Big| \geq \varepsilon/2\Big) \leq 2\exp\Big(-\frac{\varepsilon^2 k}{8}\Big)\]
and
\[\mathbb{P}\Big(\Big|\frac{1}{k} \sum_{j \in [k]} Y_j - \on{Im} \alpha\Big| \geq \varepsilon/2\Big) \leq 2\exp\Big(-\frac{\varepsilon^2 k}{8}\Big).\]
The lemma follows after combining these two bounds.\end{proof}

We need some standard elementary Fourier-analytic facts.

\begin{lemma}\label{largeSpec}Let $f \colon G \to \mathbb{D}$ and let $\varepsilon > 0$. Write $S = \{r \in G \colon |\hat{f}(r)|\, \geq \varepsilon\}$. Then $|S| \leq \varepsilon^{-2}$.\end{lemma}

\begin{proof}This follows from 
\[\varepsilon^2 |S| \leq \sum_{r \in S} |\hat{f}(r)|^2 \leq \sum_r |\hat{f}(r)|^2 = \exx_x |f(x)|^2 \leq 1.\qedhere\]\end{proof}

\begin{lemma}\label{l2approxLemma}Suppose that $f, g \colon G \to \mathbb{D}$ are two functions. Then
\[\Big\|f\conv g(x) - \sum_{r \in S} \hat{f}(r) \overline{\hat{g}(r)} \omega^{r \cdot x}\Big\|_{L^2, x} \leq 2\varepsilon,\]
where $S$ is a set such that $\{r \in G \colon |\hat{f}(r)|, |\hat{g}(r)| \,\geq \varepsilon\} \subseteq S$.
\end{lemma}

\begin{proof}This is a consequence of simple algebraic manipulation. Namely, expanding out gives
\begin{align*}\Big\|f\conv g(x) - \sum_{r \in S} \hat{f}(r) \overline{\hat{g}(r)} \omega^{r \cdot x}\Big\|_{L^2, x}^2 =& \exx_x \Big| \sum_{r \notin S} \hat{f}(r) \overline{\hat{g}(r)} \omega^{r \cdot x} \Big|^2
=\exx_x \sum_{r, s \notin S} \hat{f}(r) \overline{\hat{g}(r)} \overline{\hat{f}(s)}\hat{g}(s) \omega^{(r - s) \cdot x}\\
=& \sum_{r, s \notin S} \hat{f}(r) \overline{\hat{g}(r)} \overline{\hat{f}(s)}\hat{g}(s) \mathbbm{1}(r = s)
= \sum_{r \notin S} |\hat{f}(r)|^2|\hat{g}(r)|^2\\
\leq & \varepsilon^2 \Big(\sum_{r} |\hat{g}(r)|^2\Big) + \varepsilon^2 \Big(\sum_{r} |\hat{f}(r)|^2\Big)\\
\leq&  2\varepsilon^2,\end{align*}
from which the lemma follows.\end{proof}

\begin{lemma}[Gowers~\cite{TimSzem}]\label{convL4lemma} Let $f,g \colon G \to \mathbb{C}$ be two functions. Then
\[\bigg(\exx_{d}\Big|\exx_{x}f(x + d) \overline{g(x)}\Big|^2 \bigg)^2\leq \sum_r |\hat{f}(r)|^4.\]
\end{lemma}

\begin{proof}After manipulating the expression a bit and using the Cauchy-Schwarz inequality, we have
\begin{align*}\exx_{d}\Big|\exx_{x}f(x + d) \overline{g(x)}\Big|^2 = &\exx_d |f \conv g(d)|^2 = \sum_r \Big|\fco f\conv g \fcc(r)\Big|^2 = \sum_r |\hat{f}(r)|^2 |\hat{f}(r)|^2\\
 \leq &\sqrt{\sum_r |\hat{f}(r)|^4} \sqrt{\sum_r |\hat{g}(r)|^4} \leq \sqrt{\sum_r |\hat{f}(r)|^4} \sqrt{\sum_r |\hat{g}(r)|^2} \leq \sqrt{\sum_r |\hat{f}(r)|^4}.\qedhere\end{align*}
\end{proof}

Next, we need some facts about multilinear forms. Let $\alpha \colon G_{[k]} \to \mathbb{F}_p$ be a multilinear form. The quantity $\ex_{x_{[k]}} \omega^{\alpha(x_{[k]})}$ is called the \emph{bias} of $\alpha$, written $\bias \alpha$, and it measures the uniformity of the distribution of values of the form $\alpha$. The quantity $-\log_p \bias \alpha$ was introducted as the \emph{analytic rank} of $\alpha$ by Gowers and Wolf in~\cite{TimWolf}.\\

It turns out that the analytic rank is sub-additive, as proved by Lovett~\cite{Lovett}.

\begin{lemma}[Lovett, Reformulation of Theorem 1.5 in~\cite{Lovett}] \label{subBias}Let $\alpha, \beta\colon G_{[k]} \to \mathbb{F}_p$ be two multilinear forms. Then
\[\bias (\alpha + \beta) \geq \bias \alpha \cdot \bias \beta.\]\end{lemma}

An interesting corollary of the result above is the correlation result for multilinear varieties.

\begin{corollary}[Lovett, Claim 1.6 in~\cite{Lovett}] \label{coincPoscor} Let $U, V \subset G_{[k]}$ be two multilinear varieties. Then
\[|G_{[k]}|\,|U \cap V| \geq |U|\,|V|.\]
\end{corollary}

When $\alpha \colon G_1 \tdt G_k \to \mathbb{F}_p$ is a multiaffine form, it we may still use the definition of the bias above. Write $\alpha(x_{[k]}) = \sum_{I \subseteq [k]} \alpha_I(x_I)$ for some multilinear forms $\alpha_I \colon G_I \to \mathbb{F}_p$. We call the multilinear form $\alpha_{[k]}$ the multilinear part of $\alpha$. It turns out that the bias of $\alpha$ can be related to the bias of $\alpha_{[k]}$. 

\begin{lemma}[Lovett, Lemma 2.1 in~\cite{Lovett}] \label{mlbias}Let $\alpha \colon G_{[k]} \to \mathbb{F}_p$ be a multiaffine form with multilinear part $\alpha^{\on{ml}}$. Then
\[\Big|\exx_{x_{[k]}} \omega^{\alpha(x_{[k]})} \Big| \leq \bias \alpha^{\on{ml}}.\]
\end{lemma}

We also need the following two results from~\cite{LukaRank}.

\begin{lemma}[Approximating dense varieties externally, Lemma 12 in~\cite{LukaRank}]\label{varOuterApprox} Let $A \colon G_{[k]} \to H$ be a multiaffine map. Then for every positive integer $s$ there is a multiaffine map $\phi \colon G_{[k]} \to \mathbb{F}_p^s$ such that $A^{-1}(0) \subset \phi^{-1}(0)$ and $|\phi^{-1}(0) \setminus A^{-1}(0)| \leq p^{-s}|G_{[k]}|$. \end{lemma}

\begin{theorem}\label{invbias}Let $\alpha \colon G_{[k]} \to \mathbb{F}_p$ be a multilinear form such that $\bias \alpha \geq c$. Then there are a positive integer $m \leq O\Big((\log_p c^{-1})^{O(1)} \Big)$, subsets $\emptyset \not= I_i \subsetneq [k]$ and multilinear forms $\beta_i \colon G_{I_i} \to \mathbb{F}_p$ and $\gamma_i \colon G_{[k] \setminus I_i} \to \mathbb{F}_p$ for $i \in [m]$ such that
\[\alpha(x_{[k]}) = \sum_{i \in [m]} \beta_i(x_{I_i}) \gamma_i(x_{[k] \setminus I_i}).\]\end{theorem}

\noindent \textbf{Remark.} The least number $m$ such that $\alpha$ can be expressed in terms of $m$ pairs of forms $(\beta_i, \gamma_i)$ as above is called \emph{the partition rank} of $\alpha$, and is denoted $\prank \alpha$. This notion was introduced by Naslund in~\cite{Naslund}. Thus, high bias, or equivalently low analytic rank, implies low partition rank. In a qualitative sense, this theorem was first proved by Bhowmick and Lovett in~\cite{BhowLov}, generalizing an approach of Green and Tao~\cite{GreenTaoPolys}. An almost identical result (there is a slight difference in bounds) to the one stated here was obtained independently by Janzer in~\cite{Janzer2} (who had previously obtained tower-type bounds in this problem~\cite{Janzer1}). There is also a very recent related result of Cohen and Moshkovitz~\cite{CohenMoshkovitz}, but their work has an additional requirement that the prime is sufficiently large depending on the dimensions of spaces $G_i$.\\

The following result follows straightforwardly from Freiman's theorem~\cite{Freiman},~\cite{greenRuzsaFreiman},~\cite{Sanders} and Balog-Szemer\'edi-Gowers~\cite{BalogSzemeredi},~\cite{TimSzem} theorem. 

\begin{theorem}\label{FreimanAddQuads}Let $G, H$ be finite-dimensional vector space over $\mathbb{F}_p$. Let $A \subset G$ be a set of density $\delta$. Suppose that $\phi \colon A \to H$ is a map which \emph{respects all additive quadruples} in the sense that whenever $a,b,c,d \in A$ satisfy $a-b+c-d = 0$ then one has $\phi(a) - \phi(b) + \phi(c) - \phi(d) = 0$. Then there is a global affine map $\Phi \colon G \to H$ such that $\phi(x) = \Phi(x)$ holds for at least $\exp(-\log^{O(1)} \delta^{-1}) |G|$ of $x \in A$.\end{theorem} 

A generalization of that theorem was proved by Gowers and the author in~\cite{genPaper}.

\begin{theorem}\label{multihom}Let $G, H$ be finite-dimensional vector space over $\mathbb{F}_p$. Suppose that $A \subset G_{[k]}$ is a set of density $\delta > 0$ and let $\phi \colon A \to H$ be a Freiman multihomomorphism. Then there is a global multiaffine map $\Phi \colon G_{[k]} \to H$ which coincides with $\phi$ on at least $\bigg(\exp^{(O_{k,p}(1))}\Big(O_{k,p}(\delta^{-1})\Big)\bigg)^{-1} |G_{[k]}|$ of points in $A$.\end{theorem}

We included Theorem~\ref{FreimanAddQuads} even though it is as special case of Theorem~\ref{multihom} since in that case very good bounds are available thanks to Sander's proof of the Bogolyubov-Ruzsa lemma~\cite{Sanders}.

\vspace{\baselineskip}

\boldSection{Cubical convolutions}

Let $(f_{I})_{I \subseteq [k]}$ be a collection of $2^k$ functions  $f_I \colon G_{[k]} \to \mathbb{D}$ indexed by subsets $I \subseteq [k]$. Recall from the introductory section that we write $\square f_{\bcdot} (a_{[k]})$ for the value
\[\exx_{x_{[k]}} \prod_{I \subseteq [k]} \on{Conj}^{k - |I|} f_I((x + a)_{I}, (x)_{[k] \setminus I}),\]
defining the cubical convolution of functions $(f_{I})_{I \subseteq [k]}$. The main result of this section is Theorem~\ref{multConv}, whose statement is repeated below, says that cubical convolutions are approximately constant on layers of a multiaffine map to a low-dimensional space.

\begin{theorem*}[Theorem~\ref{multConv}]Let $f_I \colon G_{[k]} \to \mathbb{D}$ be a function for each subset $I \subseteq [k]$. Let $\varepsilon > 0$. Then, there are a positive integer $m \leq \exp^{(O_{k}(1))}\Big(O_{k,p}(\varepsilon^{-1})\Big)$, a multiaffine map $\alpha \colon G_{[k]} \to \mathbb{F}_p^m$ and a function $c \colon \mathbb{F}_p^m \to \mathbb{D}$ such that
\[\square f_\bcdot\,(a_{[k]}) \apps{\varepsilon}_{L^2, a_{[k]}} c(\alpha(a_{[k]})).\]
\end{theorem*}

The main step in the proof of the above theorem is the following proposition, which allows us to express $\square f_\bcdot$ in terms of itself. The gain is that the terms involving $\square f_\bcdot$ appearing in the  approximation sum involve a dummy variable in the argument, and thus are simpler than the starting function.

\begin{proposition}\label{multConvMainStep}Let $f_I \colon G_{[k]} \to \mathbb{D}$ be a function for each subset $I \subseteq [k]$. Let $\varepsilon > 0$. Then, there are a positive integer $m \leq \exp^{(O_{k}(1))}\Big(O_{k,p}(\varepsilon^{-1})\Big)$, constants $c_1, \dots, c_m \in \mathbb{D}$ and multiaffine forms $\alpha_1, \dots, \alpha_m, \beta_1, \dots, \beta_m \colon G_{[k]} \to \mathbb{F}_p$ such that
\[\square f_\bcdot\,(a_{[k]}) \apps{\varepsilon}_{L^2, a_{[k]}} \sum_{i \in [m]}\, c_i\omega^{\alpha_i(a_{[k]})}\, \exx_{d_k} \square f_{\bcdot}\,(a_{[k-1]}, d_k) \omega^{\beta_i(a_{[k-1]}, d_k)}.\]
\end{proposition}

Before proceeding with the proof of the proposition, we use it to prove Theorem~\ref{multConv}.

\begin{proof}[Proof of Theorem~\ref{multConv}] By induction on $\ell \in [k]$, we show that for a parameter $\eta > 0$ there are a positive integer $m \leq \exp^{(O_{k}(1))}\Big(O_{k,p}(\eta^{-1})\Big)$, constants $c_1, \dots, c_m \in \mathbb{D}$ and multiaffine forms $\alpha^{(1)}_1, \dots,$ $\alpha^{(1)}_m, \dots,$ $\alpha^{(\ell + 1)}_1, \dots,$ $\alpha^{(\ell + 1)}_m \colon G_{[k]} \to \mathbb{F}_p$ such that

\begin{equation}\label{cubeapproxeqn}\square f_\bcdot\,(a_{[k]}) \apps{\eta}_{L^2, a_{[k]}} \sum_{i \in [m]}\, \exx_{d_{[k - \ell + 1, k]}} c_i \omega^{\alpha^{(1)}_i(a_{[k]}) + \alpha^{(2)}_i(a_{[k-1]}, d_k) + \dots + \alpha^{(\ell + 1)}_i(a_{[k-\ell]}, d_{[k- \ell + 1, k]})} \, \square f_{\bcdot}\,(a_{[k-\ell]}, d_{[k- \ell + 1, k]}).\end{equation}

The base case $\ell = 1$ is exactly Proposition~\ref{multConvMainStep}. Suppose now the claim holds for some $\ell \in [k-1]$. Apply the inductive hypothesis for approximation parameter $\eta / 2$ to get a positive integer $m \leq \exp^{(O_{k}(1))}\Big(O_{k,p}(\eta^{-1})\Big)$, constants $c_1, \dots,$ $c_m \in \mathbb{D}$ and multiaffine forms $\alpha^{(1)}_1, \dots,$ $\alpha^{(1)}_m, \dots,$ $\alpha^{(\ell + 1)}_1, \dots,$ $\alpha^{(\ell + 1)}_m \colon G_{[k]} \to \mathbb{F}_p$ such that~\eqref{cubeapproxeqn} holds (with $\eta/2$ instead of $\eta$), that is
\begin{equation}\label{cubeapproxeqnIH1}\square f_\bcdot\,(a_{[k]}) \apps{\eta/2}_{L^2, a_{[k]}} \sum_{i \in [m]}\, \exx_{d_{[k - \ell + 1, k]}} c_i \omega^{\alpha^{(1)}_i(a_{[k]}) + \alpha^{(2)}_i(a_{[k-1]}, d_k) + \dots + \alpha^{(\ell + 1)}_i(a_{[k-\ell]}, d_{[k- \ell + 1, k]})} \, \square f_{\bcdot}\,(a_{[k-\ell]}, d_{[k- \ell + 1, k]}).\end{equation}
Apply Proposition~\ref{multConvMainStep} to $(f_I)_{I \subseteq [k]}$ but this time with a much smaller approximation parameter $\frac{1}{2m}\eta$ to get a positive integer $q \leq \exp^{(O_{k}(1))}\Big(O_{k,p}(m \eta^{-1})\Big)$, constants $c'_1, \dots, c'_q \in \mathbb{D}$ and multiaffine forms $\beta_1, \dots,$ $\beta_q,$ $\gamma_1, \dots,$ $\gamma_q \colon G_{[k]} \to \mathbb{F}_p$ such that
\begin{equation}\label{cubeapproxeqnIH2}\square f_\bcdot\,(a_{[k]}) \apps{\eta / 2m}_{L^2, a_{[k]}} \sum_{i \in [q]}\, c'_i \omega^{\beta_i(a_{[k]})}\, \exx_{d_{k - \ell}} \square f_{\bcdot}\,(a_{[k] \setminus \{k-\ell\}}, d_{k - \ell}) \omega^{\gamma_i(a_{[k] \setminus \{k - \ell\}}, d_{k - \ell})}.\end{equation}
Use approximation~\eqref{cubeapproxeqnIH2} instead of $f_{\bcdot}\,(a_{[k-\ell]}, d_{[k- \ell + 1, k]})$ terms on the right-hand-side of~\eqref{cubeapproxeqnIH1}. By the triangle inequality for $L^2$ norms we have
\begin{align*}\square f_\bcdot\,(a_{[k]}) \apps{\eta}_{L^2, a_{[k]}}& \sum_{i \in [m]}\, c_i \exx_{d_{[k - \ell + 1, k]}} \omega^{\alpha^{(1)}_i(a_{[k]}) + \alpha^{(2)}_i(a_{[k-1]}, d_k) + \dots + \alpha^{(\ell + 1)}_i(a_{[k-\ell]}, d_{[k- \ell + 1, k]})}\\
&\hspace{1cm}\sum_{j \in [q]}\, c'_j \omega^{\beta_j(a_{[k - \ell]}, d_{[k - \ell + 1, k]})}\, \exx_{d_{k - \ell}} \square f_{\bcdot}\,(a_{[k - \ell - 1]}, d_{[k - \ell, k]}) \omega^{\gamma_j(a_{[k - \ell - 1]}, d_{[k - \ell, k]})}\\
= & \sum_{\substack{i \in [m]\\j \in [q]}}\, \exx_{d_{[k - \ell, k]}} c_ic'_j \omega^{\alpha^{(1)}_i(a_{[k]}) + \alpha^{(2)}_i(a_{[k-1]}, d_k) + \dots + \alpha^{(\ell)}_i(a_{[k-\ell + 1]}, d_{[k- \ell + 2, k]})}\\
&\hspace{2cm}\omega^{\big(\alpha^{(\ell + 1)}_i(a_{[k-\ell]}, d_{[k- \ell + 1, k]}) + \beta_j(a_{[k - \ell]}, d_{[k - \ell + 1, k]})\big)} \square f_{\bcdot}\,(a_{[k - \ell - 1]}, d_{[k - \ell, k]}),\end{align*}
as claimed.\\

Using approximation~\eqref{cubeapproxeqn} for $\ell = k$ and approximation parameter $\varepsilon / 2$, we get a positive integer $m \leq \exp^{(O_{k}(1))}\Big(O_{k,p}(\epsilon^{-1})\Big)$, constants $c_1, \dots,$ $c_m \in \mathbb{D}$ and multiaffine forms $\alpha^{(1)}_1, \dots,$ $\alpha^{(1)}_m, \dots,$ $\alpha^{(k + 1)}_1, \dots,$ $\alpha^{(k + 1)}_m \colon G_{[k]} \to \mathbb{F}_p$ such that
\begin{equation}\square f_\bcdot\,(a_{[k]}) \apps{\varepsilon/2}_{L^2, a_{[k]}} \sum_{i \in [m]}\, \exx_{d_{[k]}} c_i\omega^{\alpha^{(1)}_i(a_{[k]}) + \alpha^{(2)}_i(a_{[k-1]}, d_k) + \dots + \alpha^{(k+1)}_i(d_{[k]})} \, \square f_{\bcdot}\,(d_{[k]}).\label{cubicalmideqn}\end{equation}
To finish the proof, we use random sampling to find a finite collection of $d_{[k]}$ so that the approximation above is still accurate with only a finite number of terms. Let $r$ be a parameter to be chosen later. Pick points $\tilde{d}^{(i)}_{[k]} \in G_{[k]}$ for $i = 1,\dots, r$ uniformly and independently at random. By Lemma~\ref{randomsampling}, for $a_{[k]} \in G_{[k]}$ and $i \in [m]$, with probability at least $1 - 4\exp\Big(-\frac{r \varepsilon^2}{128m^2}\Big)$ we have that
\begin{align*}&\Big|\exx_{d_{[k]}} \omega^{\alpha^{(1)}_i(a_{[k]}) + \alpha^{(2)}_i(a_{[k-1]}, d_k) + \dots + \alpha^{(k+1)}_i(d_{[k]})} \, \square f_{\bcdot}\,(d_{[k]})\\
&\hspace{6cm}- \frac{1}{r} \sum_{\ell \in [r]} \omega^{\alpha^{(1)}_i(a_{[k]}) + \alpha^{(2)}_i(a_{[k-1]}, \tilde{d}^{(\ell)}_k) + \dots + \alpha^{(k+1)}_i(\tilde{d}^{(\ell)}_{[k]})} \, \square f_{\bcdot}\,(\tilde{d}^{(\ell)}_{[k]})\Big| \leq \varepsilon / 4m.\end{align*}
By the union bound, for each $a_{[k]}$ we have
\begin{align}&\Big|\sum_{i \in [m]} \exx_{d_{[k]}} c_i\omega^{\alpha^{(1)}_i(a_{[k]}) + \alpha^{(2)}_i(a_{[k-1]}, d_k) + \dots + \alpha^{(k+1)}_i(d_{[k]})} \, \square f_{\bcdot}\,(d_{[k]})\nonumber\\
&\hspace{6cm}- \sum_{i \in [m]}\sum_{\ell \in [r]} \frac{1}{r}c_i  \square f_{\bcdot}\,(\tilde{d}^{(\ell)}_{[k]})\,\omega^{\alpha^{(1)}_i(a_{[k]}) + \alpha^{(2)}_i(a_{[k-1]}, \tilde{d}^{(\ell)}_k) + \dots + \alpha^{(k+1)}_i(\tilde{d}^{(\ell)}_{[k]})}\Big| \leq \varepsilon / 4\label{finiteapproxcubical}\end{align}
with probability at least $1 - 4m\exp\Big(-\frac{r \varepsilon^2}{128m^2}\Big)$. We conclude that there is a choice of $\tilde{d}^{(1)}_{[k]}, \dots,$ $\tilde{d}^{(r)}_{[k]} \in G_{[k]}$ such that~\eqref{finiteapproxcubical} holds for at least $1 - 4m\exp\Big(-\frac{r \varepsilon^2}{128m^2}\Big)|G_{[k]}|$ of $a_{[k]} \in G_{[k]}$. Let $A \subset G_{[k]}$ be the set of such $a_{[k]}$. Note on the other hand that if $a_{[k]} \notin A$ then trivially
\begin{align}&\Big|\sum_{i \in [m]} \exx_{d_{[k]}} c_i \omega^{\alpha^{(1)}_i(a_{[k]}) + \alpha^{(2)}_i(a_{[k-1]}, d_k) + \dots + \alpha^{(k+1)}_i(d_{[k]})} \, \square f_{\bcdot}\,(d_{[k]})\nonumber\\
&\hspace{6cm}- \sum_{i \in [m]}\sum_{\ell \in [r]} \frac{1}{r} c_i \square f_{\bcdot}\,(\tilde{d}^{(\ell)}_{[k]})\,\omega^{\alpha^{(1)}_i(a_{[k]}) + \alpha^{(2)}_i(a_{[k-1]}, \tilde{d}^{(\ell)}_k) + \dots + \alpha^{(k+1)}_i(\tilde{d}^{(\ell)}_{[k]})}\Big| \leq 2m.\label{finiteapproxcubical2}\end{align}
Returning to approximation~\eqref{cubicalmideqn}, we may use inequalities~\eqref{finiteapproxcubical} and~\eqref{finiteapproxcubical2} to obtain
\begin{align}\square f_\bcdot\,(a_{[k]}) \apps{\varepsilon'}_{L^2, a_{[k]}} \sum_{i \in [m]}\sum_{\ell \in [r]} \frac{1}{r}c_i  \square f_{\bcdot}\,(\tilde{d}^{(\ell)}_{[k]})\,\omega^{\alpha^{(1)}_i(a_{[k]}) + \alpha^{(2)}_i(a_{[k-1]}, \tilde{d}^{(\ell)}_k) + \dots + \alpha^{(k+1)}_i(\tilde{d}^{(\ell)}_{[k]})}\label{cubicalFinalApprox}\end{align}
where
\[\varepsilon' = \frac{\varepsilon}{2} + \sqrt{\frac{\varepsilon^2}{16} + 4m\exp\Big(-\frac{r \varepsilon^2}{128m^2}\Big) \cdot 4m^2} \leq \frac{3\varepsilon}{4} + 4m^2\exp\Big(-\frac{r \varepsilon^2}{256m^2}\Big).\]
We may pick $r = O(m^{O(1)} \varepsilon^{-O(1)})$ so that $\varepsilon' \leq \varepsilon$, which completes the proof. The total number of summands in~\eqref{cubicalFinalApprox} is $rm \leq O(m^{O(1)} \varepsilon^{-O(1)}) \leq \exp^{(O_{k}(1))}\Big(O_{k,p}(\epsilon^{-1})\Big)$, as required.\end{proof}

We now prove Proposition~\ref{multConvMainStep}. The method of the proof is similar to that of the proofs of approximations results for mixed convolutions in~\cite{U4paper} and~\cite{genPaper}.

\begin{proof}[Proof of Proposition~\ref{multConvMainStep}]Fix $a_{[k-1]} \in G_{[k-1]}$ and consider the sliced function $(\square f_{\bcdot}\,)_{a_{[k-1]}} \colon a_k \mapsto \square f_{\bcdot}\,(a_{[k]})$. Note that $(\square f_{\bcdot}\,)_{a_{[k-1]}}$ is given by an average of (single-variable) convolutions
\begin{align}(\square f_{\bcdot}\,)_{a_{[k-1]}}(b_k) &= \exx_{x_{[k-1]}} \exx_{y_k} \Big(\prod_{I \subseteq [k-1]} \on{Conj}^{k - 1 - |I|} (f_{I \cup \{k\}})_{(x + a)_I; x_{[k-1] \setminus I}}(y_k + b_k)\Big)\nonumber\\
&\hspace{8cm}\overline{\Big(\prod_{I \subseteq [k-1]} \on{Conj}^{k - 1 - |I|} (f_I)_{(x + a)_I; x_{[k-1] \setminus I}}(y_k)\Big)} \nonumber\\
&= \exx_{x_{[k-1]}} \Big(\prod_{I \subseteq [k-1]} \on{Conj}^{k - 1 - |I|} (f_{I \cup \{k\}})_{(x + a)_I; x_{[k-1] \setminus I}}\Big) \conv \Big(\prod_{I \subseteq [k-1]} \on{Conj}^{k - 1 - |I|} (f_I)_{(x + a)_I; x_{[k-1] \setminus I}}\Big)(b_k).\label{multiConvMainStepFirstEqn}\end{align}

Let $\rho > 0$ be a parameter to be specified later. For each $x_{[k-1]}, a_{[k-1]} \in G_{[k-1]}$, let $S_{x_{[k-1]}, a_{[k-1]}}$ be the set of all $r \in G_k$ such that
\begin{equation}\label{Sdefin}\bigg|\Big[\prod_{I \subseteq [k-1]} \on{Conj}^{k - 1 - |I|} (f_{I \cup \{k\}})_{(x + a)_I; x_{[k-1] \setminus I}}\Big]^\wedge(r)\Big|,\,\,\Big|\overline{\Big[\prod_{I \subseteq [k-1]}\on{Conj}^{k - 1 - |I|} (f_I)_{(x + a)_I; x_{[k-1] \setminus I}}\Big]^\wedge(r)}\bigg| \geq \rho,\end{equation}
i.e. the large Fourier coefficients. 

We now show that for each $a_{[k-1]}$ the only $r$ that matter in the approximation above are those such that $r \in S_{x_{[k-1]}, a_{[k-1]}}$ for many $x_{[k-1]}$. 

Let $\xi > 0$ a parameter to be chosen later. For $a_{[k-1]} \in G_{[k-1]}$ define $R_{a_{[k-1]}} \subset G_k$ to be the set of all $r \in G_k$ such that for at least $\xi |G_{[k-1]}|$ of $x_{[k-1]} \in G_{[k-1]}$ we have $r \in S_{x_{[k-1]}, a_{[k-1]}}$. 
%

\begin{claim}\label{frequentlargefcsclaim}Let $R \subset G_k$ be an arbitrary subset and let $a_{[k-1]} \in G_{[k-1]}$. Then
\[\Big\|\sum_{r \in R} \exx_{d_k} \omega^{r \cdot (a_k - d_k)} \square f_{\bcdot}\,(a_{[k-1]}, d_k)\Big\|_{L^2, a_k} \leq 1.\]
Moreover, if $R$ is disjoint from $R_{a_{[k-1]}}$ then we have a stronger bound
\[\Big\|\sum_{r \in R} \exx_{d_k} \omega^{r \cdot (a_k - d_k)} \square f_{\bcdot}\,(a_{[k-1]}, d_k)\Big\|_{L^2, a_k} \leq \rho^2 + \xi\rho^{-2}.\]
\end{claim}

\begin{proof} Expanding the expression we get
\begin{align}&\Big\|\sum_{r \in R} \exx_{d_k} \omega^{r \cdot (a_k - d_k)} \square f_{\bcdot}\,(a_{[k-1]}, d_k)\Big\|_{L^2, a_k}^2 = \exx_{a_k} \Big|\sum_{r \in R} \exx_{d_k} \omega^{r \cdot (a_k - d_k)} \square f_{\bcdot}\,(a_{[k-1]}, d_k)\Big|^2\nonumber\\
&\hspace{2cm}=\sum_{r,s \in R} \exx_{d_k, e_k, a_k} \omega^{r \cdot (a_k - d_k) - s \cdot (a_k - e_k)} \square f_{\bcdot}\,(a_{[k-1]}, d_k) \overline{\square f_{\bcdot}\,(a_{[k-1]}, e_k)}\nonumber\\
&\hspace{2cm}= \sum_{r,s \in R} \exx_{d_k, e_k} \Big(\exx_{a_k}  \omega^{(r -s) \cdot a_k}\Big)\omega^{s \cdot e_k - r \cdot d_k} \square f_{\bcdot}\,(a_{[k-1]}, d_k) \overline{\square f_{\bcdot}\,(a_{[k-1]}, e_k)}\nonumber\\
&\hspace{2cm}= \sum_{r \in R} \exx_{d_k, e_k} \omega^{r \cdot e_k - r \cdot d_k} \square f_{\bcdot}\,(a_{[k-1]}, d_k) \overline{\square f_{\bcdot}\,(a_{[k-1]}, e_k)}\nonumber\\
&\hspace{2cm}= \sum_{r \in R}\Big|\exx_{d_k} \square f_{\bcdot}\,(a_{[k-1]}, d_k)\omega^{- r \cdot d_k}\Big|^2.\label{rapproxspectrumineq}\end{align}
It is now easy to finish the proof of the first claim. (We deliberately stopped the argument here as we shall use~\eqref{rapproxspectrumineq} for the proof of the second part of the claim.) The expression above is at most
\begin{align*}\sum_{r \in G_k}\Big|\exx_{d_k} \square f_{\bcdot}\,(a_{[k-1]}, d_k)\omega^{- r \cdot d_k}\Big|^2 = \sum_{r \in G_k} \exx_{d_k, e_k} \omega^{r \cdot e_k - r \cdot d_k} \square f_{\bcdot}\,(a_{[k-1]}, d_k) \overline{\square f_{\bcdot}\,(a_{[k-1]}, e_k)}= \exx_{d_k}|\square f_{\bcdot}\,(a_{[k-1]}, d_k)|^2 \leq 1.\end{align*}
For the second claim, return to~\eqref{rapproxspectrumineq}. We have
\begin{align*}&\Big\|\sum_{r \in R} \exx_{d_k} \omega^{r \cdot (a_k - d_k)} \square f_{\bcdot}\,(a_{[k-1]}, d_k)\Big\|_{L^2, a_k}^2 = \sum_{r \in R}\Big|\exx_{d_k} \square f_{\bcdot}\,(a_{[k-1]}, d_k)\omega^{- r \cdot d_k}\Big|^2\\
&\hspace{2cm}= \sum_{r \in R}\Big|\exx_{d_k} \exx_{x_{[k-1]}} \exx_{y_k} \Big(\prod_{I \subseteq [k-1]} \on{Conj}^{k - 1 - |I|} (f_{I \cup \{k\}})_{(x + a)_I; x_{[k-1] \setminus I}}(y_k + d_k)\Big)\\
&\hspace{6cm}\overline{\Big(\prod_{I \subseteq [k-1]} \on{Conj}^{k - 1 - |I|} (f_I)_{(x + a)_I; x_{[k-1] \setminus I}}(y_k)\Big)}\omega^{- r \cdot d_k} \Big|^2\\
&\hspace{2cm}= \sum_{r \in R}\Big|\exx_{x_{[k-1]}} \Big[\prod_{I \subseteq [k-1]} \on{Conj}^{k - 1 - |I|} (f_{I \cup \{k\}})_{(x + a)_I; x_{[k-1] \setminus I}}\Big]^{\wedge}(r) \\
&\hspace{6cm}\overline{\Big[\prod_{I \subseteq [k-1]} \on{Conj}^{k - 1 - |I|} (f_I)_{(x + a)_I; x_{[k-1] \setminus I}}\Big]^{\wedge}(r)}\Big|^2\\
&\hspace{2cm}= \sum_{r \in R}\Big|\exx_{x_{[k-1]}} \Big(\mathbbm{1}(r \in S_{a_{[k-1]}, x_{[k-1]}}) + \mathbbm{1}(r \notin S_{a_{[k-1]}, x_{[k-1]}})\Big) \Big[\prod_{I \subseteq [k-1]} \on{Conj}^{k - 1 - |I|} (f_{I \cup \{k\}})_{(x + a)_I; x_{[k-1] \setminus I}}\Big]^{\wedge}(r) \\
&\hspace{6cm}\overline{\Big[\prod_{I \subseteq [k-1]} \on{Conj}^{k - 1 - |I|} (f_I)_{(x + a)_I; x_{[k-1] \setminus I}}\Big]^{\wedge}(r)}\Big|^2\\
&\hspace{2cm}\leq 2\sum_{r \in R}\Big|\exx_{x_{[k-1]}} \mathbbm{1}(r \in S_{a_{[k-1]}, x_{[k-1]}})\Big[\prod_{I \subseteq [k-1]} \on{Conj}^{k - 1 - |I|} (f_{I \cup \{k\}})_{(x + a)_I; x_{[k-1] \setminus I}}\Big]^{\wedge}(r) \\
&\hspace{6cm}\overline{\Big[\prod_{I \subseteq [k-1]} \on{Conj}^{k - 1 - |I|} (f_I)_{(x + a)_I; x_{[k-1] \setminus I}}\Big]^{\wedge}(r)}\Big|^2\\
&\hspace{2cm}+2\sum_{r \in R}\Big|\exx_{x_{[k-1]}} \mathbbm{1}(r \notin S_{a_{[k-1]}, x_{[k-1]}})\Big[\prod_{I \subseteq [k-1]} \on{Conj}^{k - 1 - |I|} (f_{I \cup \{k\}})_{(x + a)_I; x_{[k-1] \setminus I}}\Big]^{\wedge}(r) \\
&\hspace{6cm}\overline{\Big[\prod_{I \subseteq [k-1]} \on{Conj}^{k - 1 - |I|} (f_I)_{(x + a)_I; x_{[k-1] \setminus I}}\Big]^{\wedge}(r)}\Big|^2\\
&\hspace{2cm}\leq 2\sum_{r \in R}\Big|\exx_{x_{[k-1]}} \mathbbm{1}(r \in S_{a_{[k-1]}, x_{[k-1]}})\Big[\prod_{I \subseteq [k-1]} \on{Conj}^{k - 1 - |I|} (f_{I \cup \{k\}})_{(x + a)_I; x_{[k-1] \setminus I}}\Big]^{\wedge}(r) \\
&\hspace{6cm}\overline{\Big[\prod_{I \subseteq [k-1]} \on{Conj}^{k - 1 - |I|} (f_I)_{(x + a)_I; x_{[k-1] \setminus I}}\Big]^{\wedge}(r)}\Big|^2\\
&\hspace{2cm}+2\sum_{r \in R}\exx_{x_{[k-1]}} \mathbbm{1}(r \notin S_{a_{[k-1]}, x_{[k-1]}})\Big| \Big[\prod_{I \subseteq [k-1]} \on{Conj}^{k - 1 - |I|} (f_{I \cup \{k\}})_{(x + a)_I; x_{[k-1] \setminus I}}\Big]^{\wedge}(r) \\
&\hspace{6cm}\overline{\Big[\prod_{I \subseteq [k-1]} \on{Conj}^{k - 1 - |I|} (f_I)_{(x + a)_I; x_{[k-1] \setminus I}}\Big]^{\wedge}(r)}\Big|^2\\
&\hspace{2cm}\leq 2\sum_{r \in R}\Big|\exx_{x_{[k-1]}} \mathbbm{1}(r \in S_{a_{[k-1]}, x_{[k-1]}})\Big[\prod_{I \subseteq [k-1]} \on{Conj}^{k - 1 - |I|} (f_{I \cup \{k\}})_{(x + a)_I; x_{[k-1] \setminus I}}\Big]^{\wedge}(r) \\
&\hspace{6cm}\overline{\Big[\prod_{I \subseteq [k-1]} \on{Conj}^{k - 1 - |I|} (f_I)_{(x + a)_I; x_{[k-1] \setminus I}}\Big]^{\wedge}(r)}\Big|^2 + 2\rho^2.\end{align*}
It remains to bound
\begin{align}&\sum_{r \in R}\Big|\exx_{x_{[k-1]}} \mathbbm{1}(r \in S_{a_{[k-1]}, x_{[k-1]}})\Big[\prod_{I \subseteq [k-1]} \on{Conj}^{k - 1 - |I|} (f_{I \cup \{k\}})_{(x + a)_I; x_{[k-1] \setminus I}}\Big]^{\wedge}(r)\nonumber\\
&\hspace{6cm} \overline{\Big[\prod_{I \subseteq [k-1]} \on{Conj}^{k - 1 - |I|} (f_I)_{(x + a)_I; x_{[k-1] \setminus I}}\Big]^{\wedge}(r)}\Big|^2.\label{rsumfcbound1}\end{align}
To that end, for $r \in G_k$ we set
\begin{align*}v_{r} = \exx_{x_{[k-1]}} \mathbbm{1}\Big(r \in S_{x_{[k-1]}, a_{[k-1]}}\Big) &\Big[\prod_{I \subseteq [k-1]} \on{Conj}^{k - 1 - |I|} (f_{I \cup \{k\}})_{(x + a)_I; x_{[k-1] \setminus I}}\Big]^\wedge(r)\\
&\hspace{1cm}\overline{\Big[\prod_{I \subseteq [k-1]}\on{Conj}^{k - 1 - |I|} (f_I)_{(x + a)_I; x_{[k-1] \setminus I}}\Big]^\wedge(r)}.\end{align*}
Note that $v_{r} \in \mathbb{D}$ holds for all $r$ and that we also have
\begin{align*}\sum_{r \in G_k} |v_{r}| \leq &\sum_{r \in G_k} \exx_{x_{[k-1]}} \mathbbm{1}\Big(r \in S_{x_{[k-1]}, a_{[k-1]}}\Big) \bigg|\Big[\prod_{I \subseteq [k-1]} \on{Conj}^{k - 1 - |I|} (f_{I \cup \{k\}})_{(x + a)_I; x_{[k-1] \setminus I}}\Big]^\wedge(r)\nonumber\\
&\hspace{6cm}\overline{\Big[\prod_{I \subseteq [k-1]}\on{Conj}^{k - 1 - |I|} (f_I)_{(x + a)_I; x_{[k-1] \setminus I}}\Big]^\wedge(r)}\bigg|\nonumber\\
\leq &\exx_{x_{[k-1]}} |S_{x_{[k-1]}, a_{[k-1]}}| \leq \rho^{-2},\end{align*}
where we used Lemma~\ref{largeSpec} in the last step. Using this, we may bound the expression~\eqref{rsumfcbound1} from above by
\begin{align*}\sum_{r \in R} |v_{r}|^2 \leq \sum_{r \in R} |v_{r}|\bigg(\exx_{x_{[k-1]}}\mathbbm{1}\Big(r \in S_{x_{[k-1]}, a_{[k-1]}}\Big)\bigg) \leq \xi \sum_{r \in R} |v_{r}| \leq \xi \rho^{-2},\end{align*}
completing the proof.\end{proof}

\noindent \textbf{Step 2. Multilinear structure in $R_{a_{[k-1]}}$.} Let $\Phi$ be a function defined on a subset of $G_{[k-1]}$ such that $\Phi(a_{[k-1]}) \in S_{x_{[k-1]}, a_{[k-1]}}$ for at least $\xi |G_{[k-1]}|$ of $x_{[k-1]} \in G_{[k-1]}$. We claim that $\Phi$ is in fact a multi-2-homomorphism on a somewhat smaller subset of $G_{[k-1]}$. By a \emph{$d$-additive quadruple} we mean a quadruple $(x^{(1)}_{[k-1]}, \dots, x^{(4)}_{[k-1]})$ of points in $G_{[k-1]}$ such that $x^{(1)}_i = \dots = x^{(4)}_i$ for each $i \not= d$ and $x^{(1)}_d - x^{(2)}_d + x^{(3)}_d - x^{(4)}_d = 0$.\\

Similarly to~\cite{U4paper} and~\cite{genPaper}, we use a Fourier-analytic technique invented by Gowers in~\cite{TimSzem}.

\begin{claim}\label{daddquad}Let $A \subset G_{[k-1]}$ be a set of density $\eta$ and let $\Phi \colon A \to G_k$ be a function such that for each $a_{[k-1]} \in A$ we have $\Phi(a_{[k-1]}) \in S_{x_{[k-1]}, a_{[k-1]}}$ for at least $\xi |G_{[k-1]}|$ of $x_{[k-1]} \in G_{[k-1]}$. Let $d \in [k-1]$ be a direction. Then $\Phi$ respects at least $\xi^2 \eta^2 \rho^4|G_{[k-1]}| |G_d|^2$ $d$-additive quadruples in $A$.\end{claim}

\begin{proof}We prove the claim for $d = k-1$ which is clearly sufficient. Recalling the definining property~\eqref{Sdefin} of $S_{x_{[k-1]}, a_{[k-1]}}$, we have that
\begin{align}\xi \eta \rho^2\leq & \exx_{a_{[k-1]}} A(a_{[k-1]}) \exx_{x_{[k-1]}} \bigg|\Big[\prod_{I \subseteq [k-1]} \on{Conj}^{k - 1 - |I|} (f_{I \cup \{k\}})_{(x + a)_I; x_{[k-1] \setminus I}}\Big]^\wedge\Big(\Phi(a_{[k-1]})\Big)\nonumber\\
&\hspace{6cm} \overline{\Big[\prod_{I \subseteq [k-1]}\on{Conj}^{k - 1 - |I|} (f_I)_{(x + a)_I; x_{[k-1] \setminus I}}\Big]^\wedge\Big(\Phi(a_{[k-1]})\Big)}\bigg|^2\nonumber\\
\leq & \exx_{a_{[k-1]}, x_{[k-1]}}   A(a_{[k-1]}) \bigg|\exx_{y_k, z_k} \Big(\prod_{I \subseteq [k-1]}\on{Conj}^{k - 1 - |I|} f_{I \cup \{k\}} ((x + a)_I, x_{[k-1] \setminus I}, y_k)\Big)\nonumber\\
&\hspace{6cm} \Big(\prod_{I \subseteq [k-1]} \on{Conj}^{k - |I|} f_{I}((x + a)_I, x_{[k-1] \setminus I},z_k)\Big) \omega^{(z_k -y_k) \cdot \Phi(a_{[k-1]})}\bigg|^2\nonumber\\
= & \exx_{a_{[k-1]}, x_{[k-1]}}   A(a_{[k-1]}) \exx_{y_k, y'_k, z_k, z'_k} \Big(\prod_{I \subseteq [k-1]}\on{Conj}^{k - 1 - |I|} f_{I \cup \{k\}} ((x + a)_I, x_{[k-1] \setminus I}, y_k)\Big)\nonumber\\
&\hspace{6cm}\Big(\prod_{I \subseteq [k-1]}\on{Conj}^{k - |I|} f_{I \cup \{k\}} ((x + a)_I, x_{[k-1] \setminus I}, y'_k)\Big)\nonumber\\
&\hspace{6cm} \Big(\prod_{I \subseteq [k-1]} \on{Conj}^{k - |I|} f_{I}((x + a)_I, x_{[k-1] \setminus I},z_k)\Big)\nonumber\\
&\hspace{6cm} \Big(\prod_{I \subseteq [k-1]} \on{Conj}^{k -1- |I|} f_{I}((x + a)_I, x_{[k-1] \setminus I},z'_k)\Big) \omega^{(z_k -y_k - z'_k + y'_k) \cdot \Phi(a_{[k-1]})}.\label{squaredboundaddquads}\end{align}

Write $\bm{p}$ for the shorthand for the sequence of parameters $\Big(a_{[k-2]}, x_{[k-2]}, y_k, y'_k, z_k, z'_k\Big)$. For any fixed value of $\bm{p}$, we define functions $F_{\bm{p}} \colon G_{k-1} \to \mathbb{D}$ and $G_{\bm{p}} \colon G_{k-1} \to \mathbb{D}$ by
\begin{align*}F_{\bm{p}}(w) =& \Big(\prod_{k-1 \in I \subseteq [k-1]}\on{Conj}^{k - 1 - |I|} f_{I \cup \{k\}} ((x + a)_{I \setminus \{k-1\}}, x_{[k-2] \setminus I}, y_k,\ls{k-1}{\,w})\Big)\\
&\hspace{2cm}\Big(\prod_{k-1 \in I \subseteq [k-1]}\on{Conj}^{k - |I|} f_{I \cup \{k\}} ((x + a)_{I \setminus \{k-1\}}, x_{[k-2] \setminus I}, y'_k,\ls{k-1}{\,w})\Big)\\
&\hspace{2cm} \Big(\prod_{k-1 \in I \subseteq [k-1]} \on{Conj}^{k - |I|} f_{I}((x + a)_{I \setminus \{k-1\}}, x_{[k-2] \setminus I},z_k,\ls{k-1}{\,w})\Big)\\
&\hspace{2cm} \Big(\prod_{k-1 \in I \subseteq [k-1]} \on{Conj}^{k -1- |I|} f_{I}((x + a)_{I \setminus \{k-1\}}, x_{[k-2] \setminus I}, z'_k, \ls{k-1}{\,w})\Big)
\end{align*}
and
\[G_{\bm{p}}(w) = A(a_{[k-2]}, {}^{k-1}\, w) \omega^{(z_k -y_k - z'_k + y'_k) \cdot \Phi(a_{[k-2]}, {}^{k-1}\, w)}.\]

Using this notation and applying triangle inequality in~\eqref{squaredboundaddquads}, we get
\[\xi \eta \rho^2\leq \exx_{\bm{p}} \exx_{x_{k-1}} \Big|\exx_{a_{k-1}} F_{\bm{p}}(x_{k-1} + a_{k-1}) G_{\bm{p}}(a_{k-1})\Big|.\]
By Cauchy-Schwarz inequality we get
\[\xi^2 \eta^2 \rho^4\leq \exx_{\bm{p}} \exx_{x_{k-1}} \Big|\exx_{a_{k-1}} F_{\bm{p}}(x_{k-1} + a_{k-1}) G_{\bm{p}}(a_{k-1})\Big|^2\]
which can be bounded from above using Lemma~\ref{convL4lemma} by
\begin{align*}\exx_{\bm{p}} \sum_{r_{k-1}} |\widehat{G_{\bm{p}}}(r_{k-1})|^4 = &\exx_{a_{[k-2]}, y_k, y'_k, z_k, z'_k} \sum_{r_{k-1}} \Big|\exx_{w_{k-1}} A(a_{[k-2]}, w_{k-1}) \omega^{(z_k -y_k - z'_k + y'_k) \cdot \Phi(a_{[k-2]}, w_{k-1}) - r_{k-1} \cdot w_{k-1}}\Big|^4\\
= &\exx_{\substack{a_{[k-2]}\\y_k, y'_k, z_k, z'_k}} \sum_{r_{k-1}} \exx_{\substack{u_{k-1}, v_{k-1}\\w_{k-1}, t_{k-1}}} A(a_{[k-2]}, u_{k-1}) A(a_{[k-2]}, v_{k-1}) A(a_{[k-2]}, w_{k-1}) A(a_{[k-2]}, t_{k-1})\\
&\hspace{2cm} \omega^{(z_k -y_k - z'_k + y'_k) \cdot \big(\Phi(a_{[k-2]}, u_{k-1}) - \Phi(a_{[k-2]}, v_{k-1}) + \Phi(a_{[k-2]}, w_{k-1}) - \Phi(a_{[k-2]}, t_{k-1})\big)}\\
&\hspace{2cm} \omega^{-r_{k-1} \cdot \big(u_{k-1} - v_{k-1} + w_{k-1} - t_{k-1}\big)}\\
= &\exx_{a_{[k-2]}} \exx_{\substack{u_{k-1}, v_{k-1}\\w_{k-1}, t_{k-1}}} \,|G_{k-1}| \hspace{2pt} A(a_{[k-2]}, u_{k-1}) A(a_{[k-2]}, v_{k-1}) A(a_{[k-2]}, w_{k-1}) A(a_{[k-2]}, t_{k-1}) \\
&\hspace{2cm}\mathbbm{1}\Big(\Phi(a_{[k-2]}, u_{k-1}) - \Phi(a_{[k-2]}, v_{k-1}) + \Phi(a_{[k-2]}, w_{k-1}) - \Phi(a_{[k-2]}, t_{k-1}) = 0\Big)\\
&\hspace{2cm}\mathbbm{1}\Big(u_{k-1} - v_{k-1} + w_{k-1} - t_{k-1} = 0\Big),
\end{align*}
which is exactly the density of additive $(k-1)$-quadruples whose points lie in $A$ and are respected by $\Phi$.\end{proof}

Now combine the claim we have just proved with Theorem~\ref{FreimanAddQuads} for each direction in $[k-1]$ to deduce that $\Phi$ is a Freiman multi-homomorphism on a subset $A' \subset A$ of somewhat smaller density. Once we find a part of $\Phi$ which is a Freiman multi-hmomorphism, we apply Theorem~\ref{multihom} to pass to a global multiaffine map.

\begin{claim}\label{cubicalphidensemultihom}Let $A \subset G_{[k-1]}$ be a set of density $\eta$ and let $\Phi \colon A \to G_k$ be a function such that for each $a_{[k-1]} \in A$ we have $\Phi(a_{[k-1]}) \in S_{x_{[k-1]}, a_{[k-1]}}$ for at least $\xi |G_{[k-1]}|$ of $x_{[k-1]} \in G_{[k-1]}$. Then $\Phi$ coincides with a global multiaffine map on a set of size $\bigg(\exp^{(O_{k}(1))}\Big(O_{k,p}(\eta^{-1} \xi^{-1} \rho^{-1})\Big)\bigg)^{-1} |G_{[k-1]}|$.\end{claim}

\begin{proof}We first show inductively that for each $i \in [0,k-1]$ there is a subset $A_i \subseteq A$ of size at least
\[|A_i| \geq \Omega_{k,p}\Big(\exp(-\log^{O(i)}(\xi^{-1}\eta^{-1}\rho^{-1}))\Big) |G_{k-1}|\]
on which $\Phi$ is a Freiman homomorphism in directions $1,\dots, i$. The base case $i = 0$ is trivial as we may simply take $A_0 = A$. Suppose now that the claim holds for some $0 \leq i \leq k-2$ and let $A_i$ be the set given by the induction hypothesis. Let $\eta_i$ be the density of $A_i \in G_{[k-1]}$. Claim~\ref{daddquad} applies to the set $A_i$ in direction $i + 1$ and we deduce that $\Phi$ respects at least $\xi^2 \eta_i^2 \rho^4|G_{[k-1]}| |G_{i+1}|^2$ of $(i+1)$-additive quadruples with points in $A_i$. By averaging, we get a set $X \subset G_{[k-1] \setminus \{i+1\}}$ of density at least $\frac{1}{2}\xi^2 \eta_i^2 \rho^4$ such that for each $x_{[k-1] \setminus \{i+1\}} \in X$, the map $\Phi$ respects at least $\frac{1}{2}\xi^2 \eta_i^2 \rho^4 |G_{i+1}|^3$ $(i+1)$-additive quadruples with vertices in $A_i \cap (\{x_{[k-1] \setminus \{i+1\}}\} \times G_{i+1})$. For each $x_{[k-1] \setminus \{i+1\}} \in X$ apply Theorem~\ref{FreimanAddQuads} to deduce that there is a subset $Y_{x_{[k-1] \setminus \{i+1\}}} \subseteq (A_i)_{x_{[k-1] \setminus \{i+1\}}}$ of size
\[|Y_{x_{[k-1] \setminus \{i+1\}}}| \geq \Omega_{k,p}\Big(\exp(-\log^{O(1)}(\xi^{-1}\rho^{-1} \eta_i^{-1}))\Big) |G_{i+1}|\]
on which the sliced function $\Phi_{x_{[k-1] \setminus \{i+1\}}}$ is a Freiman homomorphism. Finally set
\[A_{i+1} = \bigcup_{x_{[k-1] \setminus \{i+1\}} \in X} \{x_{[k-1] \setminus \{i+1\}}\} \times Y_{x_{[k-1] \setminus \{i+1\}}},\]
completing the inductive step.\\

Let $A_{k-1}$ be the set obtained in the case $i = k-1$. Then $\Phi|_{A_{k-1}}$ is a Freiman multi-homomorphism so we may apply Theorem~\ref{multihom} to finish the proof.\end{proof}

Let $\xi' > 0$ be a positive parameter to be specified later. We iteratively apply Claim~\ref{cubicalphidensemultihom} to find a subset $B \subset G_{[k-1]}$ of size at least $(1 - \xi') |G_{[k-1]}|$ and global multiaffine maps $\Psi_1, \dots, \Psi_m \colon G_{[k-1]} \to G_k$, where $m \leq \exp^{(O_{k}(1))}\Big(O_{k,p}(\varepsilon^{-1} \xi^{-1} {\xi'}^{-1})\Big)$, such that whenever $a_{[k-1]} \in B$ and $r \in R_{a_{[k-1]}}$ then $r \in \{\Psi_1(a_{[k-1]}, \dots, \Psi_m(a_{[k-1]})\}$. We begin this procedure by gathering all pairs $(a_{[k-1]}, r) \in G_{[k-1]} \times G_k$ such that $r \in R_{a_{[k-1]}}$ into the set $\mathcal{L}$, which we shall modify by removing some pairs in each step. More precisely, at $i$\textsuperscript{th} step we shall find a global multiaffine map $\Phi_i \colon G_{[k-1]} \to G_k$ and remove all pairs of the form $(x_{[k-1]}, \Phi_i(x_{[k-1]}))$ from $\mathcal{L}$.\\
\indent Suppose that we are in $i$\textsuperscript{th} step, and that we have defined maps $\Phi_1, \dots, \Phi_{i-1}$ so far. As long as there is a set $X \subset G_{[k-1]}$ of size at least $\xi' |G_{[k-1]}|$ such that for each $x_{[k-1]} \in X$ there is a pair $(x_{[k-1]}, r)$ still in $\mathcal{L}$, we may define a map $\Psi \colon X \to G_k$ so that $(x_{[k-1]}, \Psi(x_{[k-1]})) \in \mathcal{L}$. By Claim~\ref{cubicalphidensemultihom}, there is a further subset $X' \subset X$ of size
\[\bigg(\exp^{(O_{k}(1))}\Big(O_{k,p}(\xi^{-1} {\xi'}^{-1} \rho^{-1})\Big)\bigg)^{-1} |G_{[k-1]}|\]
and a global multiaffine map $\Phi_i \colon G_{[k-1]} \to G_k$ such that $\Psi|_{X'} = \Phi_i|_{X'}$. Thus, removing all pairs of the form $(x_{[k-1]}, \Phi_i(x_{[k-1]}))$ form $\mathcal{L}$ decreases the size of $\mathcal{L}$ by at least $|X'|$.\\
\indent This procedure has to terminate in $m \leq \exp^{(O_{k}(1))}\Big(O_{k,p}(\varepsilon^{-1} \xi^{-1} {\xi'}^{-1})\Big)$ steps as the initial size of $\mathcal{L}$ is at most
\[\sum_{a_{[k-1]}} |R_{a_{[k-1]}}| \leq \sum_{a_{[k-1]}} \sum_{x_{[k-1]}} \frac{\xi^{-1}}{|G_{[k-1]}|} |S_{x_{[k-1]}, a_{[k-1]}}| \leq \xi^{-1}\rho^{-2}|G_{[k-1]}|,\]
where we used Lemma~\ref{largeSpec} in the last step.\\

We now return to~\eqref{multiConvMainStepFirstEqn}. We have the following identity for each $a_{[k]}$
\[\square f_{\bcdot}\,(a_{[k]}) = \sum_{r \in G_k} \omega^{r \cdot a_k}\, \exx_{d_k} \omega^{-d_k \cdot r} \square f_{\bcdot}\,(a_{[k-1]}, d_k).\]
Using the stronger conclusion of Claim~\ref{frequentlargefcsclaim} for every $a_{[k-1]} \in B$ to see that
\[\Big\|\sum_{r \notin \{\Phi_1(a_{[k-1]}), \dots, \Phi_m(a_{[k-1]})\}} \exx_{d_k} \omega^{r \cdot (a_k - d_k)} \square f_{\bcdot}\,(a_{[k-1]}, d_k)\Big\|_{L^2, a_k} \leq \rho^2 + \xi\rho^{-2},\]
and the weaker conclusion for $a_{[k-1]} \setminus B$ to see that
\[\Big\|\sum_{r \notin \{\Phi_1(a_{[k-1]}), \dots, \Phi_m(a_{[k-1]})\}} \exx_{d_k} \omega^{r \cdot (a_k - d_k)} \square f_{\bcdot}\,(a_{[k-1]}, d_k)\Big\|_{L^2, a_k} \leq 1,\]
Moreover, if $R$ is disjoint from $R_{a_{[k-1]}}$ then we have a stronger bound
we obtain approximation
\begin{align}\square f_{\bcdot}\,(a_{[k]}) &\apps{\rho^2 + \xi\rho^{-2} + {\xi'}^{1/2}}_{L^2, a_{[k]}}  \sum_{r \in \{\Phi_1(a_{[k-1]}), \dots, \Phi_m(a_{[k-1]})\}}\, \omega^{r \cdot a_k}\, \exx_{d_k} \omega^{-d_k \cdot r} \square f_{\bcdot}\,(a_{[k-1]}, d_k). \label{cubicalapproxii}
\end{align}

\noindent \textbf{Step 3. Inclusion-exclusion argument.} Like in~\cite{U4paper} and~\cite{genPaper}, we now have to be careful about whether some values among $\Psi_1(a_{[k-1]}, \dots, \Psi_m(a_{[k-1]})$ are equal. To this end, we make the convention that the term coming from $r = \Phi_i(a_{[k-1]})$
\[\omega^{\Phi_i(a_{[k-1]}) \cdot a_k}\, \exx_{d_k} \omega^{-d_k \cdot \Phi_i(a_{[k-1]})} \square f_{\bcdot}\,(a_{[k-1]}, d_k)\]
contributes only when $\Phi_i(a_{[k-1]}) \notin \{\Phi_1(a_{[k-1]}), \dots, \Phi_{i-1}(a_{[k-1]})\}$. Algebraic manipulation (which is essentially the Inclusion-Exclusion principle) yields

\begin{align*}&\sum_{r \in \{\Phi_1(a_{[k-1]}), \dots, \Phi_m(a_{[k-1]})\}}\, \omega^{r \cdot a_k}\, \exx_{d_k} \omega^{-d_k \cdot r} \square f_{\bcdot}\,(a_{[k-1]}, d_k) \\
&\hspace{1cm} = \sum_{i \in [m]} \mathbbm{1}\Big(\Phi_i(a_{[k-1]}) \notin \{\Phi_1(a_{[k-1]}), \dots, \Phi_{i-1}(a_{[k-1]})\}\Big) \omega^{\Phi_i(a_{[k-1]}) \cdot a_k}\, \exx_{d_k} \omega^{-d_k \cdot \Phi_i(a_{[k-1]})} \square f_{\bcdot}\,(a_{[k-1]}, d_k) \\
&\hspace{1cm} = \sum_{i \in [m]} \mathbbm{1}\Big(\Phi_i(a_{[k-1]}) \not= \Phi_1(a_{[k-1]})\Big)\cdots \mathbbm{1}\Big(\Phi_i(a_{[k-1]}) \not= \Phi_{i-1}(a_{[k-1]})\Big) \\
&\hspace{10cm}\omega^{\Phi_i(a_{[k-1]}) \cdot a_k}\, \exx_{d_k} \omega^{-d_k \cdot \Phi_i(a_{[k-1]})} \square f_{\bcdot}\,(a_{[k-1]}, d_k) \\
&\hspace{1cm} = \sum_{i \in [m]} \Big(1 - \mathbbm{1}\Big(\Phi_i(a_{[k-1]}) = \Phi_1(a_{[k-1]})\Big) \Big)\cdots \Big(1- \mathbbm{1}\Big(\Phi_i(a_{[k-1]}) = \Phi_{i-1}(a_{[k-1]})\Big) \Big)\\
&\hspace{10cm}\omega^{\Phi_i(a_{[k-1]})\cdot a_k}\, \exx_{d_k} \omega^{-d_k \cdot \Phi_i(a_{[k-1]})} \square f_{\bcdot}\,(a_{[k-1]}, d_k) \\
&\hspace{1cm} = \sum_{i \in [m]} \sum_{I \subset [i-1]} (-1)^{|I|} \mathbbm{1}\Big((\forall j \in I) \Phi_j(a_{[k-1]}) = \Phi_i(a_{[k-1]})\Big) \omega^{\Phi_i(a_{[k-1]})\cdot a_k}\, \exx_{d_k} \omega^{-d_k \cdot \Phi_i(a_{[k-1]})} \square f_{\bcdot}\,(a_{[k-1]}, d_k).
\end{align*}

We need to approximate varieties $\{a_{[k-1]} \in G_{[k-1]} \colon (\forall j \in I) \Phi_j(a_{[k-1]}) = \Phi_i(a_{[k-1]})\}$ by varieties of low codimension. To this end, we apply Lemma~\ref{varOuterApprox} with approximation parameter $\xi'' > 0$ (to be specified later) to obtain some $s_{i, I} \leq \log_p {\xi''}^{-1}$ and a multiaffine map $\tau_{i, I} \colon G_{[k-1]} \to \mathbb{F}_p^{s_{i, I}}$ such that
\[\Big\{a_{[k-1]} \in G_{[k-1]} \colon (\forall j \in I) \Phi_j(a_{[k-1]}) = \Phi_i(a_{[k-1]})\Big\}\, \subset\, \Big\{a_{[k-1]} \in G_{[k-1]} \colon \tau_{i, I}(a_{[k-1]}) = 0\Big\}\]
and
\[\Big|\Big\{a_{[k-1]} \in G_{[k-1]} \colon \tau_{i, I}(a_{[k-1]}) = 0\Big\}\,\setminus\,\Big\{a_{[k-1]} \in G_{[k-1]} \colon (\forall j \in I) \Phi_j(a_{[k-1]}) = \Phi_i(a_{[k-1]})\Big\}\Big| \leq \xi''|G_{[k-1]}|.\]

Finally, recalling~\eqref{cubicalapproxii}, we end up with approximation
\begin{align*}\square f_{\bcdot}\,(a_{[k]}) \apps{\varepsilon'}_{L^2, a_{[k]}} &\sum_{i \in [m]} \sum_{I \subset [i-1]} (-1)^{|I|} \mathbbm{1}\Big(\tau_{i, I}(a_{[k-1]}) = 0\Big) \omega^{\Phi_i(a_{[k-1]})\cdot a_k}\, \exx_{d_k} \omega^{-d_k \cdot \Phi_i(a_{[k-1]})} \square f_{\bcdot}\,(a_{[k-1]}, d_k)\\
=&\sum_{i \in [m]} \sum_{I \subset [i-1]} (-1)^{|I|} p^{-s_{i, I}}\sum_{\lambda \in \mathbb{F}_p^{s_{i, I}}} \omega^{\lambda \cdot \tau_{i, I}(a_{[k-1]})} \omega^{\Phi_i(a_{[k-1]})\cdot a_k}\, \exx_{d_k} \omega^{-d_k \cdot \Phi_i(a_{[k-1]})} \square f_{\bcdot}\,(a_{[k-1]}, d_k),\end{align*}
where $\varepsilon' = \rho^2 + \xi\rho^{-2} + {\xi'}^{1/2} + 2^m {\xi''}^{1/2}$.\\
\indent After a slight change of notation, we may rewrite the approximation sum as
\[\sum_{i \in [m']}\, c_i\omega^{\alpha_i(a_{[k]})}\, \exx_{d_k} \square f_{\bcdot}\,(a_{[k-1]}, d_k) \omega^{\beta_i(a_{[k-1]}, d_k)},\]
for some $m' \leq 2^m {\xi''}^{-1}$, constants $c_i \in \mathbb{D}$ and multiaffine forms $\alpha_i, \beta_i \colon G_{[k]} \to \mathbb{F}_p$. To complete the proof, we pick parameters $\rho, \xi, \xi', \xi''$ to be
\[\rho = \frac{\varepsilon}{4}, \xi = \frac{\varepsilon^3}{64}, \xi' = \frac{\varepsilon^2}{16}, \xi'' = \frac{\varepsilon^2}{2^{2m+4}}.\qedhere\]\end{proof}

\vspace{\baselineskip}

\boldSection{Inverse theorem for some directional Gowers uniformity norms}

Recall that we are interested in the norm
\[\mathsf{U}\big(G_1, G_2, \dots, G_k, \underbrace{G^{\oplus}, \dots, G^{\oplus}}_{r}\big),\]
where $G^{\oplus} = G_1 \oplus \dots \oplus G_k$ appears $r$ times in the norm subscript for some $r \in \mathbb{N}$.\footnote{When we view $G_1 \tdt G_k$ as an abelian group, we denote this product by $G^{\oplus}$. This is a different viewpoint from the previous parts of the paper where $G_1 \tdt G_k$ was abbreviated as $G_{[k]}$ and meant the set of $k$-tuples where $i$\textsuperscript{th} element belongs to $G_i$.} To simplify the notation slightly, we write $\mathsf{U}\big(G_1, G_2, \dots, G_k, G^{\oplus} \times r\big)$ instead. In this section we prove the inverse theorem for this norm, stated in the introductory section as Theorem~\ref{inverseUmixed}, which is the main result of this work. 

We prove the theorem by induction on $r$. We prove the base case separately, as it will have an important role in the proof of the general case.
\vspace{\baselineskip}

\boldSubsection{Inverse theorem for $\mathsf{U}(G_1, G_2, \dots, G_k, G^{\oplus})$ norm}

In this short subsection we use Theorem~\ref{multConv} which concerns approximating cubical convolutions to prove the base case of Theorem~\ref{inverseUmixed}.

\begin{proposition}[Inverse theorem for $\mathsf{U}(G_1, G_2, \dots, G_k, G^{\oplus})$ norm]\label{r1caseinverse}Let $f \colon G^{\oplus} \to \mathbb{D}$ be a function such that
\[\|f\|_{\mathsf{U}\big(G_1, G_2, \dots, G_k, G^{\oplus}\big)} \,\geq\, c.\]
Then there exists a multilinear form $\mu \colon G_1 \tdt G_k \to \mathbb{F}_p$ and functions $u_i \colon G_{[k] \setminus \{i\}} \to \mathbb{D}$ for $i \in [k]$ such that
\[\exx_{x_{[k]} \in G^\oplus} f(x) \omega^{\mu(x_{[k]})} \prod_{i \in [k]} u_i(x_{[k] \setminus \{i\}}) \geq \bigg(\exp^{(O_{k}(1))}\Big(O_{k,p}(c^{-1})\Big)\bigg)^{-1}.\]
\end{proposition}

\begin{proof}Expanding out the definition of the norm, we obtain
\begin{align}c^{2^{k+1}} \leq &\|f\|_{\mathsf{U}\big(G_1, G_2, \dots, G_k, G^{\oplus}\big)}^{2^{k + 1}}\nonumber\\
= & \exx_{h_1 \in G_1, \dots, h_k \in G_k} \exx_{x \in G^{\oplus}} \mder_{h_1, \dots, h_k} f(x) \overline{\Big(\exx_{a \in G^{\oplus}} \mder_{h_1, \dots, h_k} f(x - a)\Big)}\nonumber\\
= & \exx_{h_1 \in G_1, \dots, h_k \in G_k} \exx_{x \in G^{\oplus}} \mder_{h_1, \dots, h_k} f(x) \overline{\square f(h_{[k]})}.\label{genU2eq1}\end{align}

Apply Theorem~\ref{multConv} to $f$ viewed as a function on $G_1 \tdt G_k$ with approximation parameter $\frac{1}{2}c^{2^{k+1}}$ to obtain a positive integer $m \leq  \exp^{(O_{k}(1))}\Big(O_{k,p}(c^{-1})\Big)$, a multiaffine map $\alpha \colon G_{[k]} \to \mathbb{F}_p^m$ and a function $g \colon \mathbb{F}_p^m \to \mathbb{D}$ such that
\[\square f_\bcdot\,(h_{[k]}) \apps{\frac{1}{2}c^{2^{k+1}}}_{L^2, a_{[k]}} g(\alpha(h_{[k]})).\]

Going back to~\eqref{genU2eq1} and using Cauchy-Schwarz inequality, we obtain
\begin{align*}\bigg|\exx_{h_1 \in G_1, \dots, h_k \in G_k} \exx_{x \in G^{\oplus}} \mder_{h_1, \dots, h_k} f(x) \overline{g(\alpha(h_{[k]}))}\bigg| \geq &\bigg|\exx_{h_1 \in G_1, \dots, h_k \in G_k} \exx_{x \in G^{\oplus}} \mder_{h_1, \dots, h_k} f(x) \overline{\square f(h_{[k]})}\bigg|\\
&\hspace{1cm} - \bigg|\exx_{h_1 \in G_1, \dots, h_k \in G_k} \exx_{x \in G^{\oplus}} \mder_{h_1, \dots, h_k} f(x) \overline{\Big(\square f(h_{[k]}) - g(\alpha(h_{[k]}))\Big)}\bigg|\\
\geq& c^{2^{k+1}} - \exx_{h_1 \in G_1, \dots, h_k \in G_k} \Big|\Big(\square f(h_{[k]}) - g(\alpha(h_{[k]}))\Big)\Big|\\
\geq & \frac{1}{2}c^{2^{k+1}}.\end{align*}
Writing $g(\alpha(h_{[k]}))$ as $p^{-m}\sum_{\lambda, \mu \in \mathbb{F}_p^m} g(\lambda) \omega^{\mu \cdot (\alpha(h_{[k]}) - \lambda)}$, we get
\[\frac{1}{2}c^{2^{k+1}} \leq p^{-m}\sum_{\lambda, \mu \in \mathbb{F}_p^m} \bigg|\exx_{h_1 \in G_1, \dots, h_k \in G_k} \exx_{x \in G^{\oplus}} \mder_{h_1, \dots, h_k} f(x) \overline{g(\lambda)} \omega^{-\mu \cdot (\alpha(h_{[k]}) - \lambda)}\bigg|,\]
from which it follows that we may find $\lambda, \mu \in \mathbb{F}_p^m$ such that
\[\frac{1}{2p^m}c^{2^{k+1}} \leq \bigg|\exx_{h_1 \in G_1, \dots, h_k \in G_k} \exx_{x \in G^{\oplus}}\mder_{h_1, \dots, h_k} f(x) \omega^{-\mu \cdot \alpha(h_{[k]})}\bigg|.\]
Writing $\alpha(h_{[k]}) = \sum_{I \subseteq [k]} \alpha_I(h_I)$ in terms of its multilinear parts, which are multilinear maps each depending on a subset of variables, and expanding out, we get
\[\frac{1}{2p^m}c^{2^{k+1}} \leq \bigg|\exx_{h_1 \in G_1, \dots, h_k \in G_k} \exx_{x \in G^{\oplus}}\prod_{I \subseteq [k]} \on{Conj}^{k - |I|} f\Big((x_i + h_i)_{i \in I},\, (x_i)_{i \in [k] \setminus I}\Big) \omega^{-\mu \cdot \alpha_I(h_I)}\bigg|.\]
Making a change of variables $h'_i = x_i + h_i$ to replace $h_i$, we get that
\[\frac{1}{2p^m}c^{2^{k+1}} \leq \bigg|\exx_{h'_1 \in G_1, \dots, h'_k \in G_k} \exx_{x \in G^{\oplus}}\prod_{I \subseteq [k]} \on{Conj}^{k - |I|} f\Big((h'_i)_{i \in I},\, (x_i)_{i \in [k] \setminus I}\Big) \omega^{-\mu \cdot \alpha_I((h' - x)_I)}\bigg|.\]
By averaging, we find $x \in G^{\oplus}$ such that
\[\frac{1}{2p^m}c^{2^{k+1}} \leq \bigg|\exx_{h'_1 \in G_1, \dots, h'_k \in G_k} \prod_{I \subseteq [k]} \on{Conj}^{k - |I|} f\Big((h'_i)_{i \in I},\, (x_i)_{i \in [k] \setminus I}\Big) \omega^{-\mu \cdot \alpha_I((h' - x)_I)}\bigg|.\]
Note that the only term on the right-hand-side that depends on all variables $h'_1, \dots, h'_k$ is the one corresponding to $I = [k]$. Moreover we may expand
\[\omega^{-\mu \cdot \alpha_{[k]}((h' - x)_{[k]})} = \prod_{J \subseteq [k]} \omega^{-(-1)^{k - |J|}\mu \cdot \alpha_{[k]}(h'_J, x_{[k] \setminus J})},\]
out of which only $\omega^{-\mu \alpha_{[k]}(h'_{[k]})}$ depends on all of $h'_1, \dots, h'_k$. The claim follows after straightforward algebraic manipulation which allows us to reorganize the expression into the desired form.\end{proof}

We now make a digression to compare this result with the classical univariate case. Recall that in the classical theory of uniformity norms we have the fundamental fact that the large $\mathsf{U}^2$ norm corresponds to having large Fourier coefficients, i.e.\ having non-empty large spectrum, and also recall that $\mathsf{U}^2$ norm is the base case in the univariate setting for proving the general inverse theorems. The proposition that we have just proved is a direct generalization of this fact to the multivariate setting, in the sense that we now explain. Let us recall the definition of the \emph{large multilinear spectrum} of a function $f \colon G_{[k]} \to \mathbb{D}$ which was given in the introductory section.

\begin{defin}[Definition~\label{mlspecdefin}]\label{mlspecdefin2}Let $f \colon G_1 \tdt G_k \to \mathbb{D}$ be a function and let $\varepsilon > 0$. We define $\varepsilon$-large multilinear spectrum of $f$ to be the set
\[\mls_{\varepsilon}(f) = \Big\{\mu \in \on{ML}(G_{[k]} \to \mathbb{F}_p) \colon \|f \omega^{\mu}\|_{\square^k} \geq \varepsilon\Big\},\]
where $ \on{ML}(G_{[k]} \to \mathbb{F}_p) $ stands for the set of all multilinear forms on $G_1 \tdt G_k$ and where $\|\cdot\|_{\square^k}$ stands for the box norm with respect to sets $G_1, \dots, G_k$.\end{defin}

To see why we may think of the large multilinear spectrum as the generalization of the large Fourier spectrum of a function of a single variable, consider an arbitrary function $f \colon G_1 \to \mathbb{D}$. If $r$ is an $\varepsilon$-large Fourier coefficient of $f$, we have the correlation $\Big|\ex_{x \in G_1} f(x) \omega^{-r \cdot x}\Big| \geq \varepsilon$. We may set $\alpha(x) = -r\cdot x$, which is a linear form associated to $r$ in a natural way. Furthermore, notice that the `box norm' in the case one variable\footnote{In the case of single variable, the box norm is not defined.} would be given by expression
\[\|g\|_{\square^1}^2 = \exx_{x, a \in G_1} g(x + a) \overline{g(x)} = \Big|\exx_{x \in G_1} g(x)\Big|^2,\]
which is just the absolute value of the expectation squared. Hence, we may interpret the correlation $\Big|\ex_{x \in G_1} f(x) \omega^{-r \cdot x}\Big| \geq \varepsilon$ as
\[\|f \omega^\alpha \|_{\square^1} \geq \varepsilon.\]
More conceptually, the spectrum consists of algebraically structured functions whose phases determine $f$ up to lower-order terms. (In the case of single variable, the lower-order terms are exactly the constant functions.)\\

Proposition~\ref{r1caseinverse} can be rephrased as the fact that the large value of the $\mathsf{U}(G_1, G_2, \dots, G_k, G^{\oplus})$ norm implies that the large multilinear spectrum is non-empty. It turns out that the large multilinear spectrum has some further properties that are analogous to those that hold for the usual Fourier spectrum. Once we complete the proof of the inverse theorems, we shall return to the discussion of the large multilinear spectrum.\\
%
%

\vspace{\baselineskip}

\boldSubsection{Very brief overview of the proof}

We now begin the proof of the inverse theorem for the $\mathsf{U}\big(G_1, G_2, \dots, G_k, G^{\oplus}\times r\big)$ norm in the general case. The proof will consist of three steps, similarly to the univariate theory. In the first step we find a multilinear form that is related to the given function. This is the content of the next result.

\begin{proposition}\label{mlStructStep}Let $f \colon G^\oplus \to \mathbb{D}$ be a function such that
\[\|f\|_{\mathsf{U}\big(G_1, G_2, \dots, G_k, G^{\oplus} \times r\big)} \,\geq\, c.\]
Then we may find a multilinear form $\psi \colon \underbrace{G^{\oplus} \tdt G^{\oplus}}_{r-1} \times G_1 \tdt G_k \to \mathbb{F}_p$ such that
\begin{equation}\label{mlcorrelationinvthm}\Big|\exx_{\ssk{a^{(1)}, \dots, a^{(r-1)} \in G^{\oplus}\\b_1 \in G_1, \dots, b_k \in G_k\\x \in G^{\oplus}}} \mder_{a^{(1)}, \dots, a^{(r-1)}} \mder_{b_1, \dots, b_k} f(x) \omega^{\psi(a^{(1)}, \dots, a^{(r-1)}, b_1, \dots, b_k)}\Big| \geq \Big(\exp^{(O_{k,r}(1))}(O_{k,r, p}(c^{-1}))\Big)^{-1}.\end{equation}
\end{proposition}

The second step is to prove some symmetry properties of the multilinear form $\psi$ provided by Proposition~\ref{mlStructStep}, based on ideas of Green and Tao~\cite{StrongU3}. Finally, once we have proved that $\psi$ is sufficiently symmetric we may use a polarization identity to produce the desired polynomial that $f$ correlates with. However, even though the overall structures of this proof and of that in the univariate setting are similar, given the different roles that arguments $a^{(i)}$ and $b_i$ of $\psi$ play, the multivariate case is more subtle.\\
\vspace{\baselineskip}

\boldSubsection{Correlation with multilinear form}

This subsection is devoted to the proof of Proposition~\ref{mlStructStep}.

\begin{proof}[Proof of Proposition~\ref{mlStructStep}] We shall use the base case of the inverse theorem (Proposition~\ref{r1caseinverse}) in this proof, but we first need to do some preparation. Let us define the vector space $H$ by
\[H = G_1 \oplus \dots \oplus G_k \oplus \underbrace{G^\oplus \oplus \dots \oplus G^\oplus}_{r-1}.\]
Write $\overset{(i)}{G^\oplus}$ for the subgroup given by $i$\textsuperscript{th} copy of $G^\oplus$ in the space $H$, which is a subgroup of the form $\overset{(i)}{G^\oplus} = \{0\} \tdt \{0\} \times G^{\oplus} \times \{0\} \tdt \{0\}$. Then, define an auxiliary function $\tilde{f} \colon H \to \mathbb{D}$ by
\[\tilde{f}\Big(x_1, \dots, x_k, y^{(1)}_{[k]}, \dots, y^{(r-1)}_{[k]}\Big) = f\Big(x_1 + \sum_{i \in [r-1]} y^{(i)}_1, \dots, x_k + \sum_{i \in [r-1]} y^{(i)}_k\Big)\]
for all $x_1 \in G_1, \dots, x_k \in G_k, y^{(1)}_{[k]}, \dots, y^{(r-1)}_{[k]} \in G^\oplus$. The relevance of the function $\tilde{f}$ stems from the following equality of norms
\[\Big\|f\Big\|_{\mathsf{U}(G_1, \dots, G_k, G^\oplus \times r)} = \Big\|\tilde{f}\Big\|_{\mathsf{U}(G_1, \dots, G_k, \overset{(1)}{G^\oplus}, \dots, \overset{(r-1)}{G^\oplus}, H)}.\]
This equality follows from a straightforward algebraic manipulation (below we use a more explicit notation and for example we write $(0,0,\dots, b_k)$ for the element of $G_{[k]}$ that has first $k-1$ coordinates equal to $0$ and the last one equal to $b_k$):
\begin{align*}\Big\|f\Big\|_{\mathsf{U}(G_1, \dots, G_k, G^\oplus \times r)}^{2^{k + r}} =& \exx_{\ssk{b_{[k]} \in G_{[k]}\\a^{(1)}_{[k]}, \dots, a^{(r-1)}_{[k]}, d_{[k]} \in G^\oplus\\x_{[k]} \in G_{[k]}}} \mder_{\ssk{(b_1, 0, \dots, 0)\\\vdots\\(0,0,\dots, b_k)}} \mder_{a^{(1)}_{[k]}, \dots, a^{(r-1)}_{[k]}, d_{[k]}} f(x_1, \dots, x_k)\\
= & \exx_{\ssk{n_{[k]} \in G_{[k]}\\a^{(1)}_{[k]}, \dots, a^{(r-1)}_{[k]} \in G^\oplus\\x_{[k]} \in G_{[k]}\\y^{(1)}_{[k]}, \dots, y^{(r-1)}_{[k]} \in G^\oplus\\d^{(0)}_{[k]}, \dots, d^{(r-1)}_{[k]} \in G^\oplus}} \mder_{\ssk{(b_1, 0, \dots, 0)\\\vdots\\(0,0,\dots, b_k)}} \mder_{a^{(1)}_{[k]}, \dots, a^{(r-1)}_{[k]}, d^{(0)}_{[k]} + \dots + d^{(r-1)}_{[k]}} f\Big(x_1 + \sum_{i \in [r-1]} y^{(i)}_1, \dots, x_k + \sum_{i \in [r-1]} y^{(i)}_k\Big)\\
= & \exx_{\ssk{b_{[k]} \in G_{[k]}\\a^{(1)}_{[k]}, \dots, a^{(r-1)}_{[k]} \in G^\oplus\\x_{[k]} \in G_{[k]}\\y^{(1)}_{[k]}, \dots, y^{(r-1)}_{[k]} \in G^\oplus}} \mder_{\ssk{(b_1, 0, \dots, 0)\\\vdots\\(0,0,\dots, b_k, 0, \dots, 0)}}\mder_{\ssk{(0_{[k]}, a^{(1)}_{[k]}, 0, \dots, 0)\\\vdots\\(0, \dots, 0, a^{(r-1)}_{[k]})}} \\
&\hspace{3cm} \mder_{(d^{(0)}_1, \dots, d^{(0)}_k, d^{(1)}_{[k]}, \dots, d^{(r-1)}_{[k]})} \tilde{f}(x_1, \dots, x_k, y^{(1)}_{[k]}, \dots, y^{(r-1)}_{[k]})\\
= & \Big\|\tilde{f}\Big\|_{\mathsf{U}(G_1, \dots, G_k, \overset{(1)}{G^\oplus}, \dots, \overset{(r-1)}{G^\oplus}, H)}^{2^{k+r}}.\end{align*}

We may apply Proposition~\ref{r1caseinverse} to $\tilde{f}$ to obtain functions $u_i \colon G_{[k] \setminus \{i\}} \times (G^{\oplus})^{r-1} \to \mathbb{D}$ for $i \in [k]$, functions $v_i \colon G_{[k]} \times (G^{\oplus})^{r-2} \to \mathbb{D}$ for $i \in [r-1]$ and a multilinear form $\psi \colon G_1 \times G_2 \tdt G_k \times \underbrace{G^\oplus \tdt G^\oplus}_{r-1} \to \mathbb{F}_p$ such that
\begin{align*}&\Big|\exx_{\ssk{x_1, \dots, x_k\\y^{(1)}_{[k]}, \dots, y^{(r-1)}_{[k]}}} \tilde{f}(x_1, \dots, x_k, y^{(1)}_{[k]}, \dots, y^{(r-1)}_{[k]}) \prod_{i \in [k]} u_i(x_{[k] \setminus \{i\}}, y^{(1)}_{[k]}, \dots, y^{(r-1)}_{[k]})\\
&\hspace{2cm} \prod_{i \in [r-1]} v_i(x_{[k]}, y^{(1)}_{[k]}, \dots, y^{(i - 1)}_{[k]}, y^{(i + 1)}_{[k]}, \dots, y^{(r - 1)}_{[k]}) \omega^{\psi(x_1, \dots, x_k, y^{(1)}_{[k]}, \dots, y^{(r-1)}_{[k]})}\Big|\\
&\hspace{10cm}\geq \bigg(\exp^{(O_{k,r}(1))}\Big(O_{k,r,p}(c^{-1})\Big)\bigg)^{-1}.\end{align*}

Notice that the only terms having all $k + r - 1$ variables $x_1, \dots, x_k, y^{(1)}_{[k]}, \dots, y^{(r-1)}_{[k]}$ are $\tilde{f}(x_1, \dots,$ $x_k,$ $y^{(1)}_{[k]}, \dots,$ $y^{(r-1)}_{[k]})$ and $\omega^{\psi(x_1, \dots, x_k, y^{(1)}_{[k]}, \dots, y^{(r-1)}_{[k]})}$. We may apply Lemma~\ref{gcs} to deduce the bound
\[\Big|\exx_{\ssk{a^{(1)}, \dots, a^{(r-1)} \in G^{\oplus}\\b_1 \in G_1, \dots, b_k \in G_k\\x \in G^{\oplus}}} \mder_{a^{(1)}, \dots, a^{(r-1)}} \mder_{\ssk{(b_1, 0_{[2,k]})\\\vdots\\(0_{[k-1]}, b_k)}} f(x) \omega^{\psi(a_1, \dots, a_k, b^{(1)}_{[k]}, \dots, b^{(r-1)}_{[k]})}\Big|\geq \bigg(\exp^{(O_{k,r}(1))}\Big(O_{k,r,p}(c^{-1})\Big)\bigg)^{-1}.\qedhere\]
\end{proof}

\vspace{\baselineskip}

\boldSubsection{Symmetry argument}

We now focus on the symmetry properties of a multilinear form $\psi \colon \underbrace{G^{\oplus} \tdt G^{\oplus}}_{r-1} \times G_1 \tdt G_k \to \mathbb{F}_p$ which obeys
\begin{equation}\exx_{\ssk{a^{(1)}, \dots, a^{(r-1)} \in G^{\oplus}\\d_{[k]} \in G_{[k]}\\x \in G^\oplus}} \mder_{a^{(1)}, \dots, a^{(r-1)}} \mder_{\ssk{(d_1, 0_{[2,k]})\\\vdots\\(0_{[k-1]}, d_k)}}f(x) \omega^{\psi(a^{(1)}, \dots, a^{(r-1)}, d_1, \dots, d_k)} \geq \xi\label{mlstreqn}\end{equation}
for some $\xi > 0$. It turns out that such a form is approximately symmetric in the variables $a^{(1)}, \dots, a^{(r-1)}$ and that is has another approximate symmetry property which allows us to replace $i$\textsuperscript{th} component of $a^{(j)}$ by $d_i$. These two properties are articulated in the following two propositions respectively.

\begin{proposition}[Symmetry I]\label{sym1prop}Let $\xi > 0$ and suppose that a multilinear form $\psi \colon (G^{\oplus})^{r-1} \times G_{[k]} \to \mathbb{F}_p$ satisfies~\eqref{mlstreqn}. Let $i < j$ be two elements of $[r-1]$. Define the multilinear map $\psi_{ij} \colon (G^{\oplus})^{r-1} \times G_{[k]} \to \mathbb{F}_p$ by
\begin{align}\psi_{ij}\Big(a^{(1)}, \dots, a^{(r-1)}, d_1, \dots, d_k\Big) =\, &\psi\Big(a^{(1)}, \dots, a^{(i-1)}, a^{(i)}, a^{(i+1)}, \dots, a^{(j-1)}, a^{(j)}, a^{(j+1)}, \dots, a^{(r-1)}, d_1, \dots, d_k\Big) \label{psidefeqn}\\
-&\psi\Big(a^{(1)}, \dots, a^{(i-1)}, a^{(j)}, a^{(i+1)}, \dots, a^{(j-1)}, a^{(i)}, a^{(j+1)}, \dots, a^{(r-1)}, d_1, \dots, d_k\Big),\nonumber\end{align}
where the argument of the second form in the expression is obtained by swapping $a^{(i)}$ and $a^{(j)}$. Then
\[\bias \psi_{ij} \geq \xi^8.\]\end{proposition}

\begin{proposition}[Symmetry II]\label{sym2prop}Let $\xi > 0$ and suppose that a multilinear form $\psi \colon (G^{\oplus})^{r-1} \times G_{[k]} \to \mathbb{F}_p$ satisfies~\eqref{mlstreqn}. Let $i \in [r-1]$ and let $j \in [k]$. Define the multilinear map $\psi'_{ij} \colon (G^{\oplus})^{[i-1]} \times G_j \times (G^{\oplus})^{[i+1, r-1]} \times G_{[k]} \to \mathbb{F}_p$ by
\begin{align}\psi'_{ij}\Big(a^{(1)}, \dots, a^{(i-1)}, u_j, a^{(i+1)}, &\dots, a^{(r-1)}, d_1, \dots, d_k\Big) =\,\nonumber\\
&\psi\Big(a^{(1)}, \dots, a^{(i-1)}, (0_{[k] \setminus \{j\}},\ls{j}\,u_j), a^{(i+1)}, \dots, a^{(r-1)}, d_1, \dots, d_{j-1}, d_j, d_{j+1}, \dots, d_k\Big)\nonumber\\
-&\psi\Big(a^{(1)}, \dots, a^{(i-1)}, (0_{[k] \setminus \{j\}},\ls{j}\,d_j), a^{(i+1)}, \dots, a^{(r-1)}, d_1, \dots, d_{j-1}, u_j, d_{j+1}, \dots, d_k\Big).\label{psidefeqn2}\end{align}
Then
\[\bias \psi'_{ij} \geq \xi^8.\]\end{proposition}

Once these two propositions are proved, we shall show that the only sources of the approximately symmetric functions in the sense above are functions with the corresponding exact symmetry properties.

\begin{proposition}\label{approxsymtoexact}Let $\xi > 0$. Assume that $p \geq r + 1$. Suppose that $\psi \colon (G^{\oplus})^{r-1} \times G_{[k]} \to \mathbb{F}_p$ is a multilinear form. Let $\psi_{ij}$ and $\psi'_{ij}$ be defined by~\eqref{psidefeqn} and by~\eqref{psidefeqn2} respectively. Suppose that for all $i, j \in [r-1]$ we have
\begin{equation}\bias \psi_{ij} \geq \xi\label{approxtoexsym1}\end{equation}
and for each $i \in [r-1], j \in [k]$ we have
\begin{equation}\bias \psi'_{ij} \geq \xi.\label{approxtoexsym2}\end{equation}
Then there is another multilinear form $\rho \colon (G^{\oplus})^{r-1} \times G_{[k]} \to \mathbb{F}_p$ such that
\[\bias (\psi - \rho) \geq \xi^{4(k + r)!^4}\]
and if $\rho_{ij}$ and $\rho'_{ij}$ are the multilinear forms defined by~\eqref{psidefeqn} and by~\eqref{psidefeqn2} for $\rho$ instead of $\psi$, then $\rho_{ij} = 0$ and $\rho'_{ij} = 0$.\end{proposition}
\vspace{\baselineskip}
\noindent\textbf{Proofs of approximate symmetry properties.} We now proceed to prove the stated results.

\begin{proof}[Proof of Proposition~\ref{sym1prop}]It suffices to prove the claim for $i = r-2$ and $j = r-1$. We start by expanding~\eqref{mlstreqn}
\begin{align*}\xi \leq & \exx_{a^{(1)}, \dots, a^{(r-1)} \in G^{\oplus}} \exx_{x_1, d_1 \in G_1, \dots, x_k, d_k \in G_k} \bigg(\prod_{I \subseteq [k]} \on{Conj}^{k - |I|} \mder_{a^{(1)}, \dots, a^{(r-1)}}f\Big((x_i + d_i)_{i \in I}, x_{[k] \setminus I}\Big)\bigg) \,\,\omega^{\psi(a^{(1)}, \dots, a^{(r-1)}, d_{[k]})}.\end{align*}
By averaging, there is $x_{[k]} \in G_{[k]}$ such that
\begin{align*}\xi \leq & \Big|\exx_{a^{(1)}, \dots, a^{(r-1)} \in G^{\oplus}} \exx_{d_1 \in G_1, \dots, d_k \in G_k} \bigg(\prod_{I \subseteq [k]} \on{Conj}^{k - |I|} \mder_{a^{(1)}, \dots, a^{(r-1)}}f\Big((x + d)_{I}, x_{[k] \setminus I}\Big)\bigg) \,\,\omega^{\psi(a^{(1)}, \dots, a^{(r-1)}, d_{[k]})}\Big|\\
\leq & \exx_{a^{(1)}, \dots, a^{(r-2)} \in G^{\oplus}} \exx_{d_1 \in G_1, \dots, d_k \in G_k}\bigg|\exx_{a^{(r-1)} \in G^{\oplus}}  \bigg(\prod_{I \subseteq [k]} \on{Conj}^{k - |I|} \mder_{a^{(1)}, \dots, a^{(r-1)}}f\Big((x + d)_{I}, x_{[k] \setminus I}\Big)\bigg) \,\,\omega^{\psi(a^{(1)}, \dots, a^{(r-1)}, d_{[k]})}\bigg|.\end{align*}
Apply Cauchy-Schwarz inequality to obtain
\begin{align*}\xi^2 \leq &  \exx_{a^{(1)}, \dots, a^{(r-2)} \in G^{\oplus}}\exx_{d_1 \in G_1, \dots, d_k \in G_k}  \bigg|\exx_{a^{(r-1)} \in G^{\oplus}} \bigg(\prod_{I \subseteq [k]} \on{Conj}^{k - |I|} \mder_{a^{(1)}, \dots, a^{(r-1)}}f\Big((x + d)_{I}, x_{[k] \setminus I}\Big)\bigg) \,\,\omega^{\psi(a^{(1)}, \dots, a^{(r-1)}, d_{[k]})}\bigg|^2\\
\leq &  \exx_{a^{(1)}, \dots, a^{(r-2)} \in G^{\oplus}}\exx_{d_1 \in G_1, \dots, d_k \in G_k}  \bigg|\exx_{a^{(r-1)} \in G^{\oplus}} \bigg(\prod_{I \subseteq [k]} \on{Conj}^{k - |I|} \mder_{a^{(1)}, \dots, a^{(r-2)}}f\Big((x+ d + a^{(r-1)})_{I}, (x + a^{(r-1)})_{[k] \setminus I}\Big)\bigg)\\
&\hspace{10cm} \omega^{\psi(a^{(1)}, \dots, a^{(r-1)}, d_{[k]})}\bigg|^2\\
= & \exx_{a^{(1)}, \dots, a^{(r-2)} \in G^{\oplus}}\exx_{d_1 \in G_1, \dots, d_k \in G_k}  \exx_{u, v \in G^{\oplus}} \bigg(\prod_{I \subseteq [k]} \on{Conj}^{k - |I|} \mder_{a^{(1)}, \dots, a^{(r-2)}}f\Big((x + d + u)_{I}, (x + u)_{[k] \setminus I}\Big)\\
&\hspace{5cm} \on{Conj}^{k + 1 - |I|} \mder_{a^{(1)}, \dots, a^{(r-2)}}f\Big((x + d + v)_{I}, (x + v)_{[k] \setminus I}\Big)\bigg) \,\,\omega^{\psi(a^{(1)}, \dots, a^{(r-2)}, u-v, d_{[k]})}.\end{align*}
Make a change of variables and replace $a^{(r-2)}$ by $z = - u - v - a^{(r-2)}$ to obtain
\begin{align*}\xi^2 \leq &\exx_{a^{(1)}, \dots, a^{(r-3)} \in G^{\oplus}}\exx_{d_1 \in G_1, \dots, d_k \in G_k}  \exx_{u, v, z \in G^{\oplus}} \bigg(\prod_{I \subseteq [k]} \on{Conj}^{k - |I|} \mder_{a^{(1)}, \dots, a^{(r-3)}, -u-v-z}f\Big((x + d + u)_{I}, (x + u)_{[k] \setminus I}\Big)\\
&\hspace{2cm} \on{Conj}^{k + 1 - |I|} \mder_{a^{(1)}, \dots, a^{(r-3)},-u-v-z}f\Big((x + d + v)_{I}, (x+ v)_{[k] \setminus I}\Big)\bigg) \,\,\omega^{\psi(a^{(1)}, \dots, a^{(r-3)}, -u-v-z, u-v, d_{[k]})}\\
=&\exx_{a^{(1)}, \dots, a^{(r-3)} \in G^{\oplus}}\exx_{d_1 \in G_1, \dots, d_k \in G_k}  \exx_{u, v, z \in G^{\oplus}} \bigg(\prod_{I \subseteq [k]} \on{Conj}^{k - |I|} \mder_{a^{(1)}, \dots, a^{(r-3)}}f\Big((x + d + u)_{I}, (x + u)_{[k] \setminus I}\Big)\\
&\hspace{2cm}\on{Conj}^{k + 1 - |I|} \mder_{a^{(1)}, \dots, a^{(r-3)}}f\Big((x + d - v - z)_{I}, (x - v - z)_{[k] \setminus I}\Big)\\
&\hspace{2cm}\on{Conj}^{k + 1 - |I|} \mder_{a^{(1)}, \dots, a^{(r-3)}}f\Big((x + d + v)_{I}, (x + v)_{[k] \setminus I}\Big)\\
&\hspace{2cm}\on{Conj}^{k - |I|} \mder_{a^{(1)}, \dots, a^{(r-3)}}f\Big((x + d - u - z)_{I}, (x - u - z)_{[k] \setminus I}\Big)\bigg)\\
&\hspace{2cm}\omega^{-\psi(a^{(1)}, \dots, a^{(r-3)}, u + z, u, d_{[k]})} \omega^{\psi(a^{(1)}, \dots, a^{(r-3)}, v + z, v, d_{[k]})} \omega^{\psi(a^{(1)}, \dots, a^{(r-3)}, u, v, d_{[k]}) - \psi(a^{(1)}, \dots, a^{(r-3)}, v, u, d_{[k]})}.\end{align*}
Recall the notation $\psi_{r-2\, r-1}(a_1, \dots, a_{r-3}, u, v, d_{[k]}) = \psi(a_1, \dots, a_{r-3}, u, v, d_{[k]}) - \psi(a_1, \dots, a_{r-3}, v, u, d_{[k]})$. By Gowers-Cauchy-Schwarz inequality (Lemma~\ref{gcs}) for the 2-dimensional box norm with respect to variables $u$ and $v$, we see that
\begin{align*}\xi^8 \leq &\exx_{a^{(1)}, \dots, a^{(r-3)} \in G^{\oplus}}\exx_{d_1 \in G_1, \dots, d_k \in G_k}  \exx_{z \in G^{\oplus}} \exx_{\substack{u,v \in G^{\oplus}\\u',v' \in G^{\oplus}}} \omega^{\psi_{r-2\, r-1}(a^{(1)}, \dots, a^{(r-3)}, u, v, d_{[k]}) }\\
&\hspace{2cm} \omega^{-\psi_{r-2\, r-1}(a^{(1)}, \dots, a^{(r-3)}, u', v, d_{[k]}) } \omega^{-\psi_{r-2\, r-1}(a^{(1)}, \dots, a^{(r-3)}, u, v', d_{[k]}) } \omega^{\psi_{r-2\, r-1}(a^{(1)}, \dots, a^{(r-3)}, u', v', d_{[k]})}\\
= & \bias \psi_{r-2\, r-1}.\qedhere\end{align*}
\end{proof}

\begin{proof}[Proof of Proposition~\ref{sym2prop}]It suffices to prove the claim in the case $i = 1$ and $j = k$. We start by expanding~\eqref{mlstreqn}
\begin{align*}\xi \leq & \exx_{a^{(1)}, \dots, a^{(r-1)} \in G^{\oplus}} \exx_{x_1, d_1 \in G_1, \dots, x_k, d_k \in G_k} \bigg(\prod_{I \subseteq [k]} \on{Conj}^{k - |I|} \mder_{a^{(1)}, \dots, a^{(r-1)}}f\Big((x + d)_{I}, x_{[k] \setminus I}\Big)\bigg) \,\,\omega^{\psi(a^{(1)}, \dots, a^{(r-1)}, d_{[k]})}.\end{align*}
By averaging, there is $x_{[k]}$ such that
\begin{align*}\xi \leq & \bigg|\exx_{a^{(1)}, \dots, a^{(r-1)} \in G^{\oplus}} \exx_{d_{[k]} \in G_{[k]}} \bigg(\prod_{I \subseteq [k]} \on{Conj}^{k - |I|} \mder_{a^{(1)}, \dots, a^{(r-1)}}f\Big((x + d)_{I}, x_{[k] \setminus I}\Big)\bigg) \,\,\omega^{\psi(a^{(1)}, \dots, a^{(r-1)}, d_{[k]})}\bigg|\\
= & \bigg|\exx_{a^{(1)}, \dots, a^{(r-1)} \in G^{\oplus}} \exx_{d_{[k]} \in G_{[k]}} \bigg(\prod_{I \subseteq [k]} \on{Conj}^{k - |I|} \mder_{a^{(2)}, \dots, a^{(r-1)}}f\Big((x + d + a^{(1)})_I, (x + a^{(1)})_{[k] \setminus I}\Big)\\
&\hspace{5cm}\overline{\on{Conj}^{k - |I|} \mder_{a^{(2)}, \dots, a^{(r-1)}}f\Big((x + d)_I, x_{[k] \setminus I}\Big)}\bigg) \,\,\omega^{\psi(a^{(1)}, \dots, a^{(r-1)}, d_{[k]})}\bigg|\\
= &\bigg| \exx_{a^{(1)}, \dots, a^{(r-1)} \in G^{\oplus}} \exx_{d_{[k]} \in G_{[k]}} \omega^{\psi(a^{(1)}, \dots, a^{(r-1)}, d_{[k]})}\\
&\hspace{2cm}\bigg(\prod_{I \subseteq [k-1]} \on{Conj}^{k - 1 - |I|} \mder_{a^{(2)}, \dots, a^{(r-1)}}f\Big((x + d + a^{(1)})_I, (x + a^{(1)})_{[k-1] \setminus I}, {}^k\,x_k + d_k + a^{(1)}_{k}\Big)\bigg)\\
&\hspace{2cm}\bigg(\prod_{I \subseteq [k-1]} \on{Conj}^{k - |I|} \mder_{a^{(2)}, \dots, a^{(r-1)}}f\Big((x + d + a^{(1)})_I, (x + a^{(1)})_{[k-1] \setminus I}, {}^k\,x_k + a^{(1)}_{k}\Big)\bigg)\\
&\hspace{2cm}\bigg(\prod_{I \subseteq [k-1]} \on{Conj}^{k - |I|} \mder_{a^{(2)}, \dots, a^{(r-1)}}f\Big((x + d)_I, x_{[k-1] \setminus I}, {}^k\,x_k + d_k\Big)\bigg)\\
&\hspace{2cm}\bigg(\prod_{I \subseteq [k-1]} \on{Conj}^{k - 1 - |I|} \mder_{a^{(2)}, \dots, a^{(r-1)}}f\Big((x + d)_I, x_{[k-1] \setminus I}, {}^k\,x_k\Big)\bigg)\bigg|.\end{align*}
Apply Cauchy-Schwarz inequality to obtain
\begin{align*}\xi^2 \leq & \exx_{a^{(1)}, \dots, a^{(r-1)} \in G^{\oplus}} \exx_{d_{[k-1]} \in G_{[k-1]}} \Bigg|\exx_{d_k \in G_k} \omega^{\psi(a^{(1)}, \dots, a^{(r-1)}, d_{[k]})}\nonumber\\
&\hspace{2cm}\bigg(\prod_{I \subseteq [k-1]} \on{Conj}^{k - 1 - |I|} \mder_{a^{(2)}, \dots, a^{(r-1)}}f\Big((x + d + a^{(1)})_I, (x + a^{(1)})_{[k-1] \setminus I}, {}^k\,x_k + d_k + a^{(1)}_{k}\Big)\bigg)\nonumber\\
&\hspace{2cm}\bigg(\prod_{I \subseteq [k-1]} \on{Conj}^{k - |I|} \mder_{a^{(2)}, \dots, a^{(r-1)}}f\Big((x + d)_I, x_{[k-1] \setminus I}, {}^k\,x_k + d_k\Big)\bigg)\Bigg|^2\nonumber\\
= & \exx_{a^{(1)}, \dots, a^{(r-1)} \in G^{\oplus}} \exx_{d_{[k-1]} \in G_{[k-1]}}\exx_{d_k, d_k' \in G_k} \omega^{\psi(a^{(1)}, \dots, a^{(r-1)}, d_{[k-1]}, d_k - d_k')}\nonumber\\
&\hspace{2cm}\bigg(\prod_{I \subseteq [k-1]} \on{Conj}^{k - 1 - |I|} \mder_{a^{(2)}, \dots, a^{(r-1)}}f\Big((x + d + a^{(1)})_I, (x + a^{(1)})_{[k-1] \setminus I}, {}^k\,x_k + d_k + a^{(1)}_{k}\Big)\bigg)\nonumber\\
&\hspace{2cm}\bigg(\prod_{I \subseteq [k-1]} \on{Conj}^{k - |I|} \mder_{a^{(2)}, \dots, a^{(r-1)}}f\Big((x + d + a^{(1)})_I, (x + a^{(1)})_{[k-1] \setminus I}, {}^k\,x_k + d'_k + a^{(1)}_{k}\Big)\bigg)\nonumber\\
&\hspace{2cm}\bigg(\prod_{I \subseteq [k-1]} \on{Conj}^{k - |I|} \mder_{a^{(2)}, \dots, a^{(r-1)}}f\Big((x + d)_I, x_{[k-1] \setminus I}, {}^k\,x_k + d_k\Big)\bigg)\nonumber\\
&\hspace{2cm}\bigg(\prod_{I \subseteq [k-1]} \on{Conj}^{k - 1 - |I|} \mder_{a^{(2)}, \dots, a^{(r-1)}}f\Big((x + d)_I, x_{[k-1] \setminus I}, {}^k\,x_k + d'_k\Big)\bigg).\end{align*}

We now make the following change of variables. We introduce a new variable $z_k = x_k + d_k + d'_k + a^{(1)}_{k}$ instead of $a^{(1)}_{k}$. With the new variable, we get inequality
\begin{align}\xi^2 \leq & \exx_{a^{(1)}_{[k-1]} \in G_{[k-1]}} \exx_{a^{(2)}, \dots, a^{(r-1)} \in G^{\oplus}} \exx_{d_{[k-1]} \in G_{[k-1]}} \exx_{d_k, d_k', z_k \in G_k} \omega^{\psi\big((a^{(1)}_{[k-1]}, {}^k\, z_k - d_k -d'_k - x_k), a^{(2)}, \dots, a^{(r-1)}, d_{[k-1]}, d_k - d'_k\big)}\nonumber\\
&\hspace{2cm}\bigg(\prod_{I \subseteq [k-1]} \on{Conj}^{k - 1 - |I|} \mder_{a^{(2)}, \dots, a^{(r-1)}}f\Big((x + d + a^{(1)})_I, (x + a^{(1)})_{[k-1] \setminus I}, {}^k\,z_k - d'_k\Big)\bigg)\nonumber\\
&\hspace{2cm}\bigg(\prod_{I \subseteq [k-1]} \on{Conj}^{k - |I|} \mder_{a^{(2)}, \dots, a^{(r-1)}}f\Big((x + d + a^{(1)})_I, (x + a^{(1)})_{[k-1] \setminus I}, {}^k\, z_k - d_k\Big)\bigg)\nonumber\\
&\hspace{2cm}\bigg(\prod_{I \subseteq [k-1]} \on{Conj}^{k - |I|} \mder_{a^{(2)}, \dots, a^{(r-1)}}f\Big((x + d)_I, x_{[k-1] \setminus I}, {}^k\,x_k + d_k\Big)\bigg)\nonumber\\
&\hspace{2cm}\bigg(\prod_{I \subseteq [k-1]} \on{Conj}^{k - 1 - |I|} \mder_{a^{(2)}, \dots, a^{(r-1)}}f\Big((x + d)_I, x_{[k-1] \setminus I}, {}^k\,x_k + d'_k\Big)\bigg).\label{dkdprimekineq}\end{align}

From multilinearity of $\psi$, we have
\begin{align*}&\psi\Big((a^{(1)}_{[k-1]}, {}^k\, z_k - d_k -d'_k - x_k), a^{(2)}, \dots, a^{(r-1)}, d_{[k-1]}, d_k - d'_k\Big)\\
= &\psi\Big((a^{(1)}_{[k-1]}, {}^k\, z_k - x_k), a^{(2)}, \dots, a^{(r-1)}, d_{[k-1]}, d_k - d'_k\Big)\\
&\hspace{1cm} - \psi\Big((0_{[k-1]}, {}^k\, d_k), a^{(2)}, \dots, a^{(r-1)}, d_{[k-1]}, d_k - d'_k\Big) - \psi\Big((0_{[k-1]}, {}^k\, d'_k), a^{(2)}, \dots, a^{(r-1)}, d_{[k-1]}, d_k - d'_k\Big)\\
=\,&\psi\Big((0_{[k-1]}, {}^k\, d_k), a^{(2)}, \dots, a^{(r-1)}, d_{[k-1]}, d'_k\Big) - \psi\Big((0_{[k-1]}, {}^k\, d'_k), a^{(2)}, \dots, a^{(r-1)}, d_{[k-1]}, d_k\Big)\\
\hspace{1cm}& +\psi\Big((a^{(1)}_{[k-1]}, {}^k\, z_k - x_k), a^{(2)}, \dots, a^{(r-1)}, d_{[k-1]}, d_k\Big) - \psi\Big((a^{(1)}_{[k-1]}, {}^k\, z_k - x_k), a^{(2)}, \dots, a^{(r-1)}, d_{[k-1]}, d'_k\Big)\\
\hspace{1cm}& - \psi\Big((0_{[k-1]}, {}^k\, d_k), a^{(2)}, \dots, a^{(r-1)}, d_{[k-1]}, d_k\Big) + \psi\Big((0_{[k-1]}, {}^k\, d'_k), a^{(2)}, \dots, a^{(r-1)}, d_{[k-1]}, d'_k\Big).
\end{align*}

After applying this identity in~\eqref{dkdprimekineq}, the only terms involving $d_k$ and $d_k'$ remaining are phases of
\[\psi\Big((0_{[k-1]}, {}^k\, d_k), a^{(2)}, \dots, a^{(r-1)}, d_{[k-1]}, d'_k\Big)\]
and
\[\psi\Big((0_{[k-1]}, {}^k\, d'_k), a^{(2)}, \dots, a^{(r-1)}, d_{[k-1]}, d_k\Big).\]

Apply  Gowers-Cauchy-Schwarz inequality (Lemma~\ref{gcs}) with respect to variables $d_k$ and $d'_k$ to deduce
\begin{align*}\xi^2 \leq \exx_{a^{(1)}_{[k-1]} \in G_{[k-1]}} \exx_{a^{(2)}, \dots, a^{(r-1)} \in G^{\oplus}} \exx_{d_{[k-1]} \in G_{[k-1]}} \exx_{z_k \in G_k} \Big\|&\omega^{\psi\big((0_{[k-1]}, {}^k\, d_k), a^{(2)}, \dots, a^{(r-1)}, d_{[k-1]}, d'_k\big)}\\
&\omega^{-\psi\big((0_{[k-1]}, {}^k\, d'_k), a^{(2)}, \dots, a^{(r-1)}, d_{[k-1]}, d_k\big)}\Big\|_{\square(d_k, d'_k)}.\end{align*}
Apply H\"{o}lder's inequality for 4\textsuperscript{th} power to get 
\begin{align*}\xi^8 \leq \exx_{a^{(1)}_{[k-1]} \in G_{[k-1]}} \exx_{a^{(2)}, \dots, a^{(r-1)} \in G^{\oplus}} \exx_{d_{[k-1]} \in G_{[k-1]}} \exx_{z_k, d_k, d'_k \in G_k} &\omega^{\psi\big((0_{[k-1]}, {}^k\, d_k), a^{(2)}, \dots, a^{(r-1)}, d_{[k-1]}, d'_k\big)}\\
&\omega^{-\psi\big((0_{[k-1]}, {}^k\, d'_k), a^{(2)}, \dots, a^{(r-1)}, d_{[k-1]}, d_k\big)},\end{align*}
which is exactly
\[\xi^8 \leq \on{bias} \psi'_{1 k}.\qedhere\]\end{proof}

\vspace{\baselineskip}
\noindent\textbf{From approximate to exact symmetry.} We now prove Proposition~\ref{approxsymtoexact} which allows us to deduce exact symmetry properties.

\begin{proof}[Proof of Proposition~\ref{approxsymtoexact}]Let $i_1, \dots, i_{r-1} \in [k]$ be arbitrary indices. We define further map $\psi_{i_{[r-1]}} \colon G_{i_1} \times G_{i_2} \tdt G_{i_{r-1}} \times G_1 \times G_2 \tdt G_k \to \mathbb{F}_p$ by
\[\psi_{i_{[r-1]}}(b_1, \dots, b_{r-1}, x_1, \dots, x_k) = \psi\Big((0_{[k] \setminus \{i_1\}},\ls{i_1}{b_1}), \dots, (0_{[k] \setminus \{i_{r-1}\}},\ls{i_{r-1}}{b_{r-1}}), x_1, \dots, x_k\Big).\]
Since $\psi$ is a multilinear form, so is $\psi_{i_{[r-1]}}$ for any choice of indices $i_{[r-1]}$. These new maps are related to $\psi$ via the following identity
\[\psi(a^{(1)}, \dots, a^{(r-1)}, x_1, \dots, x_k) = \sum_{i_1, \dots, i_{r-1} \in [k]} \psi_{i_{[r-1]}}(a^{(1)}_{i_1}, \dots, a^{(r-1)}_{i_{r-1}}, x_1, \dots, x_k).\]
It turns out that the approximate symmetry properties of $\psi$ induce approximate symmetry properties of the maps $\psi_{i_{[r-1]}}$. We formulate these properties in the next couple of claims. We write $\on{Sym}_X$ for the group of permutations of a finite set $X$.

\begin{claim}\label{psiseqsym1}Let $i_1, \dots, i_{r-1} \in [k]$ and let $\sigma \in \on{Sym}_{[r-1]}$. Let $\sigma \overset{\bm{1}}{\circ}\psi_{i_{[r-1]}} \colon G_{i_1} \tdt G_{i_{r-1}} \times G_{[k]} \to \mathbb{F}_p$ be a multilinear form defined by
\[\sigma \overset{\bm{1}}{\circ}\psi_{i_{[r-1]}}(b_1, \dots, b_{r-1}, x_1, \dots, x_k) = \psi_{i_{\sigma(1)}, \dots, i_{\sigma(r-1)}}(b_{\sigma(1)}, \dots, b_{\sigma(r-1)}, x_1, \dots, x_k).\]
Then
\[\bias \Big(\psi_{i_{[r-1]}} - \sigma \overset{\bm{1}}{\circ}\psi_{i_{[r-1]}}\Big) \geq \xi^{r 2^{ k + r - 1}}.\]
\end{claim}

\begin{proof}Let $c_1$ and $c_2$ be two distinct elements of $[r-1]$. Using~\eqref{approxtoexsym1} for indices $c_1$ and $c_2$ and expanding the definition of $\psi_{c_1\,c_2}$ we get
\begin{align*}\xi \leq& \bias \psi_{c_1\, c_2} = \exx_{\substack{a^{(1)}, \dots, a^{(r-1)} \in G^\oplus\\x_{[k]} \in G_{[k]}}} \omega^{\psi_{c_1\, c_2}(a^{(1)}, \dots, a^{(r-1)}, x_{[k]})}\\
= &\exx_{\substack{a_1, \dots, a_{r-1} \in G^\oplus\\x_{[k]} \in G_{[k]}}} \prod_{j_1, \dots, j_{r-1} \in [k]} \omega^{\psi_{j_{[r-1]}}(a_{1, j_1}, \dots, a_{r-1, j_{r-1}}, x_{[k]})}\omega^{- \psi_{j_{[r-1]}}(a^{(1)}_{j_1}, \dots, a^{(c_2)}_{j_{c_1}}, \dots, a^{(c_1)}_{j_{c_2}}, \dots, a^{(r-1)}_{j_{r-1}}, x_{[k]})}.\end{align*}
Among all forms that appear in the last line of the expression above, the variables $a^{(1)}_{i_1}, \dots, a^{(r-1)}_{i_{r-1}}, x_1, \dots, x_k$ appear in exactly two, namely
\[\psi_{i_{[r-1]}}(a^{(1)}_{i_1}, \dots, a^{(r-1)}_{i_{r-1}}, x_{[k]})\]
and
\[-\psi_{i_{[c_1-1]}, i_{c_2}, i_{[c_1 + 1, c_2 - 1]}, i_{c_1}, i_{[c_2 + 1, r-1]}}(a^{(1)}_{i_1}, \dots, a^{(c_2)}_{i_{c_2}}, \dots, a^{(c_1)}_{i_{c_1}}, \dots, a^{(r-1)}_{i_{r-1}}, x_{[k]}).\]
Writing $\sigma$ for the transposition that exchanges $c_1$ and $c_2$, we may see that the latter form equals
\[\sigma \overset{\bm{1}}{\circ}\psi_{i_{[r-1]}}(a^{(1)}_{i_1}, \dots, a^{(r-1)}_{i_{r-1}}, x_{[k]}).\]
Using Lemma~\ref{gcs}, we see that
\[\xi \leq \Big\|\omega^{\psi_{i_{[r-1]}} - \sigma \overset{\bm{1}}{\circ}\psi_{i_{[r-1]}}}\Big\|_{\square^{r + k - 1}} = \Big(\bias(\psi_{i_{[r-1]}} - \sigma \overset{\bm{1}}{\circ}\psi_{i_{[r-1]}})\Big)^{2^{-(r + k -1)}}.\]
Recall that the symmetric group $\on{Sym}_{[r-1]}$ is generated by transpositions, and that we may in fact write any permutation as a composition of at most $r-1$ transpositions. By induction on the least length $\ell$ of a product of transpositions that gives a permutation $\sigma$ we prove that
\[\xi^{\ell 2^{r + k -1}} \leq \bias(\psi_{i_{[r-1]}} - \sigma \overset{\bm{1}}{\circ}\psi_{i_{[r-1]}}).\]
The base case has already been proved. Write $\sigma = \tau \circ \sigma'$ for a transposition $\tau$ and a permutation $\sigma'$ that can be written as a product of $\ell-1$ transpositions. By inductive hypothesis for a slightly different sequence of indices $i_{\tau(1)}, \dots, i_{\tau(r-1)}$ we have
\[\xi^{(\ell - 1)2^{r + k -1}} \leq \bias(\psi_{i_{\tau(1)}, \dots, i_{\tau(r-1)}} - \sigma' \overset{\bm{1}}{\circ}\psi_{i_{\tau(1)}, \dots, i_{\tau(r-1)}}).\]
Apply $\tau$ to both forms to obtain
\[\xi^{(\ell - 1) 2^{r + k -1}} \leq \bias(\psi_{i_{\tau(1)}, \dots, i_{\tau(r-1)}} - \sigma' \overset{\bm{1}}{\circ}\psi_{i_{\tau(1)}, \dots, i_{\tau(r-1)}}) = \bias(\tau \overset{\bm{1}}{\circ}\psi_{i_{[r-1]}} - \sigma \overset{\bm{1}}{\circ}\psi_{i_{[r-1]}}).\]
An appeal to the base case yields 
\[\xi^{2^{r + k -1}} \leq \bias(\psi_{i_{[r-1]}} - \tau \overset{\bm{1}}{\circ}\psi_{i_{[r-1]}}).\]
Finally, Lemma~\ref{subBias} gives
\[\xi^{\ell 2^{r + k -1}} \leq \bias(\psi_{i_{[r-1]}} - \sigma \overset{\bm{1}}{\circ}\psi_{i_{[r-1]}}).\qedhere\]
\end{proof}

\begin{claim}\label{psiseqsym2}Let $i_1, \dots, i_{r-1} \in [k]$. Extend the sequence by defining $i_r = 1, i_{r+1} = 2, \dots, i_{k + r - 1} = k$. Let $\sigma \in \on{Sym}_{[k + r - 1]}$ be a permutation such that $i_{\sigma(j)} = i_j$ for all $j \in [k + r -1]$. We denote the group of such permutations by $\on{Sym}(i_{[r-1]})$. Define the multilinear form $\sigma \overset{\bm{2}}{\circ}  \psi_{i_{[r-1]}} \colon G_{i_1} \tdt G_{i_{r-1}} \times G_1 \tdt G_k \to \mathbb{F}_p$ by
\[\sigma \overset{2}{\circ}  \psi_{i_{[r-1]}}(y_1, \dots, y_{k + r - 1}) = \psi_{i_{[r-1]}}(y_{\sigma(1)}, \dots, y_{\sigma(k + r - 1)}).\]
Then
\[\bias \Big(\psi_{i_{[r-1]}} - \sigma \overset{\bm{2}}{\circ}  \psi_{i_{[r-1]}}\Big) \geq \xi^{3(k + r) 2^{r + k -1}}.\]
\end{claim}

\begin{proof}Let $c_1 \in [r-1], c_2 \in [k, k + r - 1]$ be two indices such that $i_{c_1} = i_{c_2}$. (This actually means that $i_{c_2} = c_2 - k + 1$.) To simplify notation, we denote $i_{c_1}$ by $\ell$. Using~\eqref{approxtoexsym2} for indices $c_1$ and $\ell$ and expanding the definition of $\psi'_{c_1\,\ell}$ we get
\begin{align*}\xi \leq & \on{bias}\psi'_{c_1\,i_{c_2}} = \exx_{\substack{a^{(1)}, \dots, a^{(c_1 - 1)}, a^{(c_1 + 1)}, \dots, a^{(r-1)} \in G^\oplus\\x_{[k]} \in G_{[k]}, y_\ell \in G_\ell}} \omega^{\psi'_{c_1\, \ell}(a^{(1)}, \dots, a^{(c_1 - 1)}, y_\ell, a^{(c_1 + 1)}, \dots, a^{(r-1)}, x_{[k]})}\\
= &\exx_{\substack{a^{(1)}, \dots, a^{(c_1 - 1)},\\ a^{(c_1 + 1)}, \dots, a^{(r-1)} \in G^\oplus\\x_{[k]} \in G_{[k]}, y_\ell \in G_\ell}} \prod_{\substack{j_1, \dots, j_{c_1 - 1},\\j_{c_1 + 1}, \dots, c_{r-1} \in [k]}} \omega^{\psi_{j_1, \dots, j_{c_1-1}, \ell, j_{c_1 + 1}, \dots, j_{r-1}}(a^{(1)}_{j_1}, \dots, a^{(c_1 - 1)}_{j_{c_1 - 1}}, y_\ell, a^{(c_1 + 1)}_{j_{c_1 + 1}}, \dots, a^{(r-1)}_{j_{r-1}}, x_{[k] \setminus \{\ell\}}, {}^\ell\,x_\ell)}\\
&\hspace{5cm}\omega^{-\psi_{j_1, \dots, j_{c_1-1}, \ell, j_{c_1 + 1}, \dots, j_{r-1}}(a^{(1)}_{j_1}, \dots, a^{(c_1 - 1)}_{j_{c_1 - 1}}, x_\ell, a^{(c_1 + 1)}_{j_{c_1 + 1}}, \dots, a^{(r-1)}_{j_{r-1}}, x_{[k] \setminus \{\ell\}}, {}^\ell\,y_\ell)}\end{align*}
The only forms above that have the variables $a^{(1)}_{j_1}, \dots, a^{(c_1 - 1)}_{j_{c_1 - 1}},$ $a^{(c_1 + 1)}_{j_{c_1 + 1}}, \dots, a^{(r-1)}_{j_{r-1}},$ $x_1, \dots, x_k, y_\ell$ appearing occur precisely when $j_{[r-1] \setminus \{c_1\}} = i_{[r-1] \setminus \{c_1\}}$. Write $\sigma \in \on{Sym}(i_{[r-1]})$ for the transposition that swaps $c_1$ and $c_2$. Lemma~\ref{gcs} then gives
\[\xi^{2^{k+r-1}} \leq \bias \Big(\psi_{i_{[r-1]}} - \sigma \overset{\bm{2}}{\circ}  \psi_{i_{[r-1]}}\Big).\]
The group of permutations $\on{Sym}(i_{[r-1]})$ can be seen as a product of symmetric groups $\on{Sym}_{I_1} \tdt \on{Sym}_{I_k}$ where $I_s = \{j \in [k + r - 1] \colon i_j = s\}$. Using the fact that for a fixed element $y_0 \in [m]$ every permutation in $\on{Sym}_m$ can be written as a product of at most $3m$ transpositions of the form $(x \,\,y_0)$ for $x \in [m] \setminus \{y_0\}$,\footnote{Notice that we may obtain any transposition $(a\,\,b)$ as $(a\,\,y_0) (b\,\,y_0) (a\,\,y_0)$.} the claim follows after a short inductive argument closely resembling the one in Claim~\ref{psiseqsym1}.\end{proof}

Finally, define a multinear form $\rho \colon (G^{\oplus})^{r-1} \times G_{[k]} \to \mathbb{F}_p$ by
\begin{equation}\label{}\rho(a^{(1)}, \dots, a^{(r-1)}, x_{[k]}) = \sum_{i_1, \dots, i_{r-1} \in [k]} \Big(\exx_{\sigma \in \on{Sym}_{[r-1]}} \exx_{\tau \in \on{Sym}(i_{[r-1]})} \tau  \overset{\bm{2}}{\circ} \psi_{i_{[r-1]}}(a^{(\sigma(1))}_{i_1}, \dots, a^{(\sigma(r-1))}_{i_{r-1}}, x_{[k]})\Big).\label{rhodefineqn}\end{equation}
Before proceeding further, let us stress that this is the place in the proof where we use the assumption on the characteristic of the field. The expectation over the permutations in the expression above has the usual meaning, namely $\ex_{\sigma \in \on{Sym}_{[r-1]}}$ is simply a shorthand for $\frac{1}{|\on{Sym}_{[r-1]}|} \sum_{\sigma \in \on{Sym}_{[r-1]}}$ and similarly $\ex_{\tau \in \on{Sym}(i_{[r-1]})}$ stands for $\frac{1}{|\on{Sym}(i_{[r-1]})|}\sum_{\tau \in \on{Sym}(i_{[r-1]})}$. This means that the expression above has a factor
\[\frac{1}{| \on{Sym}_{[r-1]}|} \frac{1}{|\on{Sym}(i_{[r-1]})|},\]
which equals 
\[\frac{1}{(r-1)! q_1! \dots q_k!},\]
where $q_s = |\{j \in [k + r - 1] \colon i_j = s\}|$. Recall that $i_r = 1, i_{r+1} = 2, \dots, i_{r + k - 1} = k$, so we have $q_s \leq r$ for each $s$. Thus, as long as $p \geq r + 1$ we may invert the element $(r-1)! q_1! \dots q_k!$ in $\mathbb{F}_p$.\\

Let us prove that $\rho$ has the desired symmetry properties. First, we show that $\rho_{c_1\,c_2} = 0$ for $c_1, c_2 \in [r-1]$. Let $\sigma_0 \in \on{Sym}_{[r-1]}$ be the transposition that swaps $c_1$ and $c_2$. We have
\begin{align*}&\rho(a^{(1)}, \dots, a^{(c_1 - 1)}, a^{(c_2)}, a^{(c_1 + 1)}, \dots, a^{(c_2 -1)}, a^{(c_1)}, a^{(c_2 + 1)}, \dots a^{(r-1)}, x_{[k]})\\
&\hspace{2cm} = \sum_{i_1, \dots, i_{r-1} \in [k]} \Big(\exx_{\sigma \in \on{Sym}_{[r-1]}} \exx_{\tau \in \on{Sym}(i_{[r-1]})} \tau  \overset{\bm{2}}{\circ} \psi_{i_{[r-1]}}(a^{(\sigma_0\circ \sigma(1))}_{i_1}, \dots, a^{(\sigma_0 \circ \sigma (r-1))}_{i_{r-1}}, x_{[k]})\Big)\\
&\hspace{2cm} = \sum_{i_1, \dots, i_{r-1} \in [k]} \Big(\exx_{\sigma \in \on{Sym}_{[r-1]}} \exx_{\tau \in \on{Sym}(i_{[r-1]})} \tau  \overset{\bm{2}}{\circ} \psi_{i_{[r-1]}}(a^{(\sigma (1))}_{i_1}, \dots, a^{(\sigma (r-1))}_{i_{r-1}}, x_{[k]})\Big)\\
&\hspace{2cm} = \rho(a^{(1)}, \dots, a^{(r-1)}, x_{[k]}),\end{align*}
as desired.\\
\indent Secondly, we show that $\rho'_{c\,\ell} = 0$ for $c \in [r-1], \ell \in [k]$. We have
\begin{align*}&\rho'_{c\,\ell}(a^{(1)}, \dots, a^{(c -1)}, y_\ell, a^{(c+1)}, \dots, a^{(r-1)}, x_1, \dots, x_k)\\
=&\rho\Big(a^{(1)}, \dots, a^{(c -1)}, (0_{[k] \setminus \{\ell\}},\ls{\ell}\,y_\ell), a^{(c+1)}, \dots, a^{(r-1)}, x_1, \dots, x_{\ell-1}, x_\ell, x_{\ell+1}, \dots, x_k\Big)\nonumber\\
&\hspace{2cm}-\rho\Big(a^{(1)}, \dots, a^{(c -1)}, (0_{[k] \setminus \{\ell\}},\ls{\ell}\,x_\ell), a^{(c+1)}, \dots, a^{(r-1)}, x_1, \dots, x_{\ell-1}, y_\ell, x_{\ell+1}, \dots, x_k\Big)\\
=&\sum_{i_1, \dots, i_{r-1} \in [k]} \Big(\exx_{\sigma \in \on{Sym}_{[r-1]}} \exx_{\tau \in \on{Sym}(i_{[r-1]})} \mathbbm{1}(i_{\sigma^{-1}(c)} = \ell)\, \tau  \overset{\bm{2}}{\circ} \psi_{i_{[r-1]}}\Big((a^{(\sigma(s))}_{i_s})_{s \in [r-1] \setminus \{\sigma^{-1}(c)\}}, \ls{\sigma^{-1}(c)}{y_\ell}; x_{[k] \setminus \{\ell\}}, {}^\ell\,x_\ell\Big)\Big)\\
&\hspace{1cm}-\sum_{i_1, \dots, i_{r-1} \in [k]} \Big(\exx_{\sigma \in \on{Sym}_{[r-1]}} \exx_{\tau \in \on{Sym}(i_{[r-1]})} \mathbbm{1}(i_{\sigma^{-1}(c)} = \ell)\\
&\hspace{8cm}\tau  \overset{\bm{2}}{\circ} \psi_{i_{[r-1]}}\Big((a^{(\sigma(s))}_{i_s})_{s \in [r-1] \setminus \{\sigma^{-1}(c)\}}, \ls{\sigma^{-1}(c)}{x_\ell}; x_{[k] \setminus \{\ell\}}, {}^\ell\,y_\ell\Big)\Big).\end{align*}
For $\sigma \in \on{Sym}_{[r-1]}$ define a transposition $\tau^{\sigma} = (\sigma^{-1}(c)\,\,r - 1 + \ell)$. Note that when $i_\sigma^{-1}(c) = \ell$ then $\tau^{\sigma} \in \on{Sym}(i_{[r-1]})$. Hence the expression above can be rewritten as
\begin{align*}&\sum_{i_1, \dots, i_{r-1} \in [k]} \Big(\exx_{\sigma \in \on{Sym}_{[r-1]}} \exx_{\tau \in \on{Sym}(i_{[r-1]})} \mathbbm{1}(i_{\sigma^{-1}(c)} = \ell)\, \tau  \overset{\bm{2}}{\circ} \psi_{i_{[r-1]}}\Big((a^{(\sigma(s))}_{i_s})_{s \in [r-1] \setminus \{\sigma^{-1}(c)\}}, \ls{\sigma^{-1}(c)}{y_\ell}; x_{[k] \setminus \{\ell\}}, {}^\ell\,x_\ell\Big)\Big)\\
&\hspace{1cm}-\sum_{i_1, \dots, i_{r-1} \in [k]} \Big(\exx_{\sigma \in \on{Sym}_{[r-1]}} \exx_{\tau \in \on{Sym}(i_{[r-1]})} \mathbbm{1}(i_{\sigma^{-1}(c)} = \ell)\\
&\hspace{8cm}(\tau^{\sigma} \circ \tau)  \overset{\bm{2}}{\circ} \psi_{i_{[r-1]}}\Big((a^{(\sigma(s))}_{i_s})_{s \in [r-1] \setminus \{\sigma^{-1}(c)\}}, \ls{\sigma^{-1}(c)}{y_\ell}; x_{[k] \setminus \{\ell\}}, {}^\ell\,x_\ell\Big)\Big)\\
&\hspace{1cm}=0,\end{align*}
as desired, where in the last line we used the fact that $\on{Sym}(i_{[r-1]})$ is a group.\\

To finish the proof we need to show that the bias of the difference of the forms $\rho - \psi$ is large. Fix some indices $j_1, \dots, j_{r-1} \in [k]$. We shall now determine which forms on the right hand side of~\eqref{rhodefineqn} have all of the variables $a^{(1)}_{j_1}, \dots, a^{(r-1)}_{j_{r-1}}, x_{[k]}$. For $i_1, \dots, i_{r-1}$ and $\sigma$ in~\eqref{rhodefineqn} we have variables $a^{(\sigma(1))}_{i_1}, \dots, a^{(\sigma(r-1))}_{i_{r-1}}$, so we must have $j_{\sigma(s)} = i_s$ for every $s \in [r-1]$. Therefore, given a permutation $\sigma \in \on{Sym}_{[r-1]}$, the variables $a^{(1)}_{j_1}, \dots, a^{(r-1)}_{j_{r-1}}, x_{[k]}$ appear for the sequence $i_s = j_{\sigma(s)}$. Fix $\sigma \in \on{Sym}_{[r-1]}$ and the corresponding $i_{[r-1]}$. 

\begin{claim}The bias of the map
\begin{equation}\label{differenceformsigmai}(a^{(1)}_{j_1}, \dots, a^{(r-1)}_{j_{r-1}}, x_{[k]}) \mapsto \psi_{j_{[r-1]}}(a^{(1)}_{j_1}, \dots, a^{(r-1)}_{j_{r-1}}, x_{[k]}) - \tau  \overset{\bm{2}}{\circ} \psi_{i_{[r-1]}}(a^{(\sigma(1))}_{i_1}, \dots, a^{(\sigma(r-1))}_{i_{r-1}}, x_{[k]})\end{equation}
is at least $\xi^{4(k + r) 2^{r + k -1}}$.\label{claimapproxsymmmaps}\end{claim}

By Lemma~\ref{subBias} it then follows that
\[\bias \Big(\rho - \psi \Big) \geq \xi^{4(k + r) k^{r-1} (k + r -1)! (r-1)! 2^{r + k -1}}.\]

\begin{proof}[Proof of Claim~\ref{claimapproxsymmmaps}]By Claim~\ref{psiseqsym2}, we have that the map
\[\psi_{i_{[r-1]}}(a^{(\sigma(1))}_{i_1}, \dots, a^{(\sigma(r-1))}_{i_{r-1}}, x_{[k]}) - \tau  \overset{\bm{2}}{\circ} \psi_{i_{[r-1]}}(a^{(\sigma(1))}_{, i_1}, \dots, a^{(\sigma(r-1))}_{i_{r-1}}, x_{[k]})\]
has bias at least $\xi^{3(k + r) 2^{r + k -1}}$. Note that
\begin{align*}\psi_{i_{[r-1]}}(a^{(\sigma(1))}_{i_1}, \dots, a^{(\sigma(r-1))}_{i_{r-1}}, x_{[k]}) = &\psi_{j_{\sigma(1)}, \dots, j_{\sigma(r-1)}}(a^{(\sigma(1))}_{j_{\sigma(1)}}, \dots, a^{(\sigma(r-1))}_{j_{\sigma(r-1)}}, x_{[k]})\\
=&\sigma \overset{\bm{1}}{\circ} \psi_{j_{[r-1]}}(a^{(1)}_{j_1}, \dots, a^{(r-1)}_{j_{r-1}}, x_{[k]}).\end{align*}
By Claim~\ref{psiseqsym1}, we have the bound
\[\bias\Big(\psi_{j_{[r-1]}} - \sigma \overset{\bm{1}}{\circ} \psi_{j_{[r-1]}}\Big) \geq \xi^{r 2^{ k + r - 1}}.\]
Using Lemma~\ref{subBias}, we at last conclude that the bias of the map in~\eqref{differenceformsigmai} is at least $\xi^{4(k + r) 2^{r + k -1}}$.\end{proof}

Having proved Claim~\ref{claimapproxsymmmaps}, the proof of Proposition~\ref{approxsymtoexact} is also complete.\end{proof}

\vspace{\baselineskip}
\boldSubsection{Partially symmetric multlinear forms}

In this subsection, we describe the multilinear forms which have the (exact) symmetry properties in Proposition~\ref{approxsymtoexact}. We view each $G_j$ as $\mathbb{F}_p^{n_j}$. Thus, when $x_j \in G_j$, we have its further coordinates $x_{j, c}$ for $c \in [n_j]$. This further leads to coordinates of $x \in G^\oplus$ denoted by $x_{j, c}$ for $j \in [k], c \in  [n_j]$. Let us define $\mathcal{P}_{k, r}$ to be the set of polynomials of degree at most $k + r - 1$ on $G^{\oplus}$, where the monomials $x_{d_1, c_1} \cdots x_{d_{k+r-1}, c_{k+r-1}}$ are required to have every $i \in [k]$ present among $d_1, \dots, d_{k + r - 1}$. It turns out that derivatives of polynomials in $\mathcal{P}_{k, r}$ are the essentially the only source of the partially symmetric multilinear forms of interest.

\begin{proposition}\label{exactderivpoly}Assume that $p \geq r + 1$. Let $\psi \colon (G^{\oplus})^{r-1} \times G_{[k]} \to \mathbb{F}_p$ be a multilinear form such that
\begin{itemize}
\item[\textbf{(i)}] the maps $\psi_{ij}$ defined by~\eqref{psidefeqn} are all zero, and
\item[\textbf{(ii)}] the maps $\psi'_{ij}$ defined by~\eqref{psidefeqn2} are all zero.
\end{itemize}
Then there exist polynomials $P \in \mathcal{P}_{k, r}$ and $Q$ such that for all $a^{(1)}, \dots, a^{(r-1)} \in G^\oplus$ and $x_{[k]} \in G_{[k]}$ we have
\[\psi(a^{(1)}, \dots, a^{(r-1)}, x_1, \dots, x_k) = \Delta_{a^{(1)}, \dots, a^{(r-1)}} P(x_{[k]}) + Q(a^{(1)}, a^{(2)}, \dots, a^{(r-1)}, x_1, \dots, x_k)\]
and for each monomial $m$ appearing in $Q$ there is $i \in [k]$ such that no variable $x_{ic}$ appears in $m$.
\end{proposition}

(Recall the additive derivative notation is defined by $\Delta_h F(x) = F(x + h) - F(x)$.)

\begin{proof}Let $\mathcal{M}$ be the set of triples $(d,c,c')$ where $d = d_{[r-1]}$ is a sequence of elements in $[k]$, $c = c_{[r-1]}$ is a sequence such that $c_i \in [n_{d_i}]$ and $c' = c'_{[k]}$ is a sequence such that $c'_j \in [n_j]$. Since $\psi$ is a multilinear form, there are coefficients $\lambda_{d, c, c'} \in \mathbb{F}_p$ for triples of sequences $(d,c,c') \in \mathcal{M}$ such that
\[\psi(a^{(1)}, \dots, a^{(r-1)}, x_1, \dots, x_k) = \sum_{(d,c,c') \in \mathcal{M}} \lambda_{d, c, c'}\, a^{(1)}_{d_1, c_1} \cdots a^{(r - 1)}_{d_{r-1}, c_{r-1}} x_{1, c'_1} \cdots x_{k, c'_k}.\]
We shall use the symmetry properties of $\psi$ to conclude equalities between some of the coefficients $\lambda_{d,c,c'}$.

\begin{claim}\label{samelambdareorder}Let $(d,c,c'), (f,e,e') \in \mathcal{M}$ be two triples such that
\[(d_1, c_1), \dots, (d_{r-1}, c_{r-1}), (1, c'_1), \dots, (k, c'_k)\]
and
\[(f_1, e_1), \dots, (f_{r-1}, e_{r-1}), (1, e'_1), \dots, (k, e'_k)\]
are the same sequence up to reordering. Then $\lambda_{d,c,c'} = \lambda_{f,e,e'}$.\end{claim}

\begin{proof}We first prove that $\lambda_{d,c,c'} = \lambda_{f,e,e'}$ when $c' = e'$. This is equivalent to $(d_1, c_1), \dots, (d_{r-1}, c_{r-1})$ and $(f_1, e_1), \dots, (f_{r-1}, e_{r-1})$ being the same up to reordering, so there is a permutation $\pi \in \on{Sym}_{[r-1]}$ such that $(f_i, e_i) = (d_{\pi(i)}, c_{\pi(i)})$ for each $i \in [r-1]$, i.e.\ $f = d \circ \pi, e = c \circ \pi$. The facts that $(d \circ \sigma, c \circ \sigma, c') \in \mathcal{M}$ for every $\sigma \in \on{Sym}_{[r-1]}$ and that $\on{Sym}_{[r-1]}$ is generated by transpositions allow us to assume that $\pi$ is itself a transposition. Without loss of generality $\pi$ swaps 1 and 2. Using property \textbf{(i)} for coordinates 1 and 2, we see that
\begin{align*}&\sum_{(\tilde{d},\tilde{c},\tilde{c}') \in \mathcal{M}} \lambda_{\tilde{d},\tilde{c},\tilde{c}'}\, a^{(1)}_{\tilde{d}_1, \tilde{c}_1}\, a^{(2)}_{\tilde{d}_2, \tilde{c}_2} a^{(3)}_{\tilde{d}_3, \tilde{c}_3} \cdots a^{(r-1)}_{\tilde{d}_{r-1}, \tilde{c}_{r-1}} x_{1, \tilde{c}'_1} \cdots x_{k, \tilde{c}'_k}\\
&\hspace{2cm}=\psi(a^{(1)}, a^{(2)}, a^{(3)}, \dots, a^{(r-1)}, x_1, \dots, x_k)\\
&\hspace{2cm}=\psi(a^{(2)}, a^{(1)}, a^{(3)}, \dots, a^{(r-1)}, x_1, \dots, x_k)\\
 =&\sum_{(\tilde{d},\tilde{c},\tilde{c}') \in \mathcal{M}} \lambda_{\tilde{d},\tilde{c},\tilde{c}'}\, a^{(2)}_{\tilde{d}_1, \tilde{c}_1}\, a^{(1)}_{\tilde{d}_2, \tilde{c}_2} a^{(3)}_{\tilde{d}_3, \tilde{c}_3} \cdots a^{(r-1)}_{\tilde{d}_{r-1}, \tilde{c}_{r-1}} x_{1, \tilde{c}'_1} \cdots x_{k, \tilde{c}'_k}\\
 =&\sum_{(\tilde{d},\tilde{c},\tilde{c}') \in \mathcal{M}} \lambda_{\tilde{d} \circ \pi,\tilde{c} \circ \pi,c'}\, a^{(1)}_{\tilde{d}_1, \tilde{c}_1} a^{(2)}_{\tilde{d}_2,\tilde{c}_2} a^{(3)}_{\tilde{d}_3, \tilde{c}_3} \cdots a^{(r-1)}_{\tilde{d}_{r-1},\tilde{c}_{r-1}} x_{1, \tilde{c}'_1} \cdots x_{k, \tilde{c}'_k},
\end{align*}
which gives $\lambda_{d \circ \pi,c \circ \pi,c'} = \lambda_{d,c,c'}$, as desired.\\

Now we consider the case when $c'$ and $e'$ need not be equal. Assume first that $c'_1 \not= e'_1$, but $c'_2 = e'_2, \dots, c'_k = e'_k$. The first part of the proof allows us to reorder $(d,c)$ and $(f,e)$ so without loss of generality we have $d_1 = \dots = d_l = 1 \not= d_{l+1}, \dots, d_{r-1}$ and $f_1 = \dots = f_{l'} = 1 \not= f_{l'+1}, \dots, f_{r-1}$. Additionally, by further reordering, we may assume that $c_1 = e'_1$. By hypothesis, we have $l = l'$ and the sequences $(c_1, \dots, c_l, c'_1)$ and $(e_1, \dots, e_l, e'_1)$ are the same, up to reordering. Reordering further if necessary, we may assume that $c_2 = e_2, \dots, c_l = e_l$. Now use property \textbf{(ii)} to see that
\begin{align*}&\sum_{(\tilde{d},\tilde{c},\tilde{c}') \in \mathcal{M}} \lambda_{\tilde{d},\tilde{c},\tilde{c}'}\, \mathbbm{1}(\tilde{d}_1 = 1)\, u_{\tilde{c}_1} a^{(2)}_{\tilde{d}_2, \tilde{c}_2} \cdots a^{(r-1)}_{\tilde{d}_{r-1}, \tilde{c}_{r-1}} v_{\tilde{c}'_1} x_{2, \tilde{c}'_2} \cdots x_{k, \tilde{c}'_k}\\
&\hspace{2cm}=\psi\Big((0_{[2, k]},\ls{1}\, u), a^{(2)}, \dots, a^{(r-1)}, v, x_2, \dots, x_k\Big)\\
&\hspace{2cm}=\psi\Big((0_{[2, k]},\ls{1}\, v), a^{(2)}, \dots, a^{(r-1)}, u, x_2, \dots, x_k\Big)\\
= & \sum_{(\tilde{d},\tilde{c},\tilde{c}') \in \mathcal{M}} \lambda_{\tilde{d},\tilde{c},\tilde{c}'}\, \mathbbm{1}(\tilde{d}_1 = 1)\, v_{\tilde{c}_1} a^{(2)}_{\tilde{d}_2, \tilde{c}_2} \cdots a^{(r-1)}_{\tilde{d}_{r-1}, \tilde{c}_{r-1}} u_{\tilde{c}'_1} x_{2, \tilde{c}'_2} \cdots x_{k, \tilde{c}'_k}\\
= &\sum_{(\tilde{d},\tilde{c},\tilde{c}') \in \mathcal{M}} \lambda_{\tilde{d},(\tilde{c}'_1, \tilde{c}_{[2,r-1]}), (\tilde{c}_1, \tilde{c}'_{[2,k]})}\, \mathbbm{1}(\tilde{d}_1 = 1)\, u_{\tilde{c}_1} a^{(2)}_{\tilde{d}_2, \tilde{c}_2} \cdots a^{(r-1)}_{\tilde{d}_{r-1}, \tilde{c}_{r-1}} v_{\tilde{c}'_1} x_{2, \tilde{c}'_2} \cdots x_{k, \tilde{c}'_k},\end{align*}
which in particular gives that $\lambda_{d,c,c'} = \lambda_{f,e,e'}$ for the triples $(d,c,c')$ and $(f,e,e')$ above.\\

Finally, consider the general case. By induction on $\ell \in [0, k]$ we show that there is a triple $(s^{(\ell)}, t^{(\ell)}, u^{(\ell)}) \in \mathcal{M}$ such that $(d_1, c_1), \dots,$ $(d_{r-1}, c_{r-1}),$ $(1, c'_1), \dots,$ $(k, c'_k)$ and $(s^{(\ell)}_1, t^{(\ell)}_1), \dots,$ $(s^{(\ell)}_{r-1}, t^{(\ell)}_{r-1}),$ $(1, u^{(\ell)}_1), \dots,$ $(k, u^{(\ell)}_k)$ are the same up to reordering and that $u^{(\ell)}_1 = e'_1, \dots,$ $u^{(\ell)}_\ell = e'_\ell,$ $u^{(\ell)}_{\ell + 1} = c'_{\ell + 1}, \dots,$ $u^{(\ell)}_k = c'_k$. For the base case $\ell = 0$, we take the triple $(d,c,c')$. Assume now that we have constructed the triple $(s^{(\ell)}, t^{(\ell)}, u^{(\ell)})$ for some $\ell < k$. If $e'_{\ell + 1} = c'_{\ell + 1}$, we may take $s^{(\ell + 1)} = s^{(\ell)}$, $t^{(\ell + 1)} = t^{(\ell)}$ and $u^{(\ell + 1)} = u^{(\ell)}$. Thus assume that $e'_{\ell + 1} \not= c'_{\ell + 1}$. Since $(\ell + 1, e'_{\ell + 1}) \not= (\ell + 1, u^{(\ell)}_{\ell + 1})$ and the sequences $(f_1, e_1), \dots, (k, e'_k)$ and $(s^{(\ell)}_1, t^{(\ell)}_1), \dots, (k, u^{(\ell)}_k)$ are the same up to reordering we have some $i_0 \in [r-1]$ such that $(\ell + 1, e'_{\ell + 1}) = (s^{(\ell)}_{i_0}, t^{(\ell)}_{i_0})$. Now set $s^{(\ell + 1)} = s^{(\ell)}$, $t^{(\ell + 1)}_j = t^{(\ell)}_j$ for $j \not= i_0$ and $t^{(\ell + 1)}_{i_0} = c'_{\ell + 1}$ (recall that $u^{(\ell + 1)}$ is already specified) to complete the induction step.\\
\indent Using the triples we have just constructed and the work above, we have
\[\lambda_{d,c,c'} = \lambda_{s^{(0)}, t^{(0)}, u^{(0)}} = \dots = \lambda_{s^{(k)}, t^{(k)}, u^{(k)}} = \lambda_{f,e,e'}.\qedhere\]
\end{proof}
\vspace{\baselineskip}
For each triple $(d,c,c') \in \mathcal{M}$ define $\on{mon}(d,c,c')$ to be the sequence $(d_1, c_1), \dots,$ $(d_{r-1}, c_{r-1}),$ $(1, c'_1), \dots,$ $(k, c'_k)$ sorted in lexicographic order. Note that the condition in Claim~\ref{samelambdareorder} about resulting sequences being same up to reordering can be expressed as $\on{mon}(d,c,c') = \on{mon}(f, e, e')$. Let $\on{Mon}$ be the set of all images $\on{mon}(d,c,c')$ when $(d,c,c')$ ranges over $\mathcal{M}$. We denote elements of $\on{Mon}$ as $(\tilde{d}, \tilde{c})$ which stands for the sequence $\Big((\tilde{d}_1, \tilde{c}_1), \dots, (\tilde{d}_{k + r-1}, \tilde{c}_{k + r -1})\Big)$. Let $s$ be the number of different pairs that appear in this sequence, and let $v_1, \dots, v_s$ be their number of appearances. Note that $v_i \leq r$. Define $\tilde{\lambda}_{\tilde{d}, \tilde{c}}$ for $(\tilde{d}, \tilde{c}) \in \on{Mon}$ as
\[\tilde{\lambda}_{\tilde{d}, \tilde{c}} = \Big(\prod_{i = 1}^s v_i!\Big)^{-1}\lambda_{d,c,c'}\]
for arbitrary $(d,c,c') \in \mathcal{M}$ such that $\on{mon}(d,c,c') = (\tilde{d}, \tilde{c})$. Claim~\ref{samelambdareorder} tells us that this is well-defined. This is the place in the proof where we use the assumption that $p \geq r +1$ in order to be able to invert the terms $v_i!$.\\
\indent We now define the polynomial $P(x_{[k]})$ as
\[\sum_{(\tilde{d}, \tilde{c}) \in \on{Mon}} \tilde{\lambda}_{\tilde{d}, \tilde{c}} x_{\tilde{d}_1, \tilde{c}_1} \cdots x_{\tilde{d}_{k+r-1}, \tilde{c}_{k+r-1}}.\]

Note also that $P \in \mathcal{P}_{k,r}$ as every pair $(\tilde{d},\tilde{c}) \in \on{Mon}$ has the property that all elements in $[k]$ are present among $\tilde{d}_1, \dots, \tilde{d}_{r + k - 1}$. It remains to prove that for every monomial $m$ present in
\[\psi(a^{(1)}, \dots, a^{(r-1)}, x_{[k]}) - \Delta_{a^{(1)}, \dots, a^{(r-1)}} P(x_{[k]})\]
there is $i \in [k]$ such that no variable $x_{ic}$ for some $c$ appears in $m$. To that end, we prove the following slightly more general claim that allows us to understand how derivatives affect monomials.\\

\begin{claim}\label{polyderivcalc}Let $m(x_{[k]}) = x_{d_1, c_1} \cdots x_{d_{s}, c_{s}}$ be a monomial of degree $s$. Let $a^{(1)}, \dots, a^{(t)} \in G^{\oplus}$. Then
\[\Delta_{a^{(1)}, \dots, a^{(t)}}m(x_{[k]}) = \sum_{i \colon [t] \inj [s]} a^{(1)}_{c_{i(1)}, d_{i(1)}} a^{(2)}_{c_{i(2)}, d_{i(2)}} \cdots a^{(t)}_{c_{i(t)}, d_{i(t)}} \Big(\prod_{j \in [s] \setminus \on{Im} i} x_{c_j, d_j}\Big) + Q(a^{(1)}, a^{(2)}, \dots, a^{(t)}, x_{[k]}),\]
where the sum ranges over all injective maps $i \colon [t] \to [s]$ and $Q$ is a polynomial of degree at most $s$ whose monomials have at least $t+1$ variables of the form $a^{(i)}_{j,\ell}$.\end{claim}

\begin{proof} Note that it suffices to prove the claim for $t \leq s$. Indeed, supposing that the claim holds for $t = s$, we see that $\Delta_{a^{(1)}, \dots, a^{(s)}}m(x_{[k]})$ is of degree zero in $x_{ij}$ variables, and is therefore independent of $x_{[k]}$. Thus, further discrete derivatives give zero function. We proceed to prove the claim by induction on $t \leq s$.\\
\indent The base case is $t = 1$ for which we see that
\begin{align}\Delta_{a^{(1)}} m(x_{[k]}) = &(x_{d_1, c_1} + a^{(1)}_{d_1, c_1}) \cdots (x_{d_{s}, c_{s}} + a^{(1)}_{d_s, c_s}) - x_{d_1, c_1} \cdots x_{d_{s}, c_{s}}\nonumber\\
= & \sum_{\emptyset \not= I \subseteq [s]} \Big(\prod_{i \in I} a^{(1)}_{d_i, c_i}\Big) \Big(\prod_{i \in [s] \setminus I} x_{d_i, c_i}\Big)\nonumber\\
= & \sum_{i \in [s]} \Big(a^{(1)}_{c_i, d_i} \prod_{j \in [s] \setminus \{i\}} x_{c_j, d_j}\Big) + \sum_{\substack{I \subseteq [s]\\ |I| \geq 2}} \Big(\prod_{i \in I} a^{(1)}_{d_i, c_i}\Big) \Big(\prod_{i \in [s] \setminus I} x_{d_i, c_i}\Big),\label{derexpansion}
\end{align}
which has the desired form.

Suppose now that the claim holds for some $t \in [s-1]$. By the induction hypothesis we have
\[\Delta_{a^{(1)}, \dots, a^{(t)}}m(x_{[k]}) = \sum_{i \colon [t] \inj [s]} a^{(1)}_{c_{i(1)}, d_{i(1)}} a^{(2)}_{c_{i(2)}, d_{i(2)}} \cdots a^{(t)}_{c_{i(t)}, d_{i(t)}} \Big(\prod_{j \in [s] \setminus \on{Im} i} x_{c_j, d_j}\Big) + Q(a^{(1)}, a^{(2)}, \dots, a^{(t)}, x_{[k]}),\]
where $Q$ is a polynomial of degree $s$ whose monomials have at least $t+1$ variables of the form $a^{(i)}_{j,\ell}$, thus having $x$-degree at most $s - t - 1$. Apply further derivative with difference $a_{r+1}$
\begin{align*}\Delta_{a^{(1)}, \dots, a^{(t+1)}}m(x_{[k]}) = &\Delta_{a^{(1)}, \dots, a^{(t)}}m(x_{1} + a^{(t+1)}_{1}, \dots, x_k + a^{(t+1)}_{k}) - \Delta_{a^{(1)}, \dots, a^{(t)}}m(x_1, \dots, x_k)\\
= &\sum_{i \colon [t] \inj [s]} a^{(1)}_{c_{i(1)}, d_{i(1)}} a^{(2)}_{c_{i(2)}, d_{i(2)}} \cdots a^{(t)}_{c_{i(t)}, d_{i(t)}} \Big(\prod_{j \in [s] \setminus \on{Im} i} (x_{c_j, d_j} + a^{(t+1)}_{c_j, d_j}) - \prod_{j \in [s] \setminus \on{Im} i} x_{c_j, d_j}\Big)\\
&\hspace{1cm} + Q(a^{(1)}, a^{(2)}, \dots, a^{(t)}, x_1 + a^{(t+1)}_{1}, \dots, x_k + a^{(t+1)}_{k}) - Q(a^{(1)}, a^{(2)}, \dots, a^{(t)}, x_1, \dots, x_k).\end{align*}
As in the calculation~\eqref{derexpansion}, the only terms that have $x$-degree at least $s - t - 1$ in
\[\prod_{j \in [s] \setminus \on{Im} i} (x_{c_j, d_j} + a^{(t+1)}_{c_j, d_j}) - \prod_{j \in [s] \setminus \on{Im} i} x_{c_j, d_j}\]
are exactly
\[a^{(t+1)}_{c_{i'}, d_{i'}} \prod_{\substack{j \in [s] \setminus \on{Im} i\\j \not= i'}} x_{c_j, d_j}\]
for fixed $i' \in [s] \setminus \on{Im} i$. Furthermore, the polynomial
\[Q(a^{(1)}, a^{(2)}, \dots, a^{(t)}, x_1 + a^{(t+1)}_{1}, \dots, x_k + a^{(t+1)}_{k}) - Q(a^{(1)}, a^{(2)}, \dots, a^{(t)}, x_1, \dots, x_k)\]
has no monomials of $x$-degree at least $s - t - 1$. This follows from the fact that $Q(a^{(1)}, a^{(2)}, \dots, a^{(t)}, x_1, \dots, x_k)$ has no monomial of $x$-degree at least $s - t$ and the calculation~\eqref{derexpansion}. Therefore, there exists a polynomial $\tilde{Q}(a^{(1)}, a^{(2)}, \dots, a^{(t+1)}, x_{[k]})$ of degree at most $s$ whose monomials have $x$-degree strictly less than $s - t - 1$ and
\begin{align*}\Delta_{a^{(1)}, \dots, a^{(t+1)}}m(x_{[k]}) = &\sum_{i \colon [t] \inj [s]} a^{(1)}_{c_{i(1)}, d_{i(1)}} a^{(2)}_{c_{i(2)}, d_{i(2)}} \cdots a^{(t)}_{c_{i(t)}, d_{i(t)}} \sum_{i' \in [s] \setminus \on{Im} i} a^{(t+1)}_{c_{i'}, d_{i'}} \prod_{\substack{j \in [s] \setminus \on{Im} i\\j \not= i'}} x_{c_j, d_j}\\
&\hspace{2cm} + \tilde{Q}(a^{(1)}, a^{(2)}, \dots, a^{(t+1)}, x_{[k]})\\
= & \sum_{i \colon [t+1] \inj [s]} a^{(1)}_{d_{i(1)}, c_{i(1)}} a^{(2)}_{d_{i(2)}, c_{i(2)}} \cdots a^{(t+1)}_{d_{i(t + 1)}, c_{i(t + 1)}} \prod_{j \in [s] \setminus \on{Im} i} x_{d_j, c_j}\\
&\hspace{2cm} + \tilde{Q}(a^{(1)}, a^{(2)}, \dots, a^{(t+1)}, x_{[k]}),\end{align*}
completing the proof.\end{proof}

Using Claim~\ref{polyderivcalc} and the definition of $P$ we see that
\begin{align}&\psi(a^{(1)}, \dots, a^{(r-1)}, x_{[k]}) - \Delta_{a^{(1)}, \dots, a^{(r-1)}} P(x_{[k]}) = \sum_{(d,c,c') \in \mathcal{M}} \lambda_{d, c, c'}\, a^{(1)}_{d_1, c_1} \cdots a^{(r-1)}_{d_{r-1}, c_{r-1}} x_{1, c'_1} \cdots x_{k, c'_k}\nonumber\\
&\hspace{2cm} - \sum_{(\tilde{d},\tilde{c}) \in \on{Mon}} \tilde{\lambda}_{\tilde{d}, \tilde{c}} \sum_{i \colon [r-1] \inj [k + r - 1]} a^{(1)}_{\tilde{d}_{i(1)}, \tilde{c}_{i(1)}} a^{(2)}_{\tilde{d}_{i(2)}, \tilde{c}_{i(2)}} \cdots a^{(r-1)}_{\tilde{d}_{i(r-1)}, \tilde{c}_{i(r-1)}} \prod_{j \in [k + r - 1] \setminus \on{Im} i} x_{\tilde{d}_j, \tilde{c}_j}\nonumber\\
&\hspace{2cm} - Q(a^{(1)}, \dots, a^{(r-1)}, x_{[k]}),\label{polyderrelnpsi}\end{align}
where $Q$ is a polynomial of degree at most $k + r - 1$ whose monomials have $x$-degree at most $k - 1$. Let $\tilde{\mathcal{M}}$ be the set of all pairs $\Big((\tilde{d},\tilde{c}), i\Big)$ where $(\tilde{d},\tilde{c}) \in \on{Mon}$ and $i \colon [r-1] \inj [k + r - 1]$ such that $\{\tilde{d}_j \colon j \in [k + r - 1] \setminus \on{Im} i\} = [k]$. Note that the pairs $\Big((\tilde{d},\tilde{c}), i\Big)\notin \tilde{\mathcal{M}}$ in~\eqref{polyderrelnpsi} give rise to monomials that do not have none of the variables $x_{i_0, 1}, \dots, x_{i_0, n_{i_0}}$ appearing for some $i_0$ (namely $i_0 \notin \{\tilde{d}_j \colon j \in [k + r - 1] \setminus \on{Im} i\}$). Therefore, it suffices to prove
\begin{align}&\sum_{(d,c,c') \in \mathcal{M}} \lambda_{d, c, c'}\, a^{(1)}_{d_1, c_1} \cdots a^{(r-1)}_{d_{r-1}, c_{r-1}} x_{1, c'_1} \cdots x_{k, c'_k}\nonumber\\
&\hspace{2cm}= \sum_{\big((\tilde{d},\tilde{c}), i\big) \in \tilde{\mathcal{M}}} \tilde{\lambda}_{\tilde{d}, \tilde{c}} a^{(1)}_{\tilde{d}_{i(1)}, \tilde{c}_{i(1)}} a^{(2)}_{\tilde{d}_{i(2)}, \tilde{c}_{i(2)}} \cdots a^{(r-1)}_{\tilde{d}_{i(r-1)}, \tilde{c}_{i(r-1)}} \prod_{j \in [k + r - 1] \setminus \on{Im} i} x_{\tilde{d}_j, \tilde{c}_j}.\label{finalpsiderivequ}\end{align}
Firstly, observe that the monomial
\[a^{(1)}_{\tilde{d}_{i(1)}, \tilde{c}_{i(1)}} a^{(2)}_{\tilde{d}_{i(2)}, \tilde{c}_{i(2)}} \cdots a^{(r-1)}_{\tilde{d}_{i(r-1)}, \tilde{c}_{i(r-1)}} \prod_{j \in [k + r - 1] \setminus \on{Im} i} x_{\tilde{d}_j, \tilde{c}_j}\]
where $\Big((\tilde{d},\tilde{c}), i\Big) \in \tilde{\mathcal{M}}$ equals the monomial
\[a^{(1)}_{f_1, e_1} \cdots a^{(r-1)}_{f_{r-1}, e_{r-1}} x_{1, e'_1} \cdots x_{k, e'_k}\]
for some $(f,e,e') \in \mathcal{M}$. Namely, simply put $f_j = \tilde{d}_{i(j)}$ and $e_j = \tilde{e}_{i(j)}$ for $j \in [r-1]$ and $e'_1, \dots, e'_k$ to be the ordering $\tilde{c}_{\sigma(1)}, \dots, \tilde{c}_{\sigma(k)}$ where $\sigma \colon [k] \to [k + r - 1] \setminus \on{Im} i$ is chosen so that $\tilde{d}_{\sigma(j)} = j$. Now fix any $(f,e,e') \in \mathcal{M}$ and focus on the monomial $a^{(1)}_{f_1, e_1} \cdots a^{(r-1)}_{f_{r-1}, e_{r-1}} x_{1, e'_1} \cdots x_{k, e'_k}$. On the left-hand-side of the equality~\eqref{finalpsiderivequ} this monomial has $\lambda_{f,e,e'}$ as its coefficient. Let $\mathcal{I}$ be the set of all $\Big((\tilde{d},\tilde{c}), i\Big) \in \tilde{\mathcal{M}}$ such that we have the equality of the monomials
\[a^{(1)}_{f_1, e_1} \cdots a^{(r-1)}_{f_{r-1}, e_{r-1}} x_{1, e'_1} \cdots x_{k, e'_k} = a^{(1)}_{\tilde{d}_{i(1)}, \tilde{c}_{i(1)}} a^{(2)}_{\tilde{d}_{i(2)}, \tilde{c}_{i(2)}} \cdots a^{(r-1)}_{\tilde{d}_{i(r-1)}, \tilde{c}_{i(r-1)}} \prod_{j \in [k + r - 1] \setminus \on{Im} i} x_{\tilde{d}_j, \tilde{c}_j}.\]
In particular, we see that the sequences $(f_1, e_1), \dots,$ $(f_{r-1}, e_{r-1}),$ $(1, e'_1), \dots,$ $(k, e'_k)$ and $(\tilde{d}_1, \tilde{c}_1), \dots,$ $(\tilde{d}_{k + r-1}, \tilde{c}_{k + r - 1})$ are the same after a possible reordering, thus $(\tilde{d}, \tilde{c}) = \on{mon}(f,e,e')$, meaning that $(\tilde{d}, \tilde{c})$ is uniquely determined by $(f,e,e')$. Hence, the coefficient of the considered monomial on the right-hand-side of the equality~\eqref{finalpsiderivequ} is precisely $\tilde{\lambda}_{\on{mon}(f,e,e')} \cdot |\mathcal{I}|$, and $|\mathcal{I}|$ is the number of injective maps $i \colon [r-1] \to [k + r - 1]$ such that $(\tilde{d}_{i(j)}, \tilde{c}_{i(j)}) = (f_j, e_j)$.\\
\indent Misuse the notation and write $f_{r - 1 + j} = j$ and $e_{r - 1 + j} = e'_j$ for $j \in [k]$. Observe that we may uniquely extend the injection $i$ to a bijection $\overline{i} \colon [k + r - 1] \to [k + r - 1]$ such that $(\tilde{d}_{\overline{i}(j)}, \tilde{c}_{\overline{i}(j)}) = (f_j, e_j)$ holds for all $j \in [k + r- 1]$. In the opposite direction, every such bijection $\overline{i}$ has the same properties as the initial injection $i$ after restricting $\overline{i}|_{[r-1]}$. We conclude that $|\mathcal{I}|$ is in fact the number of such bijections $\overline{i}$. It is not hard to see that this number is exactly $v_1!\cdots v_s!$ where $s$ is the number of distinct elements of the sequence $(\tilde{d}_1, \tilde{c}_1), \dots, (\tilde{d}_{k + r - 1}, \tilde{c}_{k + r -1})$ and $v_i$ is the number of times $i$\textsuperscript{th} value appears in the sequence. By definition of $\tilde{\lambda}_{\tilde{d}, \tilde{c}}$, we have the equality of the coefficients on both sides, which completes the proof.\end{proof}
\vspace{\baselineskip}

\boldSubsection{Concluding the proof}

Finally, in this subsection we return to~\eqref{mlcorrelationinvthm} and use the facts about approximately symmetric multilinear forms we established in the previous subsections to conclude the proof of Theorem~\ref{inverseUmixed}.

\begin{proof}[Proof of Theorem~\ref{inverseUmixed}]Let $f \colon G^\oplus \to \mathbb{D}$ be a function such that
\[\|f\|_{\mathsf{U}\big(G_1, G_2, \dots, G_k, G^{\oplus} \times r\big)} \,\geq\, c.\]
By Proposition~\ref{mlStructStep} there exists a multilinear form $\psi \colon \underbrace{G^{\oplus} \tdt G^{\oplus}}_{r-1} \times G_1 \tdt G_k \to \mathbb{F}_p$ such that
\begin{equation}\Big|\exx_{\ssk{a^{(1)}, \dots, a^{(r-1)} \in G^{\oplus}\\b_1 \in G_1, \dots, b_k \in G_k\\x \in G^{\oplus}}} \mder_{a^{(1)}, \dots, a^{(r-1)}} \mder_{b_1, \dots, b_k} f(x) \omega^{\psi(a^{(1)}, \dots, a^{(r-1)}, b_1, \dots, b_k)}\Big| \geq c_1\label{inveresetolargespec}\end{equation}
for some $c_1 \geq \Big(\exp^{(O_{k,r}(1))}(O_{k,r, p}(c^{-1}))\Big)^{-1}$. By Propositions~\ref{sym1prop},~\ref{sym2prop} and~\ref{approxsymtoexact} we may find another multilinear form $\rho \colon (G^{\oplus})^{r-1} \times G_{[k]} \to \mathbb{F}_p$ such that
\begin{equation}\bias (\psi - \rho) \geq c_1^{O_{k,r}(1)}\label{psirhodiffbias}\end{equation}
and if $\rho_{ij}$ and $\rho'_{ij}$ are the multilinear forms defined by~\eqref{psidefeqn} and by~\eqref{psidefeqn2} for $\rho$ instead of $\psi$, then $\rho_{ij} = 0$ and $\rho'_{ij} = 0$. Applying Proposition~\ref{exactderivpoly} to $\rho$ we obtain polynomials $P \in \mathcal{P}_{k, r}$ and $Q$ such that for all $a_1, \dots, a_{r-1} \in G^\oplus$ and $x_{[k]} \in G_{[k]}$ we have
\begin{equation}\label{rhotopolyderiveq}\rho(a_1, \dots, a_{r-1}, x_{[k]}) = \Delta_{a_1, \dots, a_{r-1}} P(x_{[k]}) + Q(a_1, a_2, \dots, a_{r-1}, x_{[k]})\end{equation}
and for each monomial $m$ appearing in $Q$ there is $i \in [k]$ such that no variable $x_{ic}$ appears in $m$.\\

\indent We now use Gowers-Cauchy-Schwarz inequality (Lemma~\ref{gcs}) to replace $\psi$ by $\rho$ in~\eqref{inveresetolargespec}. Firstly apply Theorem~\ref{invbias} to find a positive integer $m \leq O_{k, r}(\log^{O_{k, r}(1)} c_1^{-1})$, subsets $I_i \subset [r-1]$, $J_i \subset [k]$ and multilinear forms $\beta_i \colon (G^\oplus)^{I_i} \times G_{J_i} \to \mathbb{F}_p, \gamma_i \colon (G^\oplus)^{[r-1] \setminus I_i} \times G_{[k] \setminus J_i} \to \mathbb{F}_p$ for $I \in [m]$ such that
\[\psi(a^{(1)}, \dots, a^{(r-1)}, x_1, \dots, x_k) - \rho(a^{(1)}, \dots, a^{(r-1)}, x_1, \dots, x_k) = \sum_{i \in [m]} \beta_i(a^{(I_i)}, x_{J_i}) \gamma_i(a^{([r-1] \setminus I_i)}, x_{[k] \setminus J_i})\]
and $0 < |I_i| + |J_i| < k + r -1$.\\
\indent Algebraic manipulation in~\eqref{inveresetolargespec} yields
\begin{align*}c_1 \leq &\exx_{a^{(1)}, \dots, a^{(r-1)} \in G^{\oplus}} \Big\|\mder_{a^{(1)}, \dots, a^{(r-1)}}f \omega^{\psi_{a^{(1)}, \dots, a^{(r-1)}}}\Big\|_{\square^k}^{2^k} \\
= & \exx_{\substack{a^{(1)}, \dots, a^{(r-1)} \in G^{\oplus}\\d_1, x_1 \in G_1, \dots, x_k,d_k \in G_k}} \omega^{\psi(a^{([r-1])}, d_{[k]})} \mder_{a^{(1)}, \dots, a^{(r-1)}, (0_{[2,k]},\ls{1}{d_1}), \dots, (0_{[k-1]}, \ls{k}{d_k})}f(x_{[k]})\\
= & \exx_{\substack{a^{(1)}, \dots, a^{(r-1)} \in G^{\oplus}\\d_1, x_1 \in G_1, \dots, x_k,d_k \in G_k}} \omega^{\rho(a^{([r-1])}, d_{[k]}) + \sum_{i \in [m]} \beta_i(a^{(I_i)}, x_{J_i}) \gamma_i(a^{([r-1] \setminus I_i)}, x_{[k] \setminus J_i})} \mder_{a^{(1)}, \dots, a^{(r-1)}, (0_{[2,k]},\ls{1}{d_1}), \dots, (0_{[k-1]}, \ls{k}{d_k})}f(x_{[k]})\\
= & \sum_{\lambda, \mu \in \mathbb{F}_p^m} \exx_{\substack{a_1, \dots, a_{r-1} \in G^{\oplus}\\d_1, x_1 \in G_1, \dots, x_k,d_k \in G_k}} \omega^{\rho(a_{[r-1]}, d_{[k]}) + \lambda \cdot \mu} \Big(\prod_{i \in [m]} \mathbbm{1}(\beta_i(a^{(I_i)}, x_{J_i}) = \lambda_i) \mathbbm{1}(\gamma_i(a^{([r-1] \setminus I_i)}, x_{[k] \setminus J_i}) = \mu_i)\Big)\\
&\hspace{5cm}\mder_{a^{(1)}, \dots, a^{(r-1)}, (0_{[2,k]},\ls{1}{d_1}), \dots, (0_{[k-1]}, \ls{k}{d_k})}f(x_{[k]})\\
= & p^{-2m}\sum_{\lambda, \lambda', \mu', \mu \in \mathbb{F}_p^m} \exx_{\substack{a^{(1)}, \dots, a^{(r-1)} \in G^{\oplus}\\d_1, x_1 \in G_1, \dots, x_k,d_k \in G_k}} \omega^{\rho(a^{([r-1])}, d_{[k]}) + \lambda \cdot \mu} \Big(\prod_{i \in [m]}\omega^{\lambda'_i\beta_i(a^{(I_i)}, x_{J_i}) - \lambda'_i\lambda_i} \omega^{\mu'_i\gamma_i(a^{([r-1] \setminus I_i)}, x_{[k] \setminus J_i}) - \mu'_i \mu_i}\Big)\\
&\hspace{5cm}\mder_{a^{(1)}, \dots, a^{(r-1)}, (0_{[2,k]},\ls{1}{d_1}), \dots, (0_{[k-1]}, \ls{k}{d_k})}f(x_{[k]}).\end{align*}
By averaging, we may find $x_{[k]} \in G^\oplus$ and $\lambda, \lambda', \mu, \mu' \in \mathbb{F}_p^m$ such that
\begin{align*} c_1p^{-2m}\leq &\bigg|\exx_{\substack{a^{(1)}, \dots, a^{(r-1)} \in G^{\oplus}\\d_1 \in G_1, \dots, d_k \in G_k}} \omega^{\rho(a^{([r-1])}, d_{[k]})} \Big(\prod_{i \in [m]}\omega^{\lambda'_i\beta_i(a^{(I_i)}, x_{J_i})} \omega^{\mu'_i\gamma_i(a^{([r-1] \setminus I_i)}, x_{[k] \setminus J_i})}\Big)\\
&\hspace{6cm} \mder_{a^{(1)}, \dots, a^{(r-1)}, (0_{[2,k]},\ls{1}{d_1}), \dots, (0_{[k-1]}, \ls{k}{d_k})}f(x_{[k]})\bigg|\\
=&\bigg|\exx_{\substack{a^{(1)}, \dots, a^{(r-1)} \in G^{\oplus}\\d_1 \in G_1, \dots, d_k \in G_k}} \omega^{\rho(a^{([r-1])}, d_{[k]})} \Big(\prod_{i \in [m]}\omega^{\lambda'_i\beta_i(a^{(I_i)}, x_{J_i})} \omega^{\mu'_i\gamma_i(a^{([r-1] \setminus I_i)}, x_{[k] \setminus J_i})}\Big)\\
&\hspace{1cm}\prod_{I \subseteq [r-1]} \prod_{J \subseteq [k]} \on{Conj}^{k + r - 1 -|I| - |J|} f\Big(x_1 + \sum_{i \in I} a^{(i)}_{1} + \mathbbm{1}(1 \in J) d_1,\dots,x_k + \sum_{i \in I} a^{(i)}_{k} + \mathbbm{1}(k \in J) d_k \Big)\bigg|.\end{align*}
We view this expression as an average of values of products of functions in variables $a^{(1)}, \dots, a^{(r-1)}, d_1, \dots, d_k$. The only terms that depend on all of these variables are
\[ \omega^{\rho(a^{([r-1])}, d_{[k]})}\]
and
\[f\Big(x_1 + \sum_{i \in [r-1]} a^{(i)}_{1} + d_1, \dots, x_k + \sum_{i \in [r-1]} a^{(i)}_{k} + d_k\Big).\]
By Lemma~\ref{gcs} we conclude that
\begin{align*}(c_1p^{-2m})^{2^{k+r-1}} \leq &\exx_{\substack{a^{(1)}, \dots, a^{(r-1)} \in G^{\oplus}\\d_1 \in G_1, \dots, d_k \in G_k}} \exx_{\substack{b_1, \dots, b_{r-1} \in G^{\oplus}\\e_1 \in G_1, \dots, e_k \in G_k}} \bigg(\prod_{I \subset [r-1]} \prod_{J \subset [k]} \omega^{(-1)^{k + r - 1 - |I| - |J|} \rho(a^{(I)}, b^{([r-1] \setminus I)}, d_J, e_{[k] \setminus J})}\\
&\hspace{1cm}\on{Conj}^{k + r - 1 - |I| - |J|} f\Big(x_1 + \sum_{i \in I} a^{(i)}_{1} + \sum_{i \in [r-1] \setminus I} b^{(i)}_{1} + e_1 + \mathbbm{1}(1 \in J) (d_1 - e_1), \dots,\\
&\hspace{5cm} x_k + \sum_{i \in I} a^{(i)}_k + \sum_{i \in [r-1] \setminus I} b^{(i)}_{k} + e_k + \mathbbm{1}(k \in J) (d_k - e_k)\Big)\bigg)\\
 =&\exx_{\substack{a^{(1)}, \dots, a^{(r-1)} \in G^{\oplus}\\d_1 \in G_1, \dots, d_k \in G_k}} \exx_{\substack{b^{(1)}, \dots, b^{(r-1)} \in G^{\oplus}\\e_1 \in G_1, \dots, e_k \in G_k}} \omega^{\rho(a^{(1)} - b^{(1)}, \dots, a^{(r-1)} - b^{(r-1)}, d_1 - e_1, \dots, d_k - e_k)}\\
&\hspace{1cm}\mder_{a^{(1)} - b^{(1)}, \dots, a^{(r-1)} - b^{(r-1)}, (0_{[2,k]},\ls{1}{d_1 - e_1}), \dots, (0_{[k-1]}, \ls{k}{d_k - e_k})}f\Big(x_1 + \sum_{i \in [r-1]} b^{(i)}_{1} + e_1, \dots,\\
&\hspace{11cm} x_k + \sum_{i \in [r-1]} b^{(i)}_{k} + e_k\Big)\\
=&\exx_{\substack{a^{(1)}, \dots, a^{(r-1)} \in G^{\oplus}\\d_1, y_1 \in G_1, \dots, d_k, y_k \in G_k}} \omega^{\rho(a^{(1)}, \dots, a^{(r-1)}, d_1, \dots, d_k)} \mder_{a^{(1)}, \dots, a^{(r-1)}, (0_{[2,k]},\ls{1}{d_1}), \dots, (0_{[k-1]}, \ls{k}{d_k})}f(y_1, \dots, y_k),\end{align*}
where we made a change of variables in the last line. Write $c_2 = (c_1p^{-2m})^{2^{k+r-1}}$. Using the identity~\eqref{rhotopolyderiveq} we conclude that (below we again use $x_1, \dots, x_k$ as dummy variables as the values we previously fixed have no further role in the proof)
\begin{align*}c_2 \leq & \exx_{\substack{a^{(1)}, \dots, a^{(r-1)} \in G^{\oplus}\\d_1, y_1 \in G_1, \dots, d_k, y_k \in G_k}} \omega^{\rho(a^{(1)}, \dots, a^{(r-1)}, d_1, \dots, d_k)} \mder_{a^{(1)}, \dots, a^{(r-1)}, (0_{[2,k]},\ls{1}{d_1}), \dots, (0_{[k-1]}, \ls{k}{d_k})}f(y_1, \dots, y_k)\\
= & \exx_{\substack{a^{(1)}, \dots, a^{(r-1)} \in G^{\oplus}\\x_1, y_1 \in G_1, \dots, x_k, y_k \in G_k}} \bigg(\prod_{I \subset [k]} \omega^{(-1)^{k - |I|}\rho(a^{(1)}, \dots, a^{(r-1)}, x_I, y_{[k] \setminus I})} \on{Conj}^{k - |I|}\mder_{a^{(1)}, \dots, a^{(r-1)}}f(x_I, y_{[k] \setminus I})\bigg)\\
= & \exx_{\substack{a^{(1)}, \dots, a^{(r-1)} \in G^{\oplus}\\x_1, y_1 \in G_1, \dots, x_k, y_k \in G_k}} \bigg(\prod_{I \subset [k]} \omega^{(-1)^{k - |I|} \Delta_{a^{(1)}, \dots, a^{(r-1)}} P(x_I, y_{[k] \setminus I})}\,\omega^{(-1)^{k - |I|}Q(a^{(1)}, a^{(2)}, \dots, a^{(r-1)}, x_I, y_{[k] \setminus I})}\\
&\hspace{10cm} \on{Conj}^{k - |I|}\mder_{a^{(1)}, \dots, a^{(r-1)}}f(x_I, y_{[k] \setminus I})\bigg).\end{align*}
We view terms in the expression above as functions in $x_1, \dots, x_k$, treating $y_1, \dots, y_k, a^{(1)}, \dots, a^{(r-1)}$ as fixed. Recall that for each monomial $m$ appearing in $Q$ there is $i \in [k]$ such that no variable $x_{ic}$ appears in $m$. Thus, we may write $Q(a^{(1)}, \dots, a^{(r-1)}, x_1, \dots, x_k) = \sum_{i \in [k]} Q_i(a^{(1)}, \dots, a^{(r-1)}, x_{[k] \setminus \{i\}})$ for further polynomials $Q_i$, $i \in [k]$. After expanding $Q$ like this in the expression above, the only terms that depend on all variables $x_1, \dots, x_k$ are
\[\omega^{\Delta_{a^{(1)}, \dots, a^{(r-1)}} P(x_{[k]})}\hspace{2cm}\text{and}\hspace{2cm}\mder_{a^{(1)}, \dots, a^{(r-1)}}f(x_{[k]}).\]
Apply Lemma~\ref{gcs} for all choices of $a^{([r-1])}$ and $y_{[k]}$ to get
\begin{align*}c_2^{2^{k}} \leq &\exx_{\substack{a^{(1)}, \dots, a^{(r-1)} \in G^{\oplus}\\d_1,x_1 \in G_1, \dots, d_k, x_k\in G_k}} \omega^{\Delta_{a^{(1)}, \dots, a^{(r-1)}, (0_{[2,k]},\ls{1}{d_1}), \dots, (0_{[k-1]}, \ls{k}{d_k})} P(x_{[k]})} \mder_{a^{(1)}, \dots, a^{(r-1)}, (0_{[2,k]},\ls{1}{d_1}), \dots, (0_{[k-1]}, \ls{k}{d_k})}f(x_{[k]})\\
= &\|\tilde{f} \|_{\mathsf{U}\big(G_1, G_2, \dots, G_k, G^{\oplus} \times r - 1\big)}^{2^{k + r- 1}},\end{align*}
where we set $\tilde{f}(x_{[k]}) = \omega^{P(x_{[k]})} f(x_{[k]})$. We may now apply inductive hypothesis to find a further polynomial $\tilde{P}$ of degree at most $k + r - 2$ and functions $g_i \colon G_{[k] \setminus \{i\}} \to \mathbb{D}$ for $i \in [k]$ such that
\[c_3 \leq \exx_{x_{[k]} \in G^\oplus} \tilde{f}(x_{[k]}) \omega^{\tilde{P}(x_{[k]})} \Big(\prod_{i \in [k]} g_i(x_{[k] \setminus \{i\}}) \Big)  = \exx_{x_{[k]} \in G^\oplus} f(x_{[k]}) \omega^{\big(P + \tilde{P}\big)(x_{[k]})} \Big(\prod_{i \in [k]} g_i(x_{[k] \setminus \{i\}}) \Big),\]
where
\[c_3 \geq \Big(\exp^{(O_{k,r}(1))}(O_{k,r, p}(c_2^{-1}))\Big)^{-1} \geq \Big(\exp^{(O_{k,r}(1))}(O_{k,r, p}(c^{-1}))\Big)^{-1}.\]
Since $P + \tilde{P}$ is a polynomial of degree at most $k + r -1$ the proof is now complete.\end{proof}

\vspace{\baselineskip}

\boldSection{Properties of large multilinear spectrum}

In this section, we prove some properties of the large multilinear spectrum.\\

\noindent \textbf{Close forms.} We begin the work by proving the following lemma which tells us that if a multilinear form $\alpha'$ is close to a form $\alpha$ which belongs to the $\varepsilon$-large multilinear spectrum of some function $f$, then $\alpha'$ also belongs to the $\varepsilon'$-large multilinear spectrum of $f$, for a somewhat smaller parameter $\varepsilon'$. This lemma is motivated by the first step of the proof of Theorem~\ref{inverseUmixed}. Using the notation of that proof, in that step we use the property~\eqref{psirhodiffbias} to replace the map $\psi$ whose sliced functions we know are in the large multilinear spectrum of $\mder_{a^{(1)}, \dots, a^{(r-1)}}f$, by the map $\rho$, which differs from $\psi$ by a map of small rank. Unfortunately, for technical reasons, we may not apply the lemma below directly in the proof of Theorem~\ref{inverseUmixed}, but it captures the essence of that step.

\begin{lemma}\label{closeformspec}Let $f \colon G_{[k]} \to \mathbb{D}$ be a function. Suppose that $\alpha$ and $\alpha'$ are two multilinear forms on $G_{[k]}$ such that $\bias(\alpha - \alpha') \geq c$ and $\alpha \in \mls_\varepsilon(f)$. Then $\alpha' \in \mls_{\varepsilon'}(f)$ for $\varepsilon' = \varepsilon p^{-O\big((\log_p c^{-1})^{O(1)} \big)} $.\end{lemma}

We remark that $\varepsilon$ is affected only very slightly, i.e.\ $\varepsilon'$ is linear in $\varepsilon$ rather than decaying polynomially or faster.

\begin{proof}By Theorem~\ref{invbias}, there are a positive integer $m \leq O\Big((\log_p c^{-1})^{O(1)} \Big)$, subsets $\emptyset \not= I_i \subsetneq [k]$ and multilinear forms $\beta_i \colon G_{I_i} \to \mathbb{F}_p$ and $\gamma_i \colon G_{[k] \setminus I_i} \to \mathbb{F}_p$ for $i \in [m]$ such that
\[\alpha'(x_{[k]}) = \alpha(x_{[k]}) + \sum_{i \in [m]} \beta_i(x_{I_i}) \gamma_i(x_{[k] \setminus I_i}).\]
From the assumption that $\alpha \in \mls_\varepsilon(f)$ and algebraic manipulation, we obtain
\begin{align*}\varepsilon^{2^k} \leq &\Big\|f \omega^{\alpha}\Big\|_{\square^k}^{2^k} = \exx_{x_{[k]}, y_{[k]}} \prod_{J \subseteq [k]} \on{Conj}^{k - |J|} f(x_J, y_{[k] \setminus J}) \omega^{(-1)^{k - |J|} \alpha(x_J, y_{[k] \setminus J})}\\
= & \exx_{x_{[k]}, y_{[k]}} \prod_{J \subseteq [k]} \on{Conj}^{k - |J|} f(x_J, y_{[k] \setminus J}) \omega^{(-1)^{k - |J|} \alpha'(x_J, y_{[k] \setminus J})}\, \omega^{(-1)^{k - |J|}\big(\sum_{i \in [m]} \beta_i(x_{I_i \cap J}, y_{I_i \setminus J}) \gamma_i(x_{J \setminus I_i}, y_{[k] \setminus (I_i \cup J)})\big)}\\
= & \sum_{\lambda, \mu \in \mathbb{F}_p^{\mathcal{P}([k]) \times [m]}} \exx_{x_{[k]}, y_{[k]}} \prod_{J \subseteq [k]} \on{Conj}^{k - |J|} f(x_J, y_{[k] \setminus J}) \omega^{(-1)^{k - |J|} \alpha'(x_J, y_{[k] \setminus J})}\, \omega^{(-1)^{k - |J|} \sum_{i \in [m]} \lambda_{J, i} \mu_{J, i}} \\
&\hspace{3cm}\mathbbm{1}\Big((\forall i \in [m]) \beta_i(x_{I_i \cap J}, y_{I_i \setminus J}) = \lambda_{J, i}\Big)\mathbbm{1}\Big((\forall i \in [m]) \gamma_i(x_{J \setminus I_i}, y_{[k] \setminus (I_i \cup J)}) = \mu_{J, i}\Big)\\
= &p^{-2^{k+1}m} \sum_{\lambda, \mu, \nu, \tau \in \mathbb{F}_p^{\mathcal{P}([k]) \times [m]}} \exx_{x_{[k]}, y_{[k]}} \prod_{J \subseteq [k]} \on{Conj}^{k - |J|} f(x_J, y_{[k] \setminus J}) \omega^{(-1)^{k - |J|} \alpha'(x_J, y_{[k] \setminus J})}\, \omega^{(-1)^{k - |J|} \sum_{i \in [m]} \lambda_{J, i} \mu_{J, i}} \\
&\hspace{3cm}\omega^{(-1)^{k - |J|}\nu_{J, i} \big(\beta_i(x_{I_i \cap J}, y_{I_i \setminus J}) - \lambda_{J, i}\big)}\, \omega^{(-1)^{k - |J|}\tau_{J, i} \big(\gamma_i(x_{J \setminus I_i}, y_{[k] \setminus (I_i \cup J)}) - \mu_{J, i}\big)}.
\end{align*}

\noindent By averaging, we may find $\lambda, \mu, \nu, \tau \in \mathbb{F}_p^{\mathcal{P}([k]) \times [m]}$ such that
\begin{align} p^{-2^{k+1}m} \varepsilon^{2^k} \leq &\bigg|\exx_{x_{[k]}, y_{[k]}} \prod_{J \subseteq [k]} \on{Conj}^{k - |J|} f(x_J, y_{[k] \setminus J}) \omega^{(-1)^{k - |J|} \alpha'(x_J, y_{[k] \setminus J})}\nonumber\\
&\hspace{3cm}\omega^{(-1)^{k - |J|}\nu_{J, i} \big(\beta_i(x_{I_i \cap J}, y_{I_i \setminus J}) - \lambda_{J, i}\big)}\, \omega^{(-1)^{k - |J|}\tau_{J, i} \big(\gamma_i(x_{J \setminus I_i}, y_{[k] \setminus (I_i \cup J)}) - \mu_{J, i}\big)}\bigg|.\label{closeformsFCineq}\end{align}
For each $J \subseteq [k]$, let $f_J \colon G_{[k]} \to \mathbb{D}$ be the function defined by
\[f_J(v_{[k]}) = f(v_{[k]}) \omega^{\alpha'(v_{[k]})} \omega^{\nu_{J, i} \big(\beta_i(v_{I_i}) - \lambda_{J, i}\big)}\, \omega^{\tau_{J, i} \big(\gamma_i(v_{[k] \setminus I_i}) - \mu_{J, i}\big)}.\]

\noindent The inequality~\eqref{closeformsFCineq} then becomes
\[ p^{-2^{k+1}m} \varepsilon^{2^k} \leq \bigg|\exx_{x_{[k]}, y_{[k]}} \prod_{J \subseteq [k]} \on{Conj}^{k - |J|} f_J(x_J, y_{[k] \setminus J})\bigg|\]
which can be bounded from above by
\[\prod_{J \subseteq [k]} \|f_J\|_{\square^k}\]
using Lemma~\ref{gcs}. Finally, observe that in fact $ \|f_J\|_{\square^k} =  \|f \omega^{\alpha'}\|_{\square^k}$ for each $J \subseteq [k]$, so we actually obtain $\|f \omega^{\alpha'}\|_{\square^k} \geq p^{-2m} \varepsilon$, as desired.
\end{proof}
\vspace{\baselineskip}

\noindent\textbf{Bounding the large mutlilinear spectrum.} A basic yet fundamental fact about the large spectrum of a function of a single variable is that it is necessarily small, which is easily proved via Parseval's identity. However, as we have seen already in Lemma~\ref{closeformspec}, in the case of the large multilinear spectrum, the situation is more complicated as the large multilinear spectrum is approximately closed under translating the forms by further forms of large bias. Still, it turns out that we can recover the result for single variable if we treat the forms whose difference has large bias as the same. In other words, if we pick many elements from $\mls_\varepsilon(f)$ then some two are almost identical.

\begin{theorem}\label{genParsevalBound}Let $k \in \mathbb{N}$. For $\varepsilon > 0$, set $b(\varepsilon) = \Big(\frac{\varepsilon}{1000}\Big)^{2^{2k + 2}}, n(\varepsilon) = \lceil 10\varepsilon^{-2^{k + 1}}\rceil$. Let $f \colon G_{[k]} \to \mathbb{D}$ be a function. Let $\mu_1, \dots, \mu_n \in \mls_\varepsilon(f)$ be multilinear forms such that
\[\bias(\mu_i - \mu_j) \leq b(\varepsilon)\]
for every $i \not= j$. Then $n < n(\varepsilon)$.
\end{theorem}

The proof is very similar to the usual one based on Parseval's identity. We write $\mathbb{S} = \{z \in \mathbb{C} \colon |z| = 1\}$ for the unit circle.

\begin{proof}Suppose, for the sake of contradiction, that $n = n(\varepsilon)$. Expanding the definition of box norms yields
\[\Big|\exx_{x_{[k]}, y_{[k]}} \prod_{I \subset [k]} \on{Conj}^{k - |I|} f(x_I, y_{[k] \setminus}) \omega^{(-1)^{k - |I|}\mu_i(x_I, y_{[k] \setminus I})}\Big| \geq \varepsilon^{2^k}.\]
By averaging we may find $y_{[k]} \in G_{[k]}$ such that
\[\Big|\exx_{x_{[k]}} \prod_{I \subset [k]} \on{Conj}^{k - |I|} f(x_I, y_{[k] \setminus}) \omega^{(-1)^{k - |I|}\mu_i(x_I, y_{[k] \setminus I})} \Big| \geq \varepsilon^{2^k}.\]
We may rewrite this expression as
\[\Big|\ex_{x_{[k]}}f(x_{[k]}) \omega^{\mu_i(x_{[k]})}v^{(i)}_1(x_{[2,k]}) \cdots v^{(i)}_k(x_{[k-1]})\Big| \geq \varepsilon^{2^k}\]
for some functions $v^{(i)}_1, \dots, v^{(i)}_k$ where $v^{(i)}_j \colon G_{[k] \setminus \{j\}} \to \mathbb{D}$. Using Lemma~\ref{unitLemma} we get functions $u^{(i)}_1, \dots, u^{(i)}_k$ taking values on the unit circle $\mathbb{S}$ such that
\[\Big|\ex_{x_{[k]}} f(x_{[k]}) \omega^{\mu_i(x_{[k]})} \prod_{j \in [k]} u^{(i)}_j(x_{[k] \setminus \{j\}})\Big| \geq \varepsilon^{2^k}.\]
Write $s_i(x_{[k]}) = \omega^{\mu_i(x_{[k]})} \prod_{j \in [k]} u^{(i)}_j(x_{[k] \setminus \{j\}})$ and let $\langle\,\cdot,\cdot\rangle$ be the usual inner product on $G_{[k]}$. Observe that
\[\langle s_i, s_i \rangle = \exx_{x_{[k]}} \prod_{j \in [k]} |u^{(i)}_j(x_{[k] \setminus \{j\}})|^2 = 1\]
and, using Corollary~\ref{gcsCor}, for $i \not= j$
\[|\langle s_i, s_j \rangle| = \Big|\exx_{x_{[k]}} \omega^{\mu_i(x_{[k]}) - \mu_j(x_{[k]})} \prod_{\ell \in [k]} \Big(u^{(i)}_\ell(x_{[k] \setminus \{\ell\}}) \overline{u^{(j)}_\ell(x_{[k] \setminus \{\ell\}})} \Big)\Big| \leq \Big\|\omega^{\mu_i - \mu_j}\Big\|_{\square^{2^k}} = \bias(\mu_i - \mu_j)^{2^{-k}} \leq b(\varepsilon)^{2^{-k}}.\]
Set $c_i = \langle f, s_i\rangle$ and set $e = f - \sum_{i \in [n]} c_i s_i$. Even though we no longer have a Fourier decomposition, we shall think of quantities $c_i$ as large Fourier coefficients and of $e$ as the error term coming from the small Fourier coefficients. Note that
\[|\langle e, s_i \rangle| = \Big|\Big\langle f - \sum_{j \in [n]} c_j s_j, s_i \Big\rangle\Big| \leq \sum_{j \in [n] \setminus \{i\}} |\langle s_i, s_j\rangle| \leq n b(\varepsilon)^{2^{-k}}.\]
Using the identity $f = \sum_{i \in [n]} c_i s_i + e$ we get
\begin{align*}1 \geq &\langle f, f\rangle = \sum_{i \in [n]} |c_i|^2 + \langle e, e \rangle + \sum_{\substack{i,j \in [n]\\i\not=j}} c_i \overline{c_j}\langle s_i, s_j\rangle + \sum_{i \in [n]} \Big(\overline{c_i}\langle e, s_i\rangle + c_i\langle s_i, e\rangle\Big)\\
&\hspace{1cm}\geq \sum_{i \in [n]} |c_i|^2 - \sum_{\substack{i,j \in [n]\\i\not=j}} |\langle s_i, s_j\rangle| - 2  \sum_{i \in [n]} |\langle e, s_i\rangle|\\
&\hspace{1cm} \geq n \varepsilon^{2^{k + 1}} - 3 n^2 b(\varepsilon)^{2^{-k}}.\end{align*}
Recall that we assummed $n = n(\varepsilon) = \lceil 10\varepsilon^{-2^{k + 1}}\rceil$. Thus, we have
\[1 \geq 10 - 3 \Big(10\varepsilon^{-2^{k+1}} + 1\Big)^2 b(\varepsilon)^{2^{-k}} \geq 10 - 6 - 600\varepsilon^{-2^{k+2}}b(\varepsilon)^{2^{-k}}.\]
Since $b(\varepsilon) = \Big(\frac{\varepsilon}{1000}\Big)^{2^{2k + 2}}$, we have a contradiction.\end{proof}

Using Theorem~\ref{invbias} we deduce the following result.

\begin{theorem}Let $f \colon G_{[k]} \to \mathbb{D}$ be a function. Let $\mu_1, \dots, \mu_n \in \mls_\varepsilon(f)$ be multilinear forms for some $n \geq 20\varepsilon^{-2^{k + 1}}$. Then there are distinct indices $i,j \in [n]$ such that
\[\on{prank} (\mu_i - \mu_j) \leq O_k(\log^{O_k(1)}_p \varepsilon^{-1}).\]\end{theorem}
\vspace{\baselineskip}

\noindent\textbf{Chang's theorem for the large multilinear spectrum.} Well-known theorem of Chang~\cite{Chang} states that the large spectrum contains rich additive structure. For a multilinear variant of Chang's theorem we need to be somewhat more careful as tightness in inequalities is crucial. Let $\alpha = \ex_x |f(x)|$. We have to assume the explicit inequalities
\[\Big|\ex_{x_{[k]}} f(x_{[k]}) \omega^{\mu_i(x_{[k]})} \prod_{j \in [k]} u^{(i)}_j(x_{[k] \setminus \{j\}})\Big| \geq \varepsilon\alpha\]
instead of just $\mls_{\alpha\varepsilon}(f)$ because of the slight inefficiencies arising from the application of Gowers-Cauchy-Schwarz inequality for box norms.

\begin{theorem}There is an absolute constant $C_0$ such that following holds. Suppose that $f \colon G_{[k]} \to \mathbb{D}$ is a function and write $\alpha =\ex_x |f(x)|$. Set $n(\varepsilon,\alpha) = C_0 \varepsilon^{-2} \log \alpha^{-1}$ and $b(\varepsilon, \alpha) = 3^{-2^k} \Big(\frac{\varepsilon^2 \alpha^2}{n(\varepsilon,\alpha)}\Big)^{2^k}$.\\
\indent Let $\mu_1, \dots, \mu_n$ be multilinear forms and let $u^{(i)}_j \colon G_{[k] \setminus \{j\}} \to \mathbb{D}$ be functions for $i \in [n]$, $j \in [k]$, such that
\[\Big|\ex_{x_{[k]}} f(x_{[k]}) \omega^{\mu_i(x_{[k]})} \prod_{j \in [k]} u^{(i)}_j(x_{[k] \setminus \{j\}})\Big| \geq \varepsilon\alpha\]
holds for all $i \in [n]$ and
\[\bias\Big(\sum_{i \in [n]} \lambda_i \mu_i\Big) \leq b(\varepsilon)\]
holds for every $\lambda \in \mathbb{F}_p^n \setminus \{0\}$. Then $n < n(\varepsilon)$.
\end{theorem}

Again, one may use Theorem~\ref{invbias} to turn the bias bound into a bound on the partition rank.

\begin{proof}We begin this proof just like the proof of Theorem~\ref{genParsevalBound}. For the sake of contradiction, we assume that $n = n(\varepsilon)$. Misusing the notation and using Lemma~\ref{unitLemma}, we may assume that $u^{(i)}_j$ takes values in $\mathbb{S}$ for all $i \in [n], j \in [k]$. Write $s_i(x_{[k]}) = \omega^{\mu_i(x_{[k]})} \prod_{j \in [k]} u^{(i)}_j(x_{[k] \setminus \{j\}})$. Rest of the proof mimics the original proof of Chang. Let $c_i = \langle f, s_i \rangle$ and define an auxiliary function $g \colon G_{[k]} \to \mathbb{C}$ by
\[g(x_{[k]}) = \frac{1}{C} \sum_{i \in [n]} c_i s_i(x_{[k]}),\]
where $C \in \mathbb{R}_{> 0}$ was chosen so that $\ex_{x_{[k]}} |g(x_{[k]})|^2 = 1$. Thus,
\[C^2 = \exx_{x_{[k]}} \Big|\sum_{i \in [n]} c_i s_i(x_{[k]})\Big|^2 = \sum_{i \in [n]} |c_i|^2 \exx_{x_{[k]}}| s_i(x_{[k]})|^2 + \sum_{\substack{i,j \in [n]\\i \not= j}} c_i \overline{c_j} \langle s_i, s_j \rangle = \sum_{i \in [n]} |c_i|^2 + \sum_{\substack{i,j \in [n]\\i \not= j}} c_i \overline{c_j} \langle s_i, s_j \rangle.\]
As in the previous proof, we have that $|\langle s_i, s_j \rangle| \leq b(\varepsilon)^{2^{-k}}$ for $i \not = j$, so we deduce
\[\Big|C^2 - \sum_{i \in [n]} |c_i|^2 \Big| \leq n^2 b(\varepsilon)^{2^{-k}}.\]
From our choices of $n(\varepsilon)$ and $b(\varepsilon)$ we conclude that 
\begin{equation}\label{CboundsChang} \frac{1}{\sqrt{2}}\sqrt{\sum_{i \in [n]} |c_i|^2} \leq C \leq 2 \sqrt{\sum_{i \in [n]} |c_i|^2}.\end{equation}
We have the following lower bound for $\langle g, f \rangle$
\begin{align} |\langle g, f \rangle|\, =\, & \Big|\frac{1}{C} \sum_{i \in [n]}c_i \langle s_i, f \rangle\Big|\, =\, \frac{1}{C} \sum_{i \in [n]}|c_i|^2\, \geq\, \frac{1}{2}\sqrt{\sum_{i \in [n]}|c_i|^2} \geq \frac{1}{2}\sqrt{n} \varepsilon \alpha.\label{changineq1}\end{align}
On the other hand, for the upper bound we use H\H{o}lder's inequality with exponents $l$ and $m$ to be chosen later
\begin{equation}|\langle g, f \rangle| \leq \|g\|_{L^m} \|f\|_{L^l} \leq \alpha^{1/l} \|g\|_{L^m} = \alpha^{1 - 1/m} \Big(\ex_{x_{[k]}} |g(x_{[k]})|^m\Big)^{1/m}.\label{changineq2}\end{equation}
We now prove a variant of Rudin's inequality~\cite{Rudin} for dissociated sets. The proof is a straightforward adaptation of the proof in~\cite{TaoVuBook}.

\begin{claim}Let $\sigma > 0$ and $\theta \in \mathbb{S}$ be given. Then
\[\exx_{x_{[k]}} \Big[\exp\Big(\sigma \on{Re} \sum_{i \in [n]} \frac{\theta c_i}{C} s_i(x_{[k]})\Big)\Big] \leq 2e^{\sigma^2}.\]
\end{claim}

\begin{proof}We use the following elementary inequality from the proof of Theorem 4.33 in~\cite{TaoVuBook}
\[e^{tu} \leq \cosh(u) + t \sinh(u)\]
which holds for all $u \geq 0$ and $t \in [-1,1]$. Write $\theta c_i = |c_i| \nu_i$ for a suitable $\nu_i \in \mathbb{S}$. Consequently
\[\exp\Big(\sigma \on{Re} \frac{\theta c_i}{C} s_i(x_{[k]})\Big) \leq \cosh\Big(\sigma \frac{|c_i|}{C}\Big) + \frac{1}{2}\nu_i s_i(x_{[k]})\sinh\Big(\sigma \frac{|c_i|}{C}\Big) + \frac{1}{2}\overline{\nu_i s_i(x_{[k]})}\sinh\Big(\sigma \frac{|c_i|}{C}\Big)\]
holds for each $i \in [n]$. Using this inequality, we see that
\begin{align}&\exx_{x_{[k]}} \Big[\exp\Big(\sigma \on{Re} \sum_{i \in [n]} \frac{\theta c_i}{C} s_i(x_{[k]})\Big)\Big] = \exx_{x_{[k]}} \Big[\prod_{i \in [n]} \exp\Big(\sigma \on{Re} \frac{\theta c_i}{C} s_i(x_{[k]})\Big)\Big]\nonumber\\
&\hspace{2cm} \leq \exx_{x_{[k]}} \bigg[\prod_{i \in [n]} \bigg(\cosh\Big(\sigma \frac{|c_i|}{C}\Big) + \frac{1}{2}\nu_i s_i(x_{[k]})\sinh\Big(\sigma \frac{|c_i|}{C}\Big) + \frac{1}{2}\overline{\nu_i s_i(x_{[k]})}\sinh\Big(\sigma \frac{|c_i|}{C}\Big)\bigg)\bigg].\label{rudinstep1}\end{align}
The product appearing above is a product of $n$ sums of three terms, so expansion results in $3^n$ terms in total, each being of the form
\[\exx_{x_{[k]}} \bigg(\prod_{i \in I_1}\cosh\Big(\sigma \frac{|c_i|}{C}\Big)\bigg)\,\bigg(\prod_{i \in I_2} \frac{1}{2}\nu_i s_i(x_{[k]})\sinh\Big(\sigma \frac{|c_i|}{C}\Big)\bigg)\,\bigg(\prod_{i \in I_3} \frac{1}{2}\overline{\nu_i s_i(x_{[k]})}\sinh\Big(\sigma \frac{|c_i|}{C}\Big)\bigg)\]
for some partition $[n] = I_1 \cup I_2 \cup I_3$. In the case when $I_2 \cup I_3 \not= \emptyset$, we may bound as follows
\begin{align*}&\bigg|\exx_{x_{[k]}} \bigg(\prod_{i \in I_1}\cosh\Big(\sigma \frac{|c_i|}{C}\Big)\bigg)\,\bigg(\prod_{i \in I_2} \frac{1}{2}\nu_i s_i(x_{[k]})\sinh\Big(\sigma \frac{|c_i|}{C}\Big)\bigg)\,\bigg(\prod_{i \in I_3} \frac{1}{2}\overline{\nu_i s_i(x_{[k]})}\sinh\Big(\sigma \frac{|c_i|}{C}\Big)\bigg)\bigg|\\
\leq\, & 2^{-|I_2| - |I_3|}\bigg(\prod_{i \in I_1}\cosh\Big(\sigma \frac{|c_i|}{C}\Big) \prod_{i \in I_2 \cup I_3}\sinh\Big(\sigma \frac{|c_i|}{C}\Big) \bigg) \Big|\exx_{x_{[k]}} \prod_{i \in I_2} s_i(x_{[k]}) \prod_{i \in I_3} \overline{s_i(x_{[k]})}\Big|\\
=\,&2^{-|I_2| - |I_3|}\bigg(\prod_{i \in I_1}\cosh\Big(\sigma \frac{|c_i|}{C}\Big) \prod_{i \in I_2 \cup I_3}\sinh\Big(\sigma \frac{|c_i|}{C}\Big) \bigg)\\
&\hspace{2cm} \Big|\exx_{x_{[k]}} \omega^{\sum_{i \in I_2}\mu_i(x_{[k]}) - \sum_{i \in I_3} \mu_i(x_{[k]})} \prod_{i \in I_2} \prod_{j \in [k]}  u^{(i)}_j(x_{[k] \setminus \{j\}})\prod_{i \in I_3} \prod_{j \in [k]} \overline{ u^{(i)}_j(x_{[k] \setminus \{j\}})}\Big|\\
\leq\,&2^{-|I_2| - |I_3|}\bigg(\prod_{i \in [k]}\cosh\Big(\sigma \frac{|c_i|}{C}\Big)\bigg)\, \bias\Big(\sum_{i \in I_2}\mu_i(x_{[k]}) - \sum_{i \in I_3} \mu_i(x_{[k]})\Big)^{2^{-k}}\\
\leq\,&\bigg(\prod_{i \in [k]}\cosh\Big(\sigma \frac{|c_i|}{C}\Big)\bigg) b(\varepsilon)^{2^{-k}}.\end{align*}
On the other hand, when $I_1 = [k]$, we get the constant term $\prod_{i \in [k]}\cosh\Big(\sigma \frac{|c_i|}{C}\Big)$. Therefore, going back to~\eqref{rudinstep1} we obtain
\begin{align*}\exx_{x_{[k]}} \Big[\exp\Big(\sigma \on{Re} \sum_{i \in [n]} \frac{\theta c_i}{C} s_i(x_{[k]})\Big)\Big] \leq& 2\prod_{i \in [k]}\cosh\Big(\sigma \frac{|c_i|}{C}\Big) \leq 2\prod_{i \in [k]} \exp\Big(\sigma^2 \frac{|c_i|^2}{2C^2}\Big) = 2\exp\Big(\sigma^2 \sum_{i \in [n]} \frac{|c_i|^2}{2C^2}\Big) \\
\leq &2\exp(\sigma^2),\end{align*}
where we used the left inequality in~\eqref{CboundsChang} in the last line.\end{proof}

Let $\sigma > 0$ to be chosen later. Recall that $g(x_{[k]}) = \frac{1}{C} \sum_{i \in [n]} c_i s_i(x_{[k]})$. For $\lambda > 0$, we get from the claim above for angles $\theta_\ell = e^{2 \pi i \ell / 6}$ for $\ell \in [6]$
\begin{align*}&\frac{1}{|G_{[k]}|} \Big|\Big\{x_{[k]} \in G_{[k]} \colon |g(x_{[k]})| \geq \lambda\Big\}\Big| \leq \frac{1}{|G_{[k]}|} \sum_{\ell \in [6]} \Big|\Big\{x_{[k]} \in G_{[k]} \colon \on{Re}\, \theta_\ell g(x_{[k]}) \geq \frac{1}{2}\lambda\Big\}\Big|\\
&\hspace{2cm}\leq \sum_{\ell \in [6]} \exp(- \sigma \lambda / 2)\exx_{x_{[k]}} \exp\Big(\sigma \on{Re}\, \theta_\ell g(x_{[k]})\Big) \leq 12 \exp(\sigma^2 - \sigma \lambda / 2).\end{align*}
Pick $\sigma = \lambda/4$ to get
\[\frac{1}{|G_{[k]}|} \Big|\Big\{x_{[k]} \in G_{[k]} \colon |g(x_{[k]})| \geq \lambda\Big\}\Big| \leq 12 \exp(-\lambda^2 / 16).\]
Finally, we estimate $\|g\|_{L^m}$
\begin{align*}&\exx_{x_{[k]}} |g(x_{[k]})|^m = \exx_{x_{[k]}} m\int_0^{|g(x_{[k]})|} \lambda^{m-1}d\lambda =  \exx_{x_{[k]}} m\int_0^1\lambda^{m-1}\mathbbm{1}(\lambda \leq |g(x_{[k]})|)d\lambda\\
&\hspace{2cm} =  m\int_0^1\lambda^{m-1}\Big(\exx_{x_{[k]}}\mathbbm{1}(\lambda \leq |g(x_{[k]})|)\Big)d\lambda\\
&\hspace{2cm} \leq  2m\int_0^1\lambda^{m-1} \exp(-\lambda^2 / 16)d\lambda\\
&\hspace{2cm} \leq m^{D} (Dm)^{m/2},\end{align*}
where $D \geq 1$ is an absolute constant independent of other parameters in this proof. Thus $\|g\|_{L^m} \leq \sqrt{D} e^D \sqrt{m}$. Combining inequalities~\eqref{changineq1} and~\eqref{changineq2} and squaring we get
\[n \leq 4 D e^{2D} \varepsilon^{-2} m\alpha^{-2/m}.\]
We put $m = \log \alpha^{-1}$ to obtain
\[n \leq 4 D e^{2D + 2} \varepsilon^{-2} \log \alpha^{-1}\]
which is a contradiction provided $C_0 > 4 D e^{2D + 2}$.\end{proof}
\vspace{\baselineskip}

\noindent\textbf{Cubical convolutions and the large multilinear spectrum.} The last property of the large multilinear spectrum that we prove here is the fact that the large multilinear spectrum is sufficient for the approximation of cubical convolutions (in the sense of Theorem~\ref{multConv}), as remarked in the introduction. 

\begin{proposition}\label{cubconvmlsapprox}Let $f_I \colon G_1 \tdt G_k \to \mathbb{D}$ be a function for each subset $I \subseteq [k]$. Let $\varepsilon > 0$. Then, there are a positive quantity $\xi \geq \bigg(\exp^{(O_{k}(1))}\Big(O_{k,p}(\varepsilon^{-1})\Big)\bigg)^{-1}$, a positive integer $m \leq \exp^{(O_{k}(1))}\Big(O_{k,p}(\varepsilon^{-1})\Big)$, multilinear forms $\alpha_1, \dots, \alpha_m \in \mls_\xi(f)$ and constants $c_1, \dots, c_m \in \mathbb{D}$ such that
\[\Big\|\square f_\bcdot - \sum_{i \in [m]} c_i \omega^{\alpha_i}\Big\|_{L^2} \leq \varepsilon.\]
\end{proposition}

\begin{proof}Let $C,D\geq 1$ be the implicit constants in the conclusion of Lemma~\ref{closeformspec} such that the final bound is actually $\varepsilon' = \varepsilon p^{-C\big((\log_p c^{-1})^D \big)}$.\\
\indent Apply Theorem~\ref{multConv} with approximation parameter $\varepsilon/2$. We obtain a positive integer $m \leq \exp^{(O_{k}(1))}\Big(O_{k,p}(\varepsilon^{-1})\Big)$, a multiaffine forms $\alpha_1, \dots, \alpha_m \colon G_1 \tdt G_k \to \mathbb{F}_p$ and a function $c \colon \mathbb{F}_p^m \to \mathbb{D}$ such that
\[\Big\|\square f_\bcdot - \sum_{\lambda \in \mathbb{F}_p^m} c(\lambda) \mathbbm{1}(\alpha = \lambda)\Big\|_{L^2} \leq \varepsilon/2.\]
We may rewrite 
\[\sum_{\lambda \in \mathbb{F}_p^m} c(\lambda) \mathbbm{1}(\alpha = \lambda) = p^{-m} \sum_{\lambda, \mu \in \mathbb{F}_p^m} c(\lambda) \omega^{\mu \cdot (\alpha - \lambda)}.\]

For $\mu \in \mathbb{F}_p^m$ we write
\[s_{\mu} = p^{-m} \Big(\sum_{\lambda \in \mathbb{F}^m_p} c(\lambda) \omega^{- \lambda \cdot \mu}\Big)\, \omega^{\mu \cdot \alpha}\]
which equals $c'_\mu \omega^{\mu \cdot \alpha}$ for some $c'_\mu \in \mathbb{D}$. Thus, the approximation above can be expressed as 
\begin{equation}\label{approxcubicalmls}\Big\|\square f_\bcdot - \sum_{\mu \in \mathbb{F}_p^m} s_{\mu}\Big\|_{L^2} \leq \varepsilon/2.\end{equation} 

Set
\[\xi_i = \bigg(p^{-C (2m + 3 + \log_p \varepsilon^{-1})^D} p^{-2^km} \frac{\varepsilon^{2^{k+1}}}{8^{2^k}}\bigg)^i\]
for $i = 0,1, \dots, p^m$. We now perform an iterative procedure in which after $i$ steps we obtain a subset $S_i \leq \mathbb{F}_p^m$ of size $i$ with the property that $\mu \cdot \alpha \in \mls_{\xi_i}(f)$ for each $\mu \in S_i$, until the procedure terminates. The condition for termination is that 
\[\Big\|\square f_\bcdot - \sum_{\mu \in S_i} s_{\mu}\Big\|_{L^2} \leq \varepsilon.\]
Note that if it does not terminate earlier, the procedure will stop after $(p^m)$\tss{th} step, due to~\eqref{approxcubicalmls}. Initially, we set $S_0 = \{0\}$.\\

Suppose therefore that the procedure does not terminate after $i$\tss{th} step. Then we have
\begin{align*}\varepsilon^2/4 \geq &\Big\|\square f_\bcdot - \sum_{\mu \in \mathbb{F}^m_p} s_{\mu}\Big\|_{L^2}^2 = \Big\|\Big(\square f_\bcdot - \sum_{\mu \in S_i} s_{\mu}\Big) - \Big(\sum_{\mu \in \mathbb{F}^m_p} s_{\mu} - \sum_{\mu \in S_i} s_{\mu}\Big)\Big\|_{L^2}^2\\
= &\Big\|\square f_\bcdot - \sum_{\mu \in S_i} s_{\mu}\Big\|_{L^2}^2 +  \Big\|\sum_{\mu \in \mathbb{F}^m_p} s_{\mu} - \sum_{\mu \in S_i} s_{\mu}\Big\|_{L^2}^2\\
&\hspace{2cm}- \Big\langle\square f_\bcdot - \sum_{\mu \in S_i} s_{\mu}, \sum_{\mu \in \mathbb{F}^m_p \setminus S_i} s_{\mu}\Big\rangle - \Big\langle\square \sum_{\mu \in \mathbb{F}^m_p \setminus S_i} s_{\mu}, f_\bcdot - \sum_{\mu \in S_i} s_{\mu}\Big\rangle\\
\geq &\,\varepsilon^2- \Big\langle\square f_\bcdot - \sum_{\mu \in S_i} s_{\mu}, \sum_{\mu \in \mathbb{F}^m_p \setminus S_i} s_{\mu}\Big\rangle - \Big\langle\square \sum_{\mu \in \mathbb{F}^m_p \setminus S_i} s_{\mu}, f_\bcdot - \sum_{\mu \in S_i} s_{\mu}\Big\rangle.\end{align*}

From this inequality we conclude that either $\Big|\langle f_\bcdot , s_{\mu}\rangle \Big| \geq p^{-m} \frac{\varepsilon^2}{8}$ for some $\mu \notin S_i$ or $\Big|\langle s_\lambda, s_{\mu}\rangle \Big| \geq p^{-2m} \frac{\varepsilon^2}{8}$, for some $\lambda \in S_i, \mu \notin S_i$.\\
\indent The former possibility implies
\[p^{-m} \frac{\varepsilon^2}{8} \leq \Big|\langle f_\bcdot , s_{\mu}\rangle \Big| = \Big|\exx_{x_{[k]}}  f_\bcdot(x_{[k]}) \omega^{\mu \cdot \alpha(x_{[k]})}\Big| \leq \exx_{a_{[k]}} \Big|\exx_{x_{[k]}} \Big( \prod_{I \subseteq [k]} \on{Conj}^{k - |I|} f((a + x)_I, a_{[k] \setminus I}) \Big)\,\omega^{\mu \cdot \alpha(x_{[k]})}\Big|.\]
By Lemma~\ref{gcs} for variables $x_{[k]}$ we get $\mu \cdot \alpha \in \mls_{\eta}(f)$, where $\eta = p^{-2^km} \frac{\varepsilon^{2^{k+1}}}{8^{2^k}}$.\\
\indent The latter possibility implies that
\[\bias\Big((\lambda - \mu) \cdot \alpha\Big) = \exx_{x_{[k]}} \omega^{(\lambda - \mu) \cdot \alpha(x_{[k]})} = \Big|\langle s_\lambda, s_{\mu}\rangle \Big| \geq p^{-2m} \frac{\varepsilon^2}{8},\]
which, combined with the fact that $\lambda \cdot \alpha \in \mls_{\xi_i}(f)$, by Lemma~\ref{closeformspec} implies that $\mu \cdot \alpha \in \mls_{\eta}(f)$ for 
\[\eta = \xi_i p^{-C (2m + 3 + \log_p \varepsilon^{-1})^D}.\]

Thus, in either case, we conclude that $\mu \cdot \alpha \in \mls_{\eta}(f)$ for 
\[\eta = \xi_i p^{-C (2m + 3 +  \log_p \varepsilon^{-1})^D} p^{-2^km} \frac{\varepsilon^{2^{k+1}}}{8^{2^k}} = \xi_{i+1}.\]
We may therefore set $S_{i+1} = S_i \cup \{\mu\}$.\\

Finally, after the procedure has terminated, note that 
\[\Big\|\square f_\bcdot - \sum_{\mu \in S_i} c'_\mu \omega^{\mu \cdot \alpha}\Big\|_{L^2} = \Big\|\square f_\bcdot - \sum_{\mu \in S_i} s_{\mu}\Big\|_{L^2} \leq \varepsilon,\]
which is the desired approximation.\end{proof}


\thebibliography{99}
\bibitem{Austin1} T. Austin, \emph{Partial difference equations over compact Abelian groups, I: modules of solutions}, arXiv preprint (2013), \verb+arXiv:1305.7269+.
\bibitem{Austin2} T. Austin, \emph{Partial difference equations over compact Abelian groups, II: step-polynomial solutions}, arXiv preprint (2013), \verb+arXiv:1309.3577+.
\bibitem{BalogSzemeredi} A. Balog and E. Szemer\'edi, \emph{A statistical theorem of set addition}, Combinatorica \textbf{14} (1994), 263--268. 
\bibitem{BergelsonTaoZiegler} V. Bergelson, T. Tao and T. Ziegler, \emph{An inverse theorem for the uniformity seminorms associated with the action of $\mathbb{F}^{\infty}_p$}, Geometric and Functional Analysis \textbf{19} (2010), 1539--1596.
\bibitem{BhowLov} A. Bhowmick and S. Lovett, \emph{Bias vs structure of polynomials in large fields, and applications in effective algebraic geometry and coding theory}, arXiv preprint (2015), \verb+arXiv:1506.02047+.
\bibitem{CamSzeg} O.A. Camarena and B. Szegedy, \emph{Nilspaces, nilmanifolds and their morphisms}, arXiv preprint (2010), \verb+arXiv:1009.3825+.
\bibitem{Chang} M.-C. Chang, \emph{A polynomial bound in Freiman's theorem}, Duke Mathematical Journal \textbf{113} (2002), 399--419.
\bibitem{CohenMoshkovitz} A. Cohen and G. Moshkovitz, \emph{An optimal inverse theorem}, arXiv preprint (2021), \verb+arXiv:2102.10509+. 
\bibitem{Freiman} G. Freiman, \textbf{Foundations of a structural theory of set addition}, Translations of Mathematical Monographs \textbf{37}, American Mathematical Society, Providence, RI, USA, 1973.
\bibitem{FurstKatz} H. Furstenberg and Y. Katznelson, \emph{An ergodic Szemer\'edi theorem for commuting trasformations}, Journal d'Analyse Math\'ematique \textbf{34} (1978), 275--291.
\bibitem{TimSzem} W.T. Gowers, \emph{A new proof of Szemer\'edi's theorem}, Geometric and Functional Analysis \textbf{11} (2001), 465--588.
\bibitem{U4paper} W.T. Gowers and L. Mili\'cevi\'c, \emph{A quantitative inverse theorem for the $U^4$ norm over finite fields}, arXiv preprint (2017), \verb+arXiv:1712.00241+.
\bibitem{genPaper} W.T. Gowers and L. Mili\'cevi\'c, \emph{An inverse theorem for Freiman multi-homomorphisms}, arXiv preprint (2020), \verb+arXiv:2002.11667+.
\bibitem{TimWolf} W.T. Gowers and J. Wolf, \emph{Linear forms and higher-degree uniformity functions on $\mathbb{F}^n_p$}, Geometric and Functional Analysis \textbf{21} (2011), 36--69.
\bibitem{greenRuzsaFreiman} B. Green and I.Z. Ruzsa, \emph{Freiman's theorem in an arbitrary abelian group}, Journal of the London Mathematical Society \textbf{75} (2007), 163--175.
\bibitem{StrongU3} B. Green and T. Tao, \emph{An inverse theorem for the Gowers $U^3(G)$-norm}, Proceedings of the Edinburgh Mathematical Society \textbf{51} (2008), 73--153.
\bibitem{GreenTaoPolys} B. Green and T. Tao. \emph{The distribution of polynomials over finite fields, with applications to the Gowers norms}, Contributions to Discrete Mathematics \textbf{4} (2009), no. 2, 1--36.
\bibitem{GreenTaoPrimes} B. Green and T. Tao, \emph{Linear equations in primes}, Annals of Mathematics \textbf{171} (2010), no. 3, 1753--1850.
\bibitem{StrongUkZ} B. Green, T. Tao and T. Ziegler, \emph{An inverse theorem for the Gowers $U^{s+1}[N]$-norm}, Annals of Mathematics \textbf{176} (2012), 1231--1372.
\bibitem{GMV1} Y. Gutman, F. Manners and P. Varj\'u, \emph{The structure theory of Nilspaces I}, Journal d'Analyse Math\'ematique \textbf{140} (2020), 299--369.
\bibitem{GMV2} Y. Gutman, F. Manners and P. Varj\'u, \emph{The structure theory of Nilspaces II: Representation as nilmanifolds}, Transactions of the American Mathematical Society \textbf{371} (2019), 4951--4992.
\bibitem{GMV3} Y. Gutman, F. Manners and P. Varj\'u, \emph{The structure theory of Nilspaces III: Inverse limit representations and topological dynamics}, Advances in Mathematics \textbf{365} (2020), 107059.
\bibitem{Janzer1} O. Janzer, \emph{Low analytic rank implies low partition rank for tensors}, arXiv preprint (2018) \verb+arXiv:1809.10931+.
\bibitem{Janzer2} O. Janzer, \emph{Polynomial bound for the partition rank vs the analytic rank of tensors}, Discrete Analysis, paper no. 7 (2020), 1--18.
\bibitem{Lovett} S. Lovett, \emph{The analytic rank of tensors and its applications}, Discrete Analysis, paper no. 7 (2019), 1--10.
\bibitem{Manners} F. Manners, \emph{Quantitative bounds in the inverse theorem for the Gowers $U^{s+1}$-norms over cyclic groups}, arXiv preprint (2018), \verb+arXiv:1811.00718+.
\bibitem{LukaRank} L. Mili\'cevi\'c, \emph{Polynomial bound for partition rank in terms of analytic rank}, Geometric and Functional Analysis \textbf{29} (2019), 1503--1530.
\bibitem{Naslund}  E. Naslund, \emph{The partition rank of a tensor and $k$-right corners in $\mathbb{F}_q^n$}, arXiv preprint (2017), \verb+arXiv:1701.04475+.
\bibitem{Rudin} W. Rudin, \emph{Trigonometric series with gaps}, Journal of Mathematics and Mechanics \textbf{9} (1960), 203--227.
\bibitem{Ruzsa} I.Z. Ruzsa, \emph{Generalized arithmetical progressions and sumsets}, Acta Mathematica Hungarica \textbf{65} (1994), 379--388.
\bibitem{Sanders} T. Sanders, \emph{On the Bogolyubov-Ruzsa lemma}, Analysis \& PDE \textbf{5} (2012), no. 3, 627--655.
\bibitem{Szeg} B. Szegedy, \emph{On higher order Fourier analysis}, arXiv preprint (2012), \verb+arXiv:1203.2260+.
\bibitem{SzemAP} E. Szemer\'edi, \emph{On sets of integers containing no $k$ elements in arithmetic progression}, Acta Arithmetica \textbf{27} (1975), 299--345.
\bibitem{TaoVuBook} T. Tao and V. Vu, \textbf{Additive combinatorics}, Cambridge Studies in Advanced Mathematics \textbf{105}, Cambridge University Press, Cambridge, UK, 2006.
\bibitem{TaoZiegler} T. Tao and T. Ziegler, \emph{The inverse conjecture for the Gowers norm over finite fields in low characteristic}, Annals of Combinatorics \textbf{16} (2012), 121--188.
\end{document}